\documentclass{article}
\usepackage{amsmath,amsthm,amssymb}
\usepackage[usenames,dvipsnames]{xcolor}
\usepackage{enumerate}
\usepackage{graphicx}
\usepackage{cite}
\usepackage{comment}
\usepackage{oands}
\usepackage{tikz}
\usepackage{changepage}
\usepackage{bbm}
\usepackage{mathtools}
\usepackage[margin=1in]{geometry}
\usepackage[pagewise,mathlines]{lineno}
\usepackage{appendix}
\usepackage{multicol}
\usepackage{microtype}

\theoremstyle{plain}
\newtheorem{thm}{Theorem}[section]

\newtheorem{lem}[thm]{Lemma}
\newtheorem{prop}[thm]{Proposition}

\def\@rst #1 #2other{#1}
\newcommand\MR[1]{\relax\ifhmode\unskip\spacefactor3000 \space\fi
  \MRhref{\expandafter\@rst #1 other}{#1}}
\newcommand{\MRhref}[2]{\href{http://www.ams.org/mathscinet-getitem?mr=#1}{MR#2}}

\usepackage[pdftitle={Scaling limit of the uniform infinite half-plane quadrangulation in the Gromov-Hausdorff-Prokhorov-uniform topology},
  pdfauthor={Ewain Gwynne and Jason Miller},
colorlinks=true,linkcolor=NavyBlue,urlcolor=RoyalBlue,citecolor=PineGreen,bookmarks=true,bookmarksopen=true,bookmarksopenlevel=2,unicode=true,linktocpage]{hyperref}

\theoremstyle{definition}
\newtheorem{defn}[thm]{Definition}
\newtheorem{remark}[thm]{Remark}

\numberwithin{equation}{section}

\newcommand{\dsb}{\begin{adjustwidth}{2.5em}{0pt}
\begin{footnotesize}}
\newcommand{\dse}{\end{footnotesize}
\end{adjustwidth}}

\newcommand{\ssb}{\begin{adjustwidth}{2.5em}{0pt}}
\newcommand{\sse}{\end{adjustwidth}}

\newcommand{\aryb}{\begin{eqnarray*}}
\newcommand{\arye}{\end{eqnarray*}}
\def\alb#1\ale{\begin{align*}#1\end{align*}}
\newcommand{\eqb}{\begin{equation}}
\newcommand{\eqe}{\end{equation}}
\newcommand{\eqbn}{\begin{equation*}}
\newcommand{\eqen}{\end{equation*}}

\newcommand{\BB}{\mathbbm}
\newcommand{\ol}{\overline}
\newcommand{\ul}{\underline}
\newcommand{\op}{\operatorname}

\newcommand{\frk}{\mathfrak}
\newcommand{\eqD}{\overset{d}{=}}
\newcommand{\ep}{\epsilon}
\newcommand{\rta}{\rightarrow}

\newcommand{\wt}{\widetilde}
\newcommand{\wh}{\widehat} 
\newcommand{\mcl}{\mathcal}

\newcommand{\bdy}{\partial}

\let\originalleft\left
\let\originalright\right
\renewcommand{\left}{\mathopen{}\mathclose\bgroup\originalleft}
\renewcommand{\right}{\aftergroup\egroup\originalright}

\title{Scaling limit of the uniform infinite half-plane quadrangulation in the Gromov-Hausdorff-Prokhorov-uniform topology}
\date{  }
\author{
\begin{tabular}{c} Ewain Gwynne\\[-5pt]\small MIT \end{tabular}
\begin{tabular}{c} Jason Miller\\[-5pt]\small Cambridge \end{tabular}
}

\begin{document}

\maketitle

\begin{abstract}
We prove that the uniform infinite half-plane quadrangulation (UIHPQ), with either general or simple boundary, equipped with its graph distance, its natural area measure, and the curve which traces its boundary, converges in the scaling limit to the Brownian half-plane.  The topology of convergence is given by the so-called Gromov-Hausdorff-Prokhorov-uniform (GHPU) metric on curve-decorated metric measure spaces, which is a generalization of the Gromov-Hausdorff metric whereby two such spaces $(X_1, d_1 , \mu_1,\eta_1)$ and $(X_2, d_2 , \mu_2,\eta_2)$ are close if they can be isometrically embedded into a common metric space in such a  way that the spaces $X_1$ and $X_2$ are close in the Hausdorff distance, the measures $\mu_1$ and $\mu_2$ are close in the Prokhorov distance, and  the curves $\eta_1$ and $\eta_2$ are close in the uniform distance.
\end{abstract}


\tableofcontents

\section{Introduction}
\label{sec-intro}

\subsection{Overview}
\label{sec-overview}

There has been substantial interest in recent years in the scaling limits of random planar maps.  Various uniform random planar maps (equipped with the graph distance) have been shown to converge in the Gromov-Hausdorff topology to \emph{Brownian surfaces}, the best known of which is the Brownian map, which is the scaling limit of uniform random quadrangulations of the sphere~\cite{legall-uniqueness,miermont-brownian-map}. These results have been generalized in~\cite{ab-simple,bjm-uniform, abraham-bipartite} to other ensembles of random maps on the sphere and in~\cite{curien-legall-plane} (resp.\ \cite{bet-mier-disk}) to give the convergence of the uniform infinite plane quadrangulation (resp.\ uniformly random quadrangulations with boundary) toward the Brownian plane (resp.\ disk).
 
A planar map is naturally endowed with a measure $\mu$ (e.g., the one which assigns mass to each vertex equal to its degree).  Many interesting random planar maps $M$ are also equipped with a curve $\eta$.  Examples include:
\begin{enumerate}
\item The path which visits the boundary $\bdy M$ in cyclic order of a planar map $M$ with boundary.
\item A simple random walk or self-avoiding walk (SAW) on $M$.
\item The Peano curve associated with a distinguished spanning tree of $M$.
\item The exploration path associated with a percolation configuration on $M$.  
\end{enumerate}
Hence it is natural to consider scaling limits of random planar maps in a topology which describes not only their metric structure but also a distinguished measure and curve. 

This article has two main aims.  First, we will introduce such a topology, which arises from the \emph{Gromov-Hausdorff-Prokhorov-uniform (GHPU) metric} on $4$-tuples $(X,d,\mu,\eta)$ consisting of a metric space $(X,d)$, a measure $\mu$ on $X$, and a curve $\eta$ in $X$. Two such $4$-tuples $(X_1, d_1 , \mu_1,\eta_1)$ and $(X_2, d_2 , \mu_2,\eta_2)$ are close in the GHPU metric if they can be isometrically embedded into a common metric space $(W,D)$ in such a way that $X_1$ and $X_2$ are close in the $D$-Hausdorff distance, $\mu_1$ and $\mu_2$ are close in the $D$-Prokhorov distance, and $\eta_1$ and $\eta_2$ are close in the $D$-uniform distance. 
We will consider a version of the GHPU metric for compact spaces as well as a local version for locally compact spaces. 
See Section~\ref{sec-ghpu-def} for a precise definition. 

The definition of the GHPU metric is inspired by other metrics on types of metric spaces such as the Gromov-Hausdorff metric~\cite{bbi-metric-geometry,gromov-metric-book}, the Gromov-Prokhorov metric~\cite{gpw-metric-measure}, and the Gromov-Hausdorff-Prokhorov metric~\cite{adh-ghp,miermont-tess}. 
 
Second, we will prove scaling limit results for the uniform infinite half-plane quadrangulation (UIHPQ) in the local GHPU topology. The UIHPQ is the Benjamini-Schramm local limit~\cite{benjamini-schramm-topology} of uniform random quadrangulations with boundary as the total number of edges and then the perimeter tends to $\infty$~\cite{curien-miermont-uihpq,caraceni-curien-uihpq}, where the map is viewed from a root which is chosen uniformly at random from the boundary. There are two variants of the UIHPQ.  The first is the UIHPQ with general boundary (which we will refer to as the UIHPQ), which may have boundary vertices with multiplicity greater than 1 in the external face; and the UIHPQ with simple boundary (UIHPQ$_{\op{S}}$), where we require that the boundary is simple (i.e., it is a path with no self-intersections). In this paper, we will prove that both the UIHPQ and the UIHPQ$_{\op{S}}$ (equipped with the measure which assigns mass to each vertex equal to its degree and the curve which traces the boundary) converge in the scaling limit in the local GHPU topology to the Brownian half-plane, which we define in Section~\ref{sec-bhp} below (see also~\cite[Section~5.3]{caraceni-curien-uihpq} for a different definition, which we expect is equivalent). Along the way, we will also improve the Gromov-Hausdorff scaling limit result for finite uniform quadrangulations with boundary toward the Brownian disk in~\cite{bet-mier-disk} to a scaling limit result in the GHPU topology.

One particular reason to be interested in random quadrangulations with simple boundary (such as the UIHPQ$_{\op{S}}$) is that one can glue two such surfaces along their boundary to obtain a uniform random quadrangulation decorated by a SAW.  See~\cite[Section~8.2]{bet-disk-tight}  (which builds on~\cite{bbg-recursive-approach,bg-simple-quad}) for the case of finite quadrangulations with simple boundary and~\cite[Part~III]{caraceni-thesis},~\cite{caraceni-curien-saw} for the case of the UIHPQ$_{\op{S}}$.

In~\cite{gwynne-miller-saw}, we will build upon the present work to prove, among other things, that the random planar map obtained by gluing a pair of independent UIHPQ$_{\op{S}}$'s together along the boundary rays lying to the right of their respective root edges (i.e., the uniform infinite SAW-decorated half-plane) converges in the scaling limit in the GHPU topology, with the SAW playing the role of the distinguished curve, to a pair of independent Brownian half-planes glued together in the same way.  We will also prove analogous scaling limit results for two independent UIHPQ$_{\op{S}}$'s glued along their entire boundary and for a single UIHPQ$_{\op{S}}$ with its positive and negative boundary rays glued together.  The proofs of these results use the scaling limit statement for the UIHPQ$_{\op{S}}$ proven in the present paper.  
See also~\cite{gwynne-miller-perc,gwynne-miller-simple-quad} for additional GHPU scaling limit results.
 
\begin{remark}
\label{rem::bmr-uihpq}
In an independent (and essentially simultaneous) work~\cite{bmr-uihpq}, Baur, Miermont, and Ray proved several scaling limit results for uniform quadrangulations with general boundary which include the statement that the UIHPQ with general boundary converges in the scaling limit to the Brownian half-plane in the Gromov-Hausdorff topology~\cite[Theorem~3.6]{bmr-uihpq}. The work \cite{bmr-uihpq} also includes a number of more general scaling limit statements for uniform random quadrangulations with boundary under different scaling regimes that we do not treat here.  In the present paper we will deduce the scaling limit of the UIHPQ to the Brownian half-plane in a stronger topology than in \cite{bmr-uihpq} and also treat the case of the UIHPQ$_{\op{S}}$.  Our proof is somewhat simpler than that in \cite{bmr-uihpq} since our coupling statement is less general. 
\end{remark}

\begin{figure}[ht!!]
\begin{center}
\includegraphics[page=3,scale=0.7]{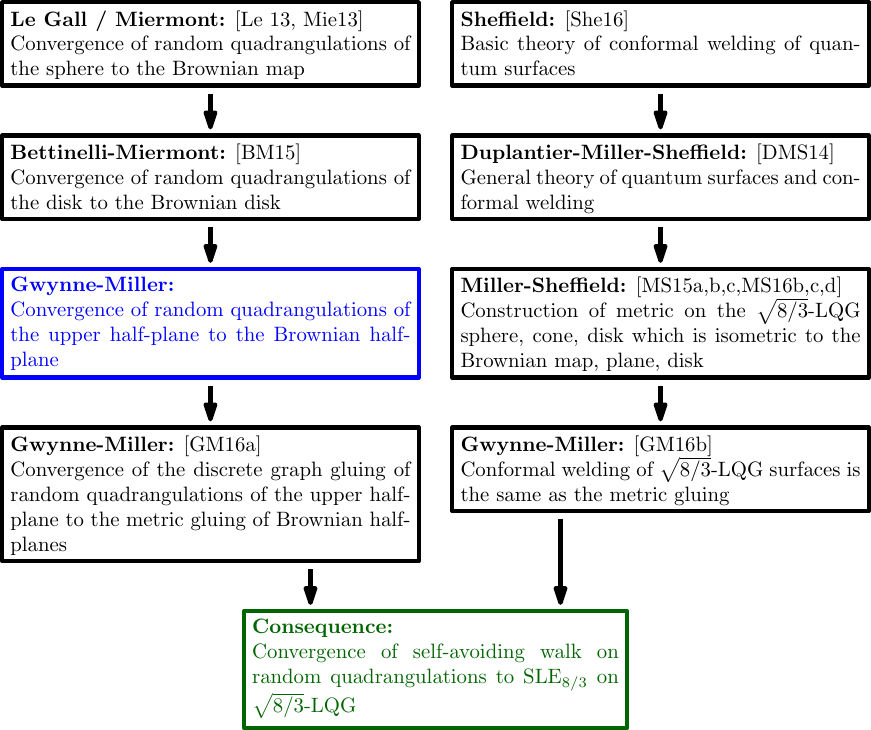}	
\end{center}
\vspace{-0.02\textheight}
\caption{\label{fig-chart} A chart of the different components which serve as input into the proof that self-avoiding walk on random quadrangulations converges to SLE$_{8/3}$ on $\sqrt{8/3}$-LQG.  The present article corresponds to the blue box and implies that a random quadrangulation of the upper half-plane converges in the GHPU topology to the Brownian half-plane.  (See also \cite{bmr-uihpq} for another proof that the UIHPQ converges to the Brownian half-plane, and a more general treatment of scaling limits of quadrangulations with boundary.)}
\end{figure}
  
We will now explain how the aforementioned results about scaling limits of glued UIHPQ$_{\op{S}}$'s allow us to identify the scaling limit of the SAW on a random quadrangulation with SLE$_{8/3}$ on a $\sqrt{8/3}$-Liouville quantum gravity (LQG) surface.  Recently, it has been proven by Miller and Sheffield~\cite{tbm-characterization,sphere-constructions,lqg-tbm1,lqg-tbm2,lqg-tbm3}, building on \cite{qle}, that Brownian surfaces are equivalent to $\sqrt{8/3}$-LQG surfaces. Heuristically speaking, a $\gamma$-LQG surface for $\gamma \in (0,2)$ is the random Riemann surface parameterized by a domain $D\subset \BB C$ whose Riemannian metric tensor is $e^{\gamma h}\,dx\otimes\,dy$, where $dx\otimes\,dy$ is the Euclidean metric tensor on $D$ and $h$ is some variant of the Gaussian free field (GFF) on $D$~\cite{shef-kpz,shef-gff,ss-contour,ig1}. This definition does not make rigorous sense since the GFF is a generalized function, not a function, so does not take values at points.

Miller and Sheffield showed that in the special case when $\gamma = \sqrt{8/3}$, one can make rigorous sense of a $\sqrt{8/3}$-LQG surface as a metric space. Certain particular types of $\sqrt{8/3}$-LQG surfaces introduced in~\cite{wedges}, namely the quantum sphere, quantum disk, and weight-$4/3$ quantum cone, respectively, are isometric to the Brownian map, Brownian disk, and Brownian plane, respectively~\cite[Corollary~1.5]{lqg-tbm2}.  In this paper we will extend this identification by proving that the Brownian half-plane is isometric to the weight-$2$ quantum wedge.

The results of~\cite{gwynne-miller-gluing} together with the identification between the Brownian half-plane and the weight-$2$ wedge proven in the present paper imply that the gluing of two Brownian half-planes along their positive boundary has the same law as a weight-$4$ quantum wedge (a particular type of $\sqrt{8/3}$-LQG surface) decorated by an independent chordal SLE$_{8/3}$ curve~\cite{schramm0}, which is the gluing interface.  Hence the scaling limit result of~\cite{gwynne-miller-saw} discussed above yields the convergence of the SAW on a random quadrangulation to SLE$_{8/3}$ on a $\sqrt{8/3}$-LQG surface.  See Figure~\ref{fig-chart} for a diagram of how the different works fit together to establish this result.

LQG surfaces arise as the scaling limits of random planar maps for all values of $\gamma \in (0,2)$, not just $\gamma = \sqrt{8/3}$.  Values of $\gamma$ other than $\sqrt{8/3}$ correspond to maps sampled with probability proportional to the partition function of some statistical mechanics model, rather than sampled uniformly. For general values of $\gamma$, certain random planar map models decorated by a space-filling curve, which is the Peano curve of a certain spanning tree, have been shown to converge to SLE-decorated LQG in the so-called \emph{peanosphere topology}. This means that the joint law of the contour functions (or some variant thereof) of the spanning tree and its dual, appropriately re-scaled, converges to the law of the correlated Brownian motion which encodes a $\gamma$-LQG cone or sphere decorated by a space-filling SLE$_{16/\gamma^2}$ curve in~\cite{wedges,sphere-constructions}. See~\cite{shef-burger,kmsw-bipolar,gkmw-burger,gms-burger-cone,gms-burger-local,gms-burger-finite,ghs-bipolar} for results of this type.

Neither peanosphere convergence nor GHPU convergence implies the other. However, we expect that the curve-decorated planar maps which converge to SLE$_{16/\gamma^2}$-decorated $\gamma$-LQG in the peanosphere topology also converge in the GHPU topology (this uses the $\gamma$-LQG metric space, which has so far only been constructed for $\gamma=\sqrt{8/3}$), and in fact converge in both topologies jointly. 
In the case of site percolation on a uniform triangulation (which corresponds to $\gamma = \sqrt{8/3}$), this joint GHPU/peanosphere convergence will be proven in the forthcoming work~\cite{ghs-metric-peano}, building on~\cite{gwynne-miller-perc} which shows GHPU convergence of a random planar map decorated by a single percolation interface.
However, it remains open for other models.

\bigskip

\noindent{\bf Acknowledgements}
We thank two anonymous referees for helpful comments on an earlier version of this paper.
E.G.\ was supported by the U.S. Department of Defense via an NDSEG fellowship.  E.G.\ also thanks the hospitality of the Statistical Laboratory at the University of Cambridge, where this work was started.  J.M.\ thanks Institut Henri Poincar\'e for support as a holder of the Poincar\'e chair, during which this work was completed.

\subsection{Preliminary definitions}
\label{sec-notation}

Before stating our main results, we set some standard notation and definitions which will be used throughout this paper. 

\subsubsection{Basic notation}

We write $\BB N$ for the set of positive integers and $\BB N_0 = \BB N\cup \{0\}$. 
\vspace{6pt}

\noindent
For $a < b \in \BB R$, we define the discrete intervals $[a,b]_{\BB Z} := [a, b]\cap \BB Z$ and $(a,b)_{\BB Z} := (a,b)\cap \BB Z$.
\vspace{6pt}

\noindent
If $a$ and $b$ are two quantities, we write $a\preceq b$ (resp.\ $a \succeq b$) if there is a constant $C$ (independent of the parameters of interest) such that $a \leq C b$ (resp.\ $a \geq C b$). We write $a \asymp b$ if $a\preceq b$ and $a \succeq b$.
\vspace{6pt}

\noindent
If $f$ is a function, we write $a = o_b(f(b))$ if $a/f(b) \rta 0$ as $b\rta\infty$ or as $b\rta 0$, depending on context. We write $a = O_b(f(b))$ if there is a constant $C>0$, independent of the parameters of interest, such that $a \leq C f(b)$. 
\vspace{6pt} 

\subsubsection{Graphs}

\noindent
For a planar map $G$, we write $\mcl V(G)$, $\mcl E(G)$, and $\mcl F(G)$, respectively, for the set of vertices, edges, and faces of~$G$.
\vspace{6pt}

\noindent
By a \emph{path} in $G$, we mean a function $ \lambda : I \rta \mcl E(G)$ for some (possibly infinite) discrete interval $I\subset \BB Z$, with the property that the edges $\{\lambda(i)\}_{i\in I}$ can be oriented in such a way that the terminal endpoint of $\lambda(i)$ coincides with the initial endpoint of $\lambda(i+1)$ for each $i\in I$ other than the right endpoint of $I$. We define the \emph{length} of~$\lambda$, denoted $|\lambda|$, to be the integer $\# I$.   
We say that $\lambda$ is \emph{simple} if the vertices hit by $\lambda$ are all distinct.
\vspace{6pt} 

\noindent
For sets $A_1,A_2$ consisting of vertices and/or edges of $G$, we write $\op{dist}\left(A_1 , A_2 ; G\right)$ for the graph distance from~$A_1$ to~$A_2$ in~$G$, i.e.\ the minimum of the lengths of paths in $G$ whose initial edge either has an endpoint which is a vertex in $A_1$ or shares an endpoint with an edge in $A_1$; and whose final edge satisfies the same condition with $A_2$ in place of $A_1$.
\vspace{6pt}

\noindent
For $r>0$, we define the graph metric ball $B_r\left( A_1 ; G\right)$ to be the subgraph of $G$ consisting of all vertices of $G$ whose graph distance from $A_1$ is at most $r$ and all edges of $G$ whose endpoints both lie at graph distance at most $r$ from $A_1$.  
If $A_1 = \{x\}$ is a single vertex or edge, we write $B_r\left( \{x\} ; G\right) =  B_r\left( x ; G\right)$.
\vspace{6pt}

\subsubsection{Metric spaces}

\noindent
If $(X,d)$ is a metric space, $A\subset X$, and $r>0$, we write $B_r(A;d)$ for the set of $x\in X$ with $d (x,A) \leq r$. We emphasize that $B_r(A;d)$ is closed (this will be convenient when we work with the local GHPU topology). 
If $A = \{y\}$ is a singleton, we write $B_r(\{y\};d) = B_r(y;d)$.   
\vspace{6pt}

\noindent
For a curve $\gamma : [a,b] \rta X$, the \emph{$d $-length} of $\gamma$ is defined by 
\eqbn
\op{len}\left( \gamma ; d  \right) := \sup_P \sum_{i=1}^{\# P} d (\gamma(t_i) , \gamma(t_{i-1})) 
\eqen
where the supremum is over all partitions $P : a= t_0 < \dots < t_{\# P} = b$ of $[a,b]$. Note that the $d$-length of a curve may be infinite. 
\vspace{6pt}
 
\noindent
We say that $(X,d)$ is a \emph{length space} if for each $x,y\in X$ and each $\ep > 0$, there exists a curve of $d$-length at most $d(x,y) + \ep$ from $x$ to $y$. 
\vspace{6pt}

\subsection{The Gromov-Hausdorff-Prokhorov-uniform metric}
\label{sec-ghpu-def}

In this paper (and in~\cite{gwynne-miller-saw}) we will consider scaling limits of metric measure spaces endowed with a distinguished continuous curve. A natural choice of topology for this convergence is the one induced by the \emph{Gromov-Hausdorff-Prokhorov-uniform (GHPU) metric}, which we introduce in this subsection and study further in Section~\ref{sec-ghpu-metric}. This topology generalizes the Gromov-Hausdorff topology~\cite{gromov-metric-book,bbi-metric-geometry}, the Gromov-Prokhorov topology~\cite{gpw-metric-measure}, and the Gromov-Hausdorff-Prokhorov topology~\cite{miermont-tess,adh-ghp}. 
 
Let $(X,d)$ be a metric space. The metric $d$ gives rise to the $d$-Hausdorff metric $\BB d_d^{\op{H}}$ on compact subsets of $X$ and the $d$-Prokhorov metric $\BB d_d^{\op{P}}$ on finite measures on $X$ in the standard way. 

The definition of the $d$-uniform metric on curves in~$X$ requires some discussion since we want to allow curves defined on an arbitrary interval in $\BB R$. 
Let $C_0(\BB R , X)$ be the set of continuous curves $\eta : \BB R\rta X$ such that for each $\ep>0$, there exists $T>0$ such that $d(\eta(t) , \eta(T)) \leq \ep$ and $d(\eta(-t) , \eta(-T)) \leq \ep$ whenever $t\geq T$. 
If $\eta: [a,b] \rta X$ is a curve defined on a compact interval, we identify $\eta$ with the element of $C_0(\BB R ,X)$ which agrees with $\eta$ on $[a,b]$ and satisfies $\eta(t) = a$ for $t \leq a$ and $\eta(t) = b$ for $t\geq b$. 
We equip $C_0(\BB R,X)$ with the \emph{$d$-uniform metric}, defined by 
\eqb \label{eqn-uniform-def}
\BB d_d^{\op{U}}(\eta_1,\eta_2) = \sup_{t\in \BB R} d(\eta_1(t) ,\eta_2(t)) ,\quad \forall \eta_1,\eta_2\in C_0(\BB R , X) . 
\eqe  

\begin{remark}[Graphs as connected metric spaces] \label{remark-ghpu-graph}
In this paper we will often be interested in a graph~$G$ equipped with its graph distance~$d_G$.  In order to study continuous curves in~$G$, we need to linearly interpolate~$G$. We do this by identifying each edge of~$G$ with a copy of the unit interval $[0,1]$.  We extend the graph metric on~$G$ by requiring that this identification is an isometry. 

If $\lambda$ is a path in $G$, mapping some discrete interval $[a,b]_{\BB Z}$ to $\mcl E(G)$, we extend $\lambda$ from $[a,b]_{\BB Z}$ to $[a-1,b] $ by linear interpolation, so that for $i\in [a,b]_{\BB Z}$, $\lambda$ traces each edge $\lambda(i)$ at unit speed during the time interval $[i-1,i]$.  In particular, the Gromov-Hausdorff-Prokhorov-uniform metric and its local variant, to be defined below, make sense for graphs equipped with a measure and a curve.
\end{remark}

\subsubsection{Compact case}
\label{sec-ghpu-def-compact}

Let $\BB M^{\op{GHPU}}$ be the set of $4$-tuples $\frk X  = (X , d , \mu , \eta)$ where $(X,d)$ is a compact metric space, $d$ is a metric on $X$, $\mu$ is a finite Borel measure on $X$, and $\eta  \in C_0( \BB R , X)$. We remark that an element of $\BB M^{\op{GHPU}}$ has a natural marked point, namely $\eta(0)$. 
 
Suppose that we are given elements $\frk X_1 = (X_1 , d_1, \mu_1 , \eta_1) $ and $\frk X_2 =  (X_2, d_2,\mu_2,\eta_2) $ of $ \BB M^{\op{GHPU}}$. 
For a compact metric space $(W, D)$ and isometric embeddings $\iota_1 : X_1\rta W$ and $\iota_2 : X_2\rta W$, we define their \emph{GHPU distortion} by
\begin{align} \label{eqn-ghpu-var}
\op{Dis}_{\frk X_1,\frk X_2}^{\op{GHPU}}\left(W,D , \iota_1, \iota_2 \right)   
:=  \BB d^{\op{H}}_D \left(\iota_1(X_1) , \iota_2(X_2) \right) +   
\BB d^{\op{P}}_D \left(( (\iota_1)_*\mu_1 ,(\iota_2)_*\mu_2) \right) + 
\BB d_D^{\op{U}}\left( \iota_1 \circ \eta_1 , \iota_2 \circ\eta_2 \right) .
\end{align}
We define the \emph{Gromov-Hausdorff-Prokhorov-uniform (GHPU) distance} by
\begin{align} \label{eqn-ghpu-def}
 \BB d^{\op{GHPU}}\left( \frk X_1 , \frk X_2 \right) 
 = \inf_{(W, D) , \iota_1,\iota_2}  \op{Dis}_{\frk X_1,\frk X_2}^{\op{GHPU}}\left(W,D , \iota_1, \iota_2 \right)      ,
\end{align}
where the infimum is over all compact metric spaces $(W,D)$ and isometric embeddings $\iota_1 : X_1 \rta W$ and $\iota_2 : X_2\rta W$.

It will be proven in Lemma~\ref{prop-ghpu-triangle} below that $\BB d^{\op{GHPU}}$ defines a pseudometric on $\BB M^{\op{GHPU}}$. It is not quite a metric since two elements $(X_1 , d_1, \mu_1 , \eta_1) , (X_2, d_2,\mu_2,\eta_2)  \in  \BB M^{\op{GHPU}}$ lie at GHPU distance zero if there is a measure-preserving isometry from $X_1$ to $X_2$ which takes $\eta_1$ to $\eta_2$. 
Let $\ol{\BB M}^{\op{GHPU}}$ be the set of equivalence classes of elements of $\BB M^{\op{GHPU}}$ under the equivalence relation whereby $(X_1 , d_1, \mu_1 , \eta_1) \sim (X_2, d_2,\mu_2,\eta_2)$ if and only if there exists such an isometry $f: (X_1,d_1) \rta (X_2,d_2)$ with $f_*\mu_1 = \mu_2$ and $f\circ\eta_1 = \eta_2$. 

The following statement will be proven in Section~\ref{sec-ghpu-properties}. 

\begin{prop} \label{prop-ghpu-metric}
The function $\BB d^{\op{GHPU}}$ is a complete separable pseudometric on $\BB M^{\op{GHPU}}$ and the quotient metric space $\BB M^{\op{GHPU}}/\{ \BB d^{\op{GHPU}} = 0\}$ is $\ol{\BB M}^{\op{GHPU}}$.
\end{prop}

Restricting $\BB d^{\op{GHPU}}$ to elements of $\BB M^{\op{GHPU}}$ for which the measure $\mu$ is identically equal to zero and/or the curve $\eta$ is constant gives a natural metric on the space of compact metric spaces which are not equipped with a measure and/or a curve. In particular, the Gromov-Hausdorff and Gromov-Hausdorff-Prokhorov metrics are special cases of the GHPU metric and we also obtain a metric on curve-decorated compact metric spaces (which should be called the Gromov-uniform metric). 

A particularly useful fact about the GHPU metric, which will be proven in Section~\ref{sec-ghpu-properties} and used in the proof of Proposition~\ref{prop-ghpu-metric}, is that GHPU convergence is equivalent to Hausdorff, Prokhorov, and uniform convergence within a fixed compact metric space, in a sense which we will now make precise. 

\begin{defn}[HPU convergence] \label{def-hpu}
Let $(W,D)$ be a metric space. Let $\frk X = (X,d,\mu,\eta)$ and  $\frk X^n = (X^n , d^n , \mu^n , \eta^n)$ for $n\in\BB N$ be elements of $\BB M^{\op{GHPU}}$ such that $X , X^n\subset W$, $D|_X = d$, and $D|_{X^n} = d^n$. 
We say that $\frk X^n\rta \frk X$ in the \emph{$D$-Hausdorff-Prokhorov-uniform (HPU) sense} if $X^n\rta X$ in the $D$-Hausdorff metric, $\mu^n\rta\mu$ in the $D$-Prokhorov metric, and $\eta^n\rta \eta$ in the $D$-uniform metric. 
\end{defn}
 
\begin{prop} \label{prop-ghpu-embed}
Suppose $\frk X^n = (X^n , d^n , \mu^n , \eta^n)$ for $n\in\BB N$ and $ \frk X = (X, d , \mu ,\eta)$ are elements of $\BB M^{\op{GHPU}}$. Then $\frk X^n \rta \frk X$ in the GHPU metric if and only if there exists a compact metric space $(W,D)$ and isometric embeddings $  X\rta W$ and $   X^n \rta W$ for $n\in\BB N$ such that if we identify $X$ and $X^n$ with their images under these embeddings, then $\frk X^n\rta \frk X$ in the $D$-HPU sense.  
\end{prop}

Analogs of Proposition~\ref{prop-ghpu-embed} for the Gromov-Hausdorff and Gromov-Prokhorov metrics are proven in~\cite[Lemmas~5.8 and~A.1]{gpw-metric-measure}, respectively. The proof of Proposition~\ref{prop-ghpu-embed} below will be similar to the proofs of these lemmas.

\subsubsection{Non-compact case}
\label{sec-ghpu-def-local}

In this paper we will also have occasion to consider non-compact curve-decorated metric measure spaces (such as the Brownian half-plane).  In this subsection we consider a variant of the GHPU metric in this setting.  We restrict attention to length spaces to avoid technical complications with convergence of metric balls. However, we expect that it is possible to relax this restriction with some modifications to the definition. See~\cite[Section~8.1]{bbi-metric-geometry} for a definition of the local Gromov-Hausdorff topology which does not require the length space condition. 

Let $\BB M_\infty^{\op{GHPU}}$ be the set of $4$-tuples $\frk X = (X,d,\mu,\eta)$ where $(X,d)$ is a locally compact length space, $\mu$ is a measure on $X$ which assigns finite mass to each finite-radius metric ball in~$X$, and $\eta : \BB R\rta X$ is a curve in~$X$. Note that $\BB M^{\op{GHPU}}$ is not contained in $\BB M_\infty^{\op{GHPU}}$ since elements of the former are not required to be length spaces.

Let $\ol{\BB M}_\infty^{\op{GHPU}}$ be the set of equivalence classes of elements of $\BB M_\infty^{\op{GHPU}}$ under the equivalence relation whereby $(X_1 , d_1, \mu_1 , \eta_1) \sim (X_2, d_2,\mu_2,\eta_2)$ if and only if there is an isometry $f : X_1\rta X_2$ such that $f_*\mu_1=\mu_2$ and $f\circ \eta_1=\eta_2$. 

We will define a local version of the GHPU metric on $\ol{\BB M}_\infty^{\op{GHPU}}$ by truncating $\frk X$ at the metric ball $B_r( \eta(0) ;d)$, then integrating the GHPU metric over all metric balls. The truncation is done in the following manner. 
 
\begin{defn} \label{def-ghpu-truncate}
Let $\frk X = (X , d, \mu,\eta)$ be an element of $\BB M_\infty^{\op{GHPU}}$. For $r > 0$, let 
\eqb
\ul\tau_r^\eta := (-r) \vee \sup\left\{t < 0 \,:\, d(\eta(0) ,\eta(t)) = r\right\} \quad \op{and}\quad \ol\tau_r^\eta := r\wedge \inf\left\{t > 0 \,:\, d(\eta(0),\eta(t)) = r\right\} .
\eqe
The \emph{$r$-truncation} of $\eta$ is the curve $\frk B_r\eta \in C_0(\BB R ; X)$ defined by
\eqbn
\frk B_r\eta(t) = 
\begin{cases}
\eta(\ul\tau_r^{\eta}) ,\quad &t\leq \ul\tau_r^\eta  \\
\eta(t) ,\quad &t\in (\ul\tau^\eta , \ol\tau_r^\eta) \\
\eta( \ol\tau_r^{\eta}) ,\quad &t\geq  \ol\tau_r^\eta  .
\end{cases}
\eqen
The \emph{$r$-truncation} of $X$ is the curve-decorated metric measure space
\eqbn
\frk B_r \frk X  = \left( B_r(\eta(0) ;d) , d|_{B_r(\eta(0) ;d)} , \mu|_{B_r(\eta(0) ;d)} , \frk B_r\eta \right) .
\eqen
\end{defn}

If $\frk X = (X,d,\mu,\eta)\in \BB M_\infty^{\op{GHPU}}$, then the Hopf-Rinow theorem~\cite[Theorem~2.5.28]{bbi-metric-geometry} implies that every closed metric ball in $X$ is compact. 
Hence for $\frk X\in \BB M_\infty^{\op{GHPU}}$, we have $\frk B_r\frk X \in  \BB M^{\op{GHPU}}$ for each $r>0$. Furthermore, for $R > r> 0$ we have $\frk B_r\frk B_R \frk X = \frk B_r \frk X$. 

The \emph{local GHPU metric} on $\BB M_\infty^{\op{GHPU}}$ is the function on $\BB M_\infty^{\op{GHPU}} \times  \BB M_\infty^{\op{GHPU}} \rta [0,\infty)$ defined by
\eqb \label{eqn-ghpu-local-def}
\BB d^{\op{GHPU}}_\infty \left( \frk X_1,\frk X_2\right) = \int_0^\infty e^{-r} \left(1 \wedge \BB d^{\op{GHPU}}\left(\frk B_r\frk X_1, \frk B_r\frk X_2 \right)  \right) \, dr
\eqe 
where $\BB d^{\op{GHPU}}$ is as in~\eqref{eqn-ghpu-def}. 
 
We let $\ol{\BB M}_\infty^{\op{GHPU}}$ be the set of equivalence classes of elements of $\BB M_\infty^{\op{GHPU}}$ under the equivalence relation whereby $(X_1 , d_1, \mu_1 , \eta_1) \sim (X_2, d_2,\mu_2,\eta_2)$ if and only if there is an isometry $f : (X_1,d_1) \rta (X_2,d_2)$ such that $f_*\mu_1=\mu_2$ and $f\circ \eta_1=\eta_2$.   
The following is the analog of Proposition~\ref{prop-ghpu-metric} for the local GHPU metric, and will be proven in Section~\ref{sec-ghpu-local}.  

\begin{prop} \label{prop-local-ghpu-metric}
The function $\BB d_\infty^{\op{GHPU}}$ is a complete separable pseudometric on $\BB M_\infty^{\op{GHPU}}$ and the quotient metric space $\BB M_\infty^{\op{GHPU}}/\{ \BB d_\infty^{\op{GHPU}} = 0\}$ is $\ol{\BB M}_\infty^{\op{GHPU}}$.
\end{prop}

We next want to state an analog of Proposition~\ref{prop-ghpu-embed} for the local GHPU metric. 
To this end, we make the following definition.
  
\begin{defn}[Local HPU convergence] \label{def-hpu-local}
Let $(W ,D)$ be a metric space. Let $\frk X^n = (X^n , d^n , \mu^n , \eta^n)$ for $n\in\BB N$ and $\frk X = (X,d,\mu,\eta)$ be elements of $\BB M^{\op{GHPU}}_\infty$ such that $X$ and each $X^n$ is a subset of $W$ satisfying $D|_X = d$ and $D |_{X^n} = d^n$. We say that $\frk X^n\rta \frk X$ in the \emph{$D$-local Hausdorff-Prokhorov-uniform (HPU) sense} if the following is true. 
\begin{itemize}
\item For each $r>0$ we have $B_r(\eta^n(0) ; d^n) \rta B_r(\eta(0) ; d)$ in the $D$-Hausdorff metric.
\item For each $r > 0$ such that $\mu\left(\bdy B_r(\eta(0) ;d)\right) = 0$, we have $\mu^n|_{B_r(\eta^n(0) ;d^n)} \rta \mu|_{B_r(\eta(0) ;d)}$ in the $D$-Prokhorov metric. 
\item For each $a,b\in\BB R$ with $a<b$, we have $\eta^n|_{[a,b]}\rta \eta|_{[a,b]}$ in the $D$-uniform metric.  
\end{itemize} 
\end{defn}
 
The following is our analog of Proposition~\ref{prop-ghpu-embed} for the local GHPU metric.  
 
\begin{prop} \label{prop-ghpu-embed-local}
Let $ \frk X^n = (X^n , d^n , \mu^n , \eta^n)$ for $ n\in\BB N $ and $\frk X = (X,d,\mu,\eta)$ be elements of $\BB M_\infty^{\op{GHPU}}$. Then $\frk X^n\rta \frk X$ in the local GHPU topology if and only if there exists a boundedly compact (i.e., closed bounded sets are compact) metric space $(W , D )$ and isometric embeddings $X^n \rta W$ for $n\in\BB N$ and $X\rta W$ such that the following is true. If we identify $X^n$ and $X$ with their embeddings into $W$, then $\frk X^n \rta \frk X$ in the $D$-local HPU sense.
\end{prop}

\subsection{Basic definitions for quadrangulations}
\label{sec-quad-def}

The main results of this paper are scaling limit statements for quadrangulations with boundary in the GHPU topology. In this subsection we introduce notation to describe these objects. 
\vspace{6pt}

\noindent
A \emph{quadrangulation with boundary} is a (finite or infinite) planar map $Q$ with a distinguished face $f_\infty$, called the \emph{exterior face}, such that every face of $Q$ other than $f_\infty$ has degree 4. The \emph{boundary} of $Q$, denoted by $\bdy Q$, is the smallest subgraph of $Q$ which contains every edge of $Q$ incident to $f_\infty$. The \emph{perimeter} $\op{Perim}(Q)$ of $Q$ is defined to be the degree of the exterior face, with edges counted with multiplicity (i.e., the number of half-edges on the boundary). 
\vspace{6pt}

\noindent
A \emph{boundary path} of $Q$ is a path $\lambda$ from $[1,\op{Perim}(Q)]_{\BB Z}$ (if $\bdy Q$ is finite) or $\BB Z$ (if $\bdy Q$ is infinite) to $\mcl E(\bdy Q)$ which traces the edge of $\bdy Q$  (counted with multiplicity) in cyclic order around the exterior face. Choosing a boundary path is equivalent to choosing an oriented root (half-)edge on the boundary. This root edge is $\lambda(\op{Perim}(Q))$, oriented toward $\lambda(1)$ in the finite case; or $\lambda(0)$, oriented toward $\lambda(1)$, in the infinite case (here we note that a quadrangulation cannot have any self-loops).  
\vspace{6pt}

\noindent
We say that $\bdy Q$ is \emph{simple} if some (equivalently every) boundary path for $Q$ hits each vertex exactly once.
\vspace{6pt}

\noindent
For $n,l\in \BB N_0$, we write $\mcl Q(n,l)$ for the set of quadrangulations with general boundary having $n$ interior faces and $2l$ boundary edges (counted with multiplicity).  

We write $\mcl Q^\bullet(n,l)$ for the set of triples $(Q, \BB e_0 , \BB v_*)$ where $Q\in \mcl Q(n,l)$, $\BB e_0  $ is a distinguished oriented half-edge of $\bdy Q$ (meaning that if $\BB e_0$ has multiplicity 2, we need to specify a ``side" of $\BB e_0$), and $\BB v_* \in \mcl V(Q)$ is a distinguished vertex. 
\vspace{6pt}

\noindent
The \emph{uniform infinite half-plane quadrangulation (UIHPQ)} is the infinite boundary-rooted quadrangulation $(Q_\infty , \BB e_\infty)$ which is the limit in law with respect to the Benjamini-Schramm topology~\cite{benjamini-schramm-topology} of a uniform sample from $\mcl Q(n,l)$ (rooted at a uniformly random boundary edge) if we first send $n \rta \infty$ and then $l\rta\infty$~\cite{curien-miermont-uihpq,caraceni-curien-uihpq}. 
\vspace{6pt}

\noindent
The \emph{uniform infinite planar quadrangulation with simple boundary} (UIHPQ$_{\op{S}}$) is the infinite boundary-rooted quadrangulation $(Q_{\op{S}} , \BB e_{\op{S}})$ with simple boundary which is the limit in law with respect to the Benjamini-Schramm topology~\cite{benjamini-schramm-topology} of a uniformly random quadrangulation with simple boundary (rooted at a uniformly random boundary edge) with $n$ interior faces and $2l$ boundary edges if we first send $n \rta \infty$ and then $l\rta\infty$. 
The UIHPQ$_{\op{S}}$ can be recovered from the UIHPQ by ``pruning" the dangling quadrangulations which are disconnected from $\infty$ by a single vertex~\cite{curien-miermont-uihpq,caraceni-curien-uihpq}; see Section~\ref{sec-pruning} for a review of this procedure.

\subsection{The Brownian half-plane}
\label{sec-bhp}

The limiting object in the main scaling limit results of this paper is the Brownian half-plane, which we define in this section. 
The construction given here is of the ``unconstrained" type (corresponding to the version of the Schaeffer bijection in which labels are not required to be positive). There is also a constrained construction of the Brownian half-plane in~\cite[Section~5.3]{caraceni-curien-uihpq}. We expect (but do not prove) that this construction is equivalent to the one we give here. Our construction is a continuum analog of the Schaeffer-type construction of the UIHPQ$_{\op{S}}$ found in~\cite{curien-miermont-uihpq} (c.f.~\cite{caraceni-curien-uihpq}), which we review in Section~\ref{sec-uihpq}. 

Let $X_\infty : \BB R\rta[0,\infty)$ be the process such that $\{X_\infty(t) \}_{t\geq 0}$ is a standard linear Brownian motion and $\{X_\infty(-t)\}_{t\geq 0}$ is an independent Brownian motion conditioned to stay positive (i.e., a 3-dimensional Bessel process). 
For $r\in \BB R$, let
\eqbn
T_\infty(r) := \inf\left\{ t \in \BB R \,:\, X_\infty(t) = -r\right\} ,
\eqen
so that $r\mapsto T_\infty(r)$ is non-decreasing and for each $r\in \BB R$,   
\eqb \label{eqn-bhp-reroot}
\{ X_\infty(T_\infty(r) + t)  + r \}_{t\in\BB R} \eqD \{ X_\infty(t) \}_{t\in\BB R}  .
\eqe 
Also let $T_\infty^{-1} : \BB R\rta \BB R$ be the right-continuous inverse of $T$, so that
\eqb \label{eqn-bhp-inverse}
T_\infty^{-1}(t) = -\inf\left\{ X_\infty(s) : s \leq t   \right\} .
\eqe 
 
For $s,t\in \BB R$, let
\eqb \label{eqn-X_infty-dist}
d_{X_\infty}(s,t) := X_\infty(s) + X_\infty(t)  - 2\inf_{u \in [s\wedge t , s\vee t]} X_\infty(u) .
\eqe 
Then $d_{X_\infty}$ defines a pseudometric on $\BB R$ and the quotient metric space $\BB R / \{d_{X_\infty} = 0\}$ is a forest of continuum random trees, indexed by the excursions of $X_\infty$ away from its running infimum. 
 
Conditioned on $X_\infty$, let $ Z_\infty^0$ be the centered Gaussian process with
\eqb \label{eqn-Z^0-def-infty}
\op{Cov}(Z_\infty^0(s) , Z_\infty^0(t) ) = \inf_{u\in [s\wedge t , s\vee t]} \left( X_\infty(u) - \inf_{v \leq u } X_\infty(v) \right) , \quad s,t\in \BB R.
\eqe 
By the Kolmogorov continuity criterion, $Z_\infty^0$ a.s.\ admits a continuous modification which is locally $\alpha$-H\"older continuous for each $\alpha <1/4$. For this modification we have $Z_\infty^0(s) = Z_\infty^0(t)$ whenever $d_{X_\infty}(s,t) = 0$, so $Z_\infty^0$ defines a function on the continuum random forest $\BB R / \{d_{X_\infty} = 0\}$. 

Let $\frk b_\infty : \BB R\rta \BB R$ be $\sqrt 3$ times a two-sided standard linear Brownian motion. For $t\in \BB R$, define
\eqbn
Z_\infty(t) := Z_\infty^0(t)  +  \frk b_\infty(T_\infty^{-1}(t))  ,
\eqen
with $T_\infty^{-1}$ as in~\eqref{eqn-bhp-inverse}. 

For $s,t\in \BB R$, define
\eqb \label{eqn-d_Z-infty}
d_{Z_\infty} \left(s,t \right) = Z_\infty(s) + Z_\infty(t) - 2\inf_{u\in [s\wedge t , s\vee t]} Z_\infty(u) .
\eqe 
Also define the pseudometric
\eqb \label{eqn-dist0-def-infty}
d_\infty^0(s,t) = \inf \sum_{i=1}^k d_{Z_\infty}(s_i , t_i)
\eqe 
where the infimum is over all $k\in\BB N$ and all $2k+2$-tuples $( t_0, s_1 , t_1 , \dots , s_{k } , t_{k } , s_{k+1}) \in \BB R^{2k+2}$ with $t_0 = s$, $s_{k+1} = t$, and $d_{X_\infty}(t_{i-1} , s_i) = 0$ for each $i\in [1,k+1]_{\BB Z}$. In other words, $d_\infty^0$ is the largest pseudometric on $\BB R$ which is at most $d_{Z_\infty}$ and is zero whenever $d_{X_\infty}$ is $0$. 

The \emph{Brownian half-plane} is the quotient space $H_\infty = \BB R/\{d_\infty^0 = 0\}$ equipped with the quotient metric, which we call $d_\infty$.  
We write $\BB p_\infty : \BB R \rta H_\infty$ for the quotient map. 
It follows from~\cite[Theorem 2]{bet-disk-tight} (which says that the Brownian disk has the topology of the closed disk) and Proposition~\ref{prop-bhp-coupling} below that $H_\infty$ has the topology of the closed half-plane. 

The \emph{boundary} of $H_\infty$ is the set $\bdy H_\infty = \BB p \left(\{T_\infty(r) \,:\, r \in \BB R\} \right)$. 
It follows from the proof of Proposition~\ref{prop-bhp-coupling} below and the analogous property of the Brownian disk~\cite[Proposition 21]{bet-disk-tight} that $\bdy H_\infty$ is in fact the boundary of $H_\infty$ in the topological sense (i.e., the set of points which do not have a neighborhood which is homeomorphic to the disk). 

The \emph{area measure} of $H_\infty$ is the pushforward of Lebesgue measure on $\BB R$ under $\BB p_\infty$, and is denoted by $\mu_\infty$. 
The \emph{boundary measure} of $H_\infty$ is the pushforward of Lebesgue measure on $\BB R$ under the map $r\mapsto \BB p_\infty(T_\infty(r))$.
The \emph{boundary path} of $H_\infty$ is the path $\eta_\infty : \BB R\rta \BB R$ defined by $\eta_\infty(r) = \BB p_\infty(T_\infty(r))$. 
Note that $\eta_\infty$ travels one unit of boundary length in one unit of time. 

Although it is not needed for the statement or proof of our main results, we record for reference the $\sqrt{8/3}$-LQG description of the Brownian half-plane, which will be proven in this paper. 

\begin{prop}
\label{prop-bhp-wedge}
Let $(\BB H ,h, 0,\infty)$ be a $\sqrt{8/3}$-quantum gravity wedge (i.e., a quantum wedge of weight equal to $2$) with LQG parameter $\gamma = \sqrt{8/3}$~\cite{wedges}. Let $\mu_h$ and $\nu_h$, respectively, be the $\sqrt{8/3}$-LQG area and boundary length measures induced by $h$~\cite{shef-kpz}. Also let $\frk d_h$ be the $\sqrt{8/3}$-LQG metric induced by $h$~\cite{lqg-tbm1,lqg-tbm2,lqg-tbm3}.  Let $\eta_h : \BB R\rta \BB R$ be the curve which parameterizes $\BB R$ according to $\sqrt{8/3}$-LQG length and satisfies $\eta_h(0) = 0$.  Then $(\BB H , \frk d_h , \mu_h , \eta_h)$ and $(H_\infty ,d_\infty, \mu_\infty,\eta_\infty)$ (as defined just above) agree as elements of $\BB M_\infty^{\op{GHPU}}$, i.e.\ there exists an isometry $f : (\BB H ,\frk d_h) \rta (H_\infty ,d_\infty)$ satisfying $f_*\mu_h = \mu_\infty$ and $f\circ \eta_h = \eta_\infty$.
\end{prop}

Proposition~\ref{prop-bhp-wedge} follows from the results of this paper together with the same argument to prove the analogous $\sqrt{8/3}$-LQG description of the Brownian plane in~\cite[Corollary~1.5]{lqg-tbm2}.  Indeed, Proposition~\ref{prop-bhp-coupling} below tells us that the Brownian half-plane is the local limit of Brownian disks when we zoom in near a boundary point sampled uniformly from the boundary measure.  The $\sqrt{8/3}$-quantum wedge is the local limit of quantum disks when we zoom in near a boundary point~\cite{wedges}.  We already know from~\cite[Corollary~1.5]{lqg-tbm2} that Brownian disks coincide with quantum disks in the sense of Proposition~\ref{prop-bhp-wedge}, so the proposition follows.

We remark that the weight-$2$ quantum wedge mentioned in Proposition~\ref{prop-bhp-wedge} comes with some additional structure, namely an embedding into $\BB H$ with the two marked points respectively sent to $0$ and $\infty$.  It follows from the main result of \cite{lqg-tbm3} that this embedding is a.s.\ determined by the quantum wedge, viewed as a random variable taking values in $\BB M_\infty^{\op{GHPU}}$.  This in particular implies that the Brownian half-plane a.s.\ determines its embedding into $\sqrt{8/3}$-LQG.

\subsection{Theorem statements: scaling limit of the UIHPQ and UIHPQ$_{\op{S}}$}
\label{sec-results}

In this subsection we state scaling limit results for the UIHPQ with general and simple boundary in the local GHPU topology. 

Let $(H_\infty , d_\infty)$ be an instance of the Brownian half-plane, as in Section~\ref{sec-bhp}.  
Let $\mu_\infty$ and $\eta_\infty : \BB R\rta \bdy H_\infty$, respectively, be its area measure and natural boundary path. 
Let
\eqb \label{eqn-bhp-ghpu}
\frk H_\infty := (H_\infty , d_\infty , \mu_\infty , \eta_\infty ) ,
\eqe
so that $\frk H_\infty$ is an element of $\BB M_\infty^{\op{GHPU}}$, defined as in Section~\ref{sec-ghpu-def-local}. 
 
Let $(Q_\infty , \BB e_\infty)$ be a UIHPQ (with general boundary). 
We view $Q_\infty$ as a connected metric space by replacing each edge with an isometric copy of the unit interval, as in Remark~\ref{remark-ghpu-graph}.
For $n\in\BB N$, let $d_\infty^n$ be the graph distance on $Q_\infty $, re-scaled by $(9/8)^{1/4} n^{-1/4}$.  
Let $\mu_\infty^n$ be the measure on $\mcl V(Q_\infty )$ which assigns a mass to each vertex equal to $(4n)^{-1}$ times its degree (in the scaling limit, this is equivalent to assigning each face mass $n^{-1}$). 
Let $\lambda_\infty : \BB R\rta \bdy Q_\infty $ be the boundary path of $Q_\infty$ started from $\BB e_\infty$ and extended by linear interpolation. 
Let $\eta_\infty^n(t) := \lambda_\infty\left(2^{3/2} n^{1/2} t \right)$ for $t\in \BB R$. 
For $n\in\BB N$, let
\eqb \label{eqn-uihpq-ghpu}
\frk Q_\infty^n := \left(Q_\infty , d_\infty^n , \mu_\infty^n, \eta_\infty^n \right)  
\eqe 
so that $\frk Q_\infty^n$ is an element of $\BB M_\infty^{\op{GHPU}}$.

\begin{thm} \label{thm-uihpq-ghpu}
In the setting described just above, we have $\frk Q_\infty^n \rta \frk H_\infty$ in law in the local GHPU topology,
i.e.\ the UIHPQ converges in law in the scaling limit to the Brownian half-plane in the local GHPU topology.
\end{thm}  

Next we state an analog of Theorem~\ref{thm-uihpq-ghpu} for the UIHPQ with simple boundary. 
Let $(Q_{\op{S}} , \BB e_{\op{S}})$ be a UIHPQ$_{\op{S}}$. 
We will define for each $n\in\BB N$ an element of $\BB M_\infty^{\op{GHPU}}$ associated with $(Q_{\op{S}} , \BB e_{\op{S}})$ in the same manner as in the case of the UIHPQ, except that the time scaling for the boundary path is different.

As above we view $Q_{\op{S}}$ as a connected metric space in the manner of Remark~\ref{remark-ghpu-graph}.
For $n\in\BB N$, let $d_{\op{S}}^n$ be the graph distance on $Q_{\op{S}}^n$, re-scaled by $(9/8)^{1/4} n^{-1/4}$ and let  
$\mu_{\op{S}}^n$ be the measure on $\mcl V(Q_{\op{S}} )$ which assigns a mass to each vertex equal to $(4n)^{-1}$ times its degree. 
Let $\lambda_{\op{S}} : \BB R\rta \bdy Q_{\op{S}} $ be the boundary path of $Q_{\op{S}}$, started from $\BB e_{\op{S}}$ and extended by linear interpolation. 
Let $\eta_{\op{S}}^n(t) := \lambda_\infty\left(\frac{2^{3/2}}{3} n^{1/2} t \right)$ for $t\in \BB R$ (note that $\frac{2^{3/2}}{3}$ is replaced by $2^{3/2}$ in the UIHPQ case). 
For $n\in\BB N$, let
\eqb \label{eqn-uihpqS-ghpu}
\frk Q_{\op{S}}^n := \left(Q_{\op{S}} , d_{\op{S}}^n , \mu_{\op{S}}^n, \eta_{\op{S}}^n \right) .
\eqe 

\begin{thm} \label{thm-uihpqS-ghpu}
With $\frk H_\infty$ as in~\eqref{eqn-bhp-ghpu}, we have $\frk Q_{\op{S}}^n \rta \frk H_\infty$ in law in the local GHPU topology,
i.e.\ the UIHPQ$_{\op{S}}$ converges in law in the scaling limit to the Brownian half-plane in the local GHPU topology.
\end{thm}  

We will also prove in Section~\ref{sec-disk-ghpu} below a scaling limit result for finite quadrangulations with boundary toward the Brownian disk in the GHPU topology. We do not state this result here, however, as its proof is a straightforward extension of the proof of the analogous convergence statement in the Gromov-Hausdorff topology from~\cite{bet-mier-disk}.

\subsection{Outline}
\label{sec-outline}

The remainder of this article is structured as follows. 
In Section~\ref{sec-ghpu-metric}, we prove the statements about the Gromov-Hausdorff-Prokhorov-uniform metric described in Section~\ref{sec-ghpu-def} plus some additional properties, including compactness criteria and a measure-theoretic condition for GHPU convergence.

In Section~\ref{sec-quad-prelim}, we review some facts about random planar maps in preparation for our proofs of Theorems~\ref{thm-uihpq-ghpu} and~\ref{thm-uihpqS-ghpu}, including the Schaeffer-type constructions of uniform quadrangulations with simple boundary and the UIHPQ, the relationship between the UIHPQ and the UIHPQ$_{\op{S}}$ via the pruning procedure, and the definition of the Brownian disk.

In Section~\ref{sec-uihpq-conv}, we prove Theorems~\ref{thm-uihpq-ghpu} and~\ref{thm-uihpqS-ghpu}. The proof of Theorem~\ref{thm-uihpq-ghpu} is similar in spirit to the proof of the scaling limit result for the Brownian plane in~\cite{curien-legall-plane}. It proceeds by showing that the Brownian half-plane (resp.\ UIHPQ) can be closely approximated by a Brownian disk (resp.\ uniform quadrangulation with boundary) and applying a strengthened version of the scaling limit result for uniform quadrangulations with boundary from~\cite{bet-mier-disk}. Theorem~\ref{thm-uihpqS-ghpu} is deduced from Theorem~\ref{thm-uihpq-ghpu} and the pruning procedure.


\section{Properties of the Gromov-Hausdorff-Prokhorov-uniform metric}
\label{sec-ghpu-metric}

In this section we will establish the important properties of the GHPU and local GHPU metrics, defined in Section~\ref{sec-ghpu-def}, and in particular prove Propositions~\ref{prop-ghpu-metric},~\ref{prop-ghpu-embed},~\ref{prop-local-ghpu-metric}, and~\ref{prop-ghpu-embed-local}.  We start in Section~\ref{sec-isometry-limit} by proving some elementary topological lemmas which give conditions for a sequence of 1-Lipschitz maps or isometries defined on a sequence of metric spaces to have a subsequential limit. These lemmas will be used several times in this section and in~\cite{gwynne-miller-saw}.  In Section~\ref{sec-ghpu-properties}, we establish the basic properties of the GHPU metric on compact curve-decorated metric measure spaces. In Section~\ref{sec-ghpu-local}, we establish the basic properties of the local GHPU metric on non-compact curve-decorated metric measure spaces.  In Section~\ref{sec-ghpu-condition}, we introduce a generalization of Gromov-Prokhorov convergence and use it to give a criterion for GHPU convergence which will be used for the proof of our scaling limit results in \cite{gwynne-miller-saw}.

\subsection{Subsequential limits of isometries}
\label{sec-isometry-limit}

In this subsection we record two elementary topological lemmas which will be useful for our study of the GHPU metric. 
  
\begin{lem} \label{prop-lip-limit}
Let $(W, D  , w )$ be a separable pointed metric space and let $(\wh W , \wh D)$ be any metric space. Let $\{X^n\}_{n\in\BB N}$ and $X$ be closed subsets of $W$ and for $n\in\BB N$, let $f^n : X^n \rta \wh W$ be a 1-Lipschitz map. Suppose that the following are true.
\begin{enumerate}
\item\label{item-lip-limit-conv}  For each $r>0$, $B_r(w ; D)\cap X^n \rta B_r(w;D)\cap X$ in the $D$-Hausdorff metric. 
\item\label{item-lip-limit-image}   For each $r> 0$, there exists a compact set $\wh W_r \subset \wh W$ such that $f^n\left(B_r(w ; D) \cap X^n \right) \subset \wh W_r$ for each $n\in\BB N$.  
\end{enumerate}  
Then there is a sequence $\mcl N$ of positive integers tending to $\infty$ and a 1-Lipschitz map $f : X \rta \wh W$ such that $f^n \rta f$ as $\mcl N \ni n \rta\infty$ in the following sense. For any $x \in X$, any subsequence $\mcl N'$ of $\mcl N$, and any sequence of points $x^n \in X^n$ for $n\in\mcl N'$ with $x^n \rta x $, we have $f^n (x^n) \rta f(x)$ as $\mcl N'\ni n \rta\infty$.
Moreover, if each $f^n$ is an isometry onto its image, then $f$ is also an isometry onto its image. 
\end{lem} 
\begin{proof}
Let $\{x_j\}_{j\in\BB N}$ be a countable dense subset of $X$ (which exists since $X$ is separable). By assumption~\ref{item-lip-limit-conv}, $B_r(w ; D) \cap X^n \rta B_r(w;D)\cap X$ in the $D$-Hausdorff topology for each $r>0$, so for each $j\in\BB N$ we can choose points $x_j^n \in X^n$ such that $D(x_j^n , x_j) \rta 0$. 
By condition~\ref{item-lip-limit-image}, each of the sequences $\{x_j^n\}_{n\in\BB N}$ is contained in a compact subset of $\wh W$. 

By a diagonalization argument we can find a sequence $\mcl N$ of positive integers tending to $\infty$ and points $\{\wh x_j\}_{j\in\BB N}$ in $\wh W$ such that $f^{n}( x^n_j) \rta \wh x_j$ for each $j\in\BB N$ as $\mcl N\ni n \rta\infty$. Let $f(x_j) := \wh x_j$ for $j\in\BB N$. Then for $j_1,j_2\in\BB N$,  
\eqbn
\wh D(f(x_{j_1}) , f(x_{j_2}) ) = \lim_{\mcl N \ni n \rta\infty} \wh D(f^n(x_{j_1}^n) , f^n(x_{j_2}^n) ) \leq \lim_{\mcl N \ni n \rta\infty} D(x_{j_1}^n , x_{j_2}^n) = D(x_{j_1} , x_{j_2}) ,
\eqen
with equality throughout if in fact each $f^n$ is an isometry onto its image.
Since $\{x_j\}_{j\in\BB N}$ is dense in $X$, the map $f$ extends by continuity to a 1-Lipschitz map $X\rta \wh W$, which preserves distances in the case when each $f^n$ is an isometry. 

It remains to check that $f^n \rta f$ in the sense described in the lemma. Suppose that we are given a subsequence $\mcl N'$ of $\mcl N$ and a sequence of points $x^n \in X^n$ with $x^n \rta x \in X$. Fix $\ep > 0$ and choose $j\in \BB N$ with $x_j \in B_{\ep }(x ; D)$. Then for large enough $n \in \mcl N'$,
\eqbn
\wh D(f^{n}(x_j^{n}) , f^{n}(x^{n} ) ) 
\leq D(x_j^{n} , x^{n}) 
\leq D(x_j^{n} , x_j ) +  D(x_j , x )  + D(x ,x^{n}) 
\leq \ep + o_{n}(1) .
\eqen
Since $f^{n}(x_j^{n}) \rta f(x_j)$, $f^{n}(x^{n})$ lies within $\wh D$-distance $2 \ep$ of $f(x_j)$ for large enough $n \in \mcl N'$. On the other hand, $\wh D(f(x_j) , f(x)) \leq \ep$ (since $f$ is 1-Lipschitz). Therefore $f^{n}(x^{n}) \rta f(x)$ along the subsequence $\mcl N'$. 
\end{proof} 
 
In the case when the $f^n$'s are isometries onto their images and we assume convergence in the HPU sense, we obtain existence of an isometry $f$ satisfying additional properties.
 
\begin{lem} \label{prop-isometry-limit}
Let $\frk X^n = (X^n ,d^n, \mu^n , \eta^n)$ for $n\in\BB N$ and $\frk X = (X ,d , \mu , \eta)$ be elements of $\BB M^{\op{GHPU}}$ (resp.\ $\BB M_\infty^{\op{GHPU}}$) such that $X^n$ and $X$ are subsets of a common boundedly compact (i.e., closed bounded subsets are compact) metric space $(W,D)$ satisfying $d^n = D|_{X^n}$ and $d = D|_X$. 
Let $(\wh W, \wh D)$ be another boundedly compact metric space and let $\wh{\frk X} = (\wh X , \wh d , \wh\mu , \wh\eta)$ be an element of $\BB M^{\op{GHPU}}$ (resp.\ $\BB M_\infty^{\op{GHPU}}$) such that $\wh X\subset \wh D$ and $\wh D|_{\wh X} = \wh d$. 

Suppose that we are given distance-preserving maps $f^n : X^n \rta \wh W$ for each $n\in\BB N$ such that the following are true (using the terminology as in Definitions~\ref{def-hpu} and~\ref{def-hpu-local}).
\begin{enumerate}
\item $\frk X^n \rta \frk X$ in the $D$-(local) HPU sense. \label{item-iso-limit-conv} 
\item $ (f^n(X^n), \wh D|_{f^n(X^n)} ,  f^n_*\mu^n , f^n\circ\eta^n) \rta \wh{\frk X}$ in the $\wh D$-(local) HPU sense.  \label{item-iso-limit-image}  
\end{enumerate}  
There is a sequence $\mcl N$ of positive integers tending to $\infty$ and an isometry $f : X \rta \wh X$ with $f\circ\eta = \wh\eta$ and $f_*\mu = \wh\mu$ such that $f^n \rta f$ as $\mcl N \ni n \rta\infty$ in the following sense. For any $x \in X$, any subsequence $\mcl N'$ of $\mcl N$, and any sequence of points $x^n \in X^n$ for $n\in\mcl N'$ with $x^n \rta x $, we have $f^n (x^n) \rta f(x)$ as $\mcl N'\ni n \rta\infty$. 
\end{lem}
\begin{proof}
By Lemma~\ref{prop-lip-limit} and since $(\wh W , \wh d)$ is boundedly compact, there is a sequence $\mcl N$ of positive integers tending to $\infty$ and a distance-preserving map $f : X\rta \wh X$ such $f^{n} \rta f$ in the sense described in the statement of Lemma~\ref{prop-lip-limit}. It remains to check that $f$ is surjective and satisfies $f\circ\eta =\wh\eta$ and $f_*\mu = \wh\mu$. 
 
We start with surjectivity. Suppose given $\wh x \in \wh X$. Since $f^{n}( X^{n}) \rta \wh X$ in the $\wh D$-local Hausdorff metric, we can find a sequence $\{x^{n}\}_{n\in\mcl N}$ of points in $X^{n}$ such that $f^{n}(x^{n}) \rta \wh x$.  There is an $r>0$ such that $\wh D(f^{n}(x) , f^{n}(\eta^{n}(0))) \leq r$ for each $n\in\mcl N$, so since each $f^{n}$ is an isometry we also have $d^{n}(x , \eta^{n}(0) ) \leq r$ for each $n\in\mcl N $. Since $W$ is boundedly compact, by possibly passing to a subsequence, we can arrange so that $x^{n} \rta x \in X$. Then the convergence of $f^{n}$ to $f$ implies that $f(x) = \wh x$. 

Next we check that $f\circ \eta = \wh\eta$. For each $t\geq 0$ we have $\eta^n(t) \rta \eta(t)$ and $f^n(\eta^n(t)) \rta \wh\eta(t)$. Therefore our convergence condition for $f^{n }$ toward $f$ implies that $f(\eta(t)) = \wh\eta(t)$. 
 
Finally, we show that $f_* \mu = \wh\mu$. 
We do this in the case of $\BB M_\infty^{\op{GHPU}}$; the case of $\BB M^{\op{GHPU}}$ is similar (but in fact simpler). 
Let $r>0$ be a radius for which both
\eqb \label{eqn-iso-limit-bdy-mass}
\mu\left( \bdy B_r(\eta(0) ; d) \right) = 0 \quad \op{and}\quad \wh\mu\left( \bdy B_r(\wh\eta(0) ; d) \right) =0.
\eqe  
For $n\in\mcl N$ let $x^{n }$ be sampled uniformly from $\mu^{n }|_{B_r(\eta^{n }(0) ; d^{n })}$. By assumption~\ref{item-iso-limit-conv}, $\mu^{n }|_{B_r(\eta^{n }(0) ; d^{n })} \rta \mu|_{B_r(\eta(0) ; d)}$ in the $D$-Prokhorov metric. Therefore, $x^{n} \rta x$ in law, where $x$ is sampled uniformly from $\mu|_{B_r(\eta(0) ; d)}$. By the Skorokhod representation theorem, we can couple $\{x^{n}\}_{n\in\mcl N}$ with $x$ in such a way that $x^{n } \rta x$ a.s. Then our convergence condition for $f^{n}$ implies that $f^{n} (x^{n}) \rta f(x)$ a.s. Since each $f^{n}$ is an isometry, the random variable $f^{n}(x^{n})$ is sampled uniformly from $f^{n}_*\mu^{n} |_{  B_r( f^{n}(\eta^{n}(0) ) ; \wh D) \cap f^n(X^n)  }$. By~\eqref{eqn-iso-limit-bdy-mass} and assumption~\ref{item-lip-limit-image}, we find that $f^{n}(x^{n})$ converges in law to $\wh x$, where $\wh x$ is sampled uniformly from $\mu|_{B_r(\wh\eta(0) ; \wh d)}$. Hence the law of $f(x)$ is that of a uniform sample from $\mu|_{B_r(\wh\eta(0) ; \wh d)}$. Since $\mu^{n}\left(B_r(\eta^{n}(0) ; d^{n}) \right)$ converges to both $\mu(B_r(\eta(0); d))$ and $\wh\mu(\wh\eta(0) ; \wh d))$, we find that 
\eqbn
f_*\left( \mu|_{B_r(\eta(0) ; d)} \right) = \wh\mu|_{B_r(\wh \eta(0) ; \wh d)}  
\eqen
for all but countably many $r>0$. Therefore $f_* \mu = \wh\mu$. 
\end{proof}

\subsection{Proofs for the GHPU metric}
\label{sec-ghpu-properties}

In this subsection we prove Propositions~\ref{prop-ghpu-metric} and~\ref{prop-ghpu-embed}. Most of the arguments in this subsection are straightforward adaptations of standard proofs for the Gromov-Hausdorff, Gromov-Prokhorov, and Gromov-Hausdorff-Prokhorov metrics; see~\cite{bbi-metric-geometry,gromov-metric-book,gpw-metric-measure,adh-ghp,miermont-tess}, but we give full proofs here for the sake of completeness. 
 
The following lemma tells us that the definition of $\BB d^{\op{GHPU}}$ in~\eqref{eqn-ghpu-def} is unaffected if when taking the infimum we require that $W= X_1 \sqcup X_2$ and $\iota_1$ and $\iota_2$ are the natural inclusions.
 
\begin{lem} \label{prop-ghpu-coprod} 
Let $\frk X_1 = (X_1 , d_1, \mu_1 , \eta_1) $ and $\frk X_2 =  (X_2, d_2,\mu_2,\eta_2) $ be in $ \BB M^{\op{GHPU}}$ and identify $X_1$ and $X_2$ with their inclusions into the disjoint union $X_1\sqcup X_2$. For a metric $d_\sqcup$ on $X_1\sqcup X_2$ with $d_\sqcup|_{X_1} = d_1$ and $d_\sqcup|_{X_2} = d_2$, we define
\eqb \label{eqn-ghpu-coprod-var}
\op{Dis}_{\frk X_1,\frk X_2}^{\op{GHPU},\sqcup}(d_\sqcup) = \BB d^{\op{H}}_{d_\sqcup} \left(X_1 , X_2 \right) +   
\BB d^{\op{P}}_{d_\sqcup} \left( \mu_1,\mu_2  \right) + 
\BB d_{d_\sqcup}^{\op{U}}\left( \eta_1 , \eta_2 \right) .
\eqe 
Then
\eqb \label{eqn-ghpu-coprod-def}
\BB d^{\op{GHPU}}\left( \frk X_1 , \frk X_2 \right)  = \inf \op{Dis}_{\frk X_1,\frk X_2}^{\op{GHPU}}(d_\sqcup) 
\eqe
where the infimum is over all metrics on $X_1\sqcup X_2$ with $d_\sqcup|_{X_1} = d_1$ and $d_\sqcup|_{X_2} = d_2$. 
\end{lem}
\begin{proof}
It is clear that the infimum in~\eqref{eqn-ghpu-coprod-def} is at most the infimum in~\eqref{eqn-ghpu-def}, so we only need to prove the reverse inequality.
Suppose given a compact metric space $(W,D)$ and isometric embeddings $\iota_1 : X_1 \rta W$ and $\iota_2 : X_2\rta W$. Given $\ep > 0$, we define a metric on $X_1\sqcup X_2$ by 
\eqbn
d_\sqcup^\ep(x,y) = 
\begin{cases}
d_1(x,y) ,\quad & x,y \in X_1 \\
d_2(x,y) ,\quad &x,y\in X_2 \\
D(\iota_1(x) , \iota_2(y) ) + \ep ,\quad &x \in X_1 , \, y\in X_2 \\
D(\iota_2(x) , \iota_1(y) ) + \ep ,\quad &x \in X_2 , \, y\in X_1 .
\end{cases}
\eqen
It is easy to see that $d_\sqcup^\ep$ defines a metric on $X_1\sqcup X_2$. Furthermore, since $\iota_1$ and $\iota_2$ are isometric embeddings, it follows that $d_\sqcup^\ep(x,y)$ differs from $D(x,y)$ by as most $\ep$. Hence
\eqbn
 \op{Dis}_{\frk X_1,\frk X_2}^{\op{GHPU},\sqcup}(d_\sqcup)  \leq  \op{Dis}_{\frk X_1,\frk X_2}^{\op{GHPU}}\left(W,D , \iota_1, \iota_2 \right)   + 3\ep  .
\eqen
 Since $\ep>0$ is arbitrary the statement of the lemma follows. 
\end{proof}

We now verify the triangle inequality for $\BB d^{\op{GHPU}}$, which in particular implies that $\BB d^{\op{GHPU}}$ is a pseudometric. 

\begin{lem} \label{prop-ghpu-triangle}
The function $\BB d^{\op{GHPU}}$ satisfies the triangle inequality, i.e.\ for $\frk X_1 , \frk X_2 , \frk X_3    \in \BB M^{\op{GHPU}}$ we have
\eqbn
\BB d^{\op{GHPU}}\left(  \frk X_1,\frk X_3  \right) \leq \BB d^{\op{GHPU}}\left(  \frk X_1,\frk X_2  \right) + \BB d^{\op{GHPU}}\left(  \frk X_2,\frk X_3 \right) .
\eqen
\end{lem}
\begin{proof}
Write $\frk X_i =  (X_i, d_i,\mu_i,\eta_i) $ for $i\in \{1,2,3\}$ and fix $\ep > 0$. By Lemma~\ref{prop-ghpu-coprod} we can find metrics $d_{12}$ and $d_{23}$ on $X_1\sqcup X_2$ and $X_2\sqcup X_3$, respectively, which restrict to the given metrics on each factor such that 
\eqbn
\op{Dis}_{\frk X_1,\frk X_2}^{\op{GHPU},\sqcup}(d_{12}) \leq \BB d^{\op{GHPU}}\left(  \frk X_1,\frk X_2  \right) + \ep \quad \op{and} \quad
\op{Dis}_{\frk X_2,\frk X_3}^{\op{GHPU},\sqcup}(d_{23}) \leq \BB d^{\op{GHPU}}\left(  \frk X_2,\frk X_3 \right) + \ep .
\eqen
We define a metric on $X_1\sqcup X_2\sqcup X_3$ as follows. If $x,y \in X_1\sqcup X_2\sqcup X_3$ and both $x$ and $y$ belong to $X_1\sqcup X_2$ or $X_2\sqcup X_3$ we set $d_{13}(x,y) = d_{12}(x,y)$ or $d_{23}(x,y)$, respectively. 
For $x_1 \in X_1$ and $x_3 \in X_3$, we set
\eqbn
d_{13}(x_1,x_3) = d_{13}(x_3,x_1) = \inf_{x_2\in X_2} \left( d_{12}(x_1 , x_2) + d_{23}(x_2,x_3) \right) .
\eqen
It is easily checked that $d_{13}$ is a metric on $X_1\sqcup X_2\sqcup X_3$, so restricts to a metric on $X_1\sqcup X_3$ which in turn restricts to $d_1$ on $X_1$ and $d_3$ on $X_3$. Furthermore, by the triangle inequalities for the $d_{13}$-Hausdorff, Prokhorov, and uniform metrics, we have
\eqbn
\op{Dis}_{\frk X_1,\frk X_3}^{\op{GHPU},\sqcup}(d_{13}) \leq \op{Dis}_{\frk X_1,\frk X_2}^{\op{GHPU},\sqcup}(d_{12})  + \op{Dis}_{\frk X_2,\frk X_3}^{\op{GHPU},\sqcup}(d_{23}) .
\eqen
The statement of the lemma follows.
\end{proof}

Now we can prove Proposition~\ref{prop-ghpu-embed}, using a similar argument to that used to prove~\cite[Lemma~A.1]{gpw-metric-measure}. 

\begin{proof}[Proof of Proposition~\ref{prop-ghpu-embed}]
By Lemma~\ref{prop-ghpu-coprod}, for each $n\in\BB N$ there exists a metric $d_\sqcup^n$ on $X\sqcup X^n$ such that $\op{Dis}_{\frk X,\frk X^n}^{\op{GHPU},\sqcup}(d_\sqcup^n) \rta 0$.  Let $W:= X\sqcup \bigsqcup_{n=1}^\infty X^n$ and identify $X$ and each $X^n$ with its natural inclusion into~$W$.  We define a metric $D$ on $W$ as follows.  If $x , y \in W$ such that $x,y\in X \sqcup X^n$ for some $n\in\BB N$, we set $D(x,y) = d_\sqcup^n(x,y)$.  If $x \in X^n$ and $y\in X^m$ for some $n,m\in\BB N$, we set
\eqbn
D(x,y) = \inf_{u \in X} \left( d_\sqcup^n(x, u) + d_\sqcup^m(u , y) \right) .
\eqen
As in the proof of Lemma~\ref{prop-ghpu-triangle}, $D$ is a metric on $W$ and by definition this metric restricts to $d^n$ on each $X^n$ and to $d$ on $X$. Furthermore we have $\op{Dis}_{\frk X,\frk X^n}^{\op{GHPU},\sqcup}(d_\sqcup^n) = \op{Dis}_{\frk X , \frk X^n}^{\op{GHPU},\sqcup}(D)  \rta 0$ as $n\rta\infty$, which implies that $X^n \rta X$ in the $D$-Hausdorff metric, $\mu^n \rta \mu$ in the $D$-Prokhorov metric, and $\eta^n\rta \eta$ in the $D$-uniform metric. 

By replacing $W$ with its metric completion, we can assume that $W$ is complete. Since $X$ is totally bounded, for each $\ep > 0$ we can find $N\in\BB N$ and $x_1 , \dots , x_N$ such that $X \subset \bigcup_{i=1}^N B_\ep(x_i ; d)$.
Since $X^n \rta X$ in the $D$-Prokhorov metric and $D|_X = d$, it follows that there exists $n_0\in\BB N$ such that $X^n \subset \bigcup_{i=1}^N  B_\ep(x_i ; D)$ for $n\geq n_0$. Since each $X^n$ for $n\leq n_0$ is totally bounded, we infer that $W$ is totally bounded, hence compact.
\end{proof}

\begin{lem}[Positive definiteness] \label{prop-ghpu-pos-def}
Let $\frk X_1 = (X_1 , d_1, \mu_1 , \eta_1) $ and $\frk X_2 =  (X_2, d_2,\mu_2,\eta_2) $ be in $ \BB M^{\op{GHPU}}$. If $\BB d^{\op{GHPU}}(\frk X_1, \frk X_2) = 0$, then there is an isometry $f: X_1\rta X_2$ with $f_*\mu_1 = \mu_2$ and $f\circ\eta_1 = \eta_2$. 
\end{lem} 
\begin{proof}
By Lemma~\ref{prop-ghpu-coprod}, we can find a sequence of metrics $d_\sqcup^n$ on $X_1\sqcup X_2$ whose GHPU distortion tends to~$0$. By Proposition~\ref{prop-ghpu-embed} there is a compact metric space $(W,D)$ and isometric embeddings $\iota : X_1 \rta W$ and $\iota^n : X_2  \rta W$ for $n\in\BB N$ such that
\eqbn
\left( \iota^n(X_2) , D|_{\iota^n(X_2)} , \iota^n_*\mu_2  ,\iota^n\circ\eta \right) \rta (X_1,d_1,\mu_1,\eta_1)
\eqen
in the $D$-HPU topology. By Lemma~\ref{prop-isometry-limit} (applied with $\frk X^n =\frk X_2$ for each $n\in\BB N$) we can find a subsequence along which the embeddings $\iota^n$ converge to an isometry $g : X_2 \rta \iota(X_1)$ with $g_* \mu_2 = \iota_* \mu_2$ and $g\circ\eta_2 = \iota\circ \eta_1$. The statement of the lemma follows by taking $f = g^{-1} \circ \iota$. 
\end{proof}

For our proof of completeness, we will use the following compactness criterion for $\BB M^{\op{GHPU}}$, which is also of independent interest.

\begin{lem}[Compactness criterion] \label{prop-ghpu-compact-cond}
Let $\mcl K$ be a subset of $ \BB M^{\op{GHPU}}$ and suppose the following conditions are satisfied.
\begin{enumerate}
\item $\mcl K$ is uniformly totally bounded, i.e.\ for each $\ep > 0$, there exists $N\in\BB N$ such that for each $(X,d,\mu,\eta) \in \mcl K$, the set $X$ can be covered by at most $N$ $d$-balls of radius at most $\ep$. \label{item-ghpu-cmpt-h}
\item There is a constant $C>0$ such that for each $(X,d,\mu,\eta) \in \mcl K$, we have $\mu(X)\leq C$. \label{item-ghpu-cmpt-p}
\item $\mcl K$ is equicontinuous, i.e.\ for each $\ep > 0$, there exists $\delta>0$ such that for each $(X,d,\mu,\eta)$ in $\mcl K$ and each $s,t\in \BB R$ with $|s-t| \leq \delta$, we have $d(\eta(s) ,\eta(t)) \leq \ep$; and for each $t\in \BB R$ with $|t| \geq \delta^{-1}$, we have $d(\eta(s) , \eta(\pm \delta^{-1})) \leq \ep$, where $\pm$ is the sign of $t$. \label{item-ghpu-cmpt-u}
\end{enumerate}
Then every sequence in $\mcl K$ has a subsequence which converges with respect to the GHPU topology.
\end{lem} 
\begin{proof}
Let $\frk X^n = (X^n,d^n,\mu^n,\eta^n)$ for $n\in\BB N$ be elements of $\mcl K$. 
By condition~\ref{item-ghpu-cmpt-h} and the Gromov compactness criterion~\cite[Theorem~7.4.15]{bbi-metric-geometry}, we can find a sequence $n_k \rta\infty$ and a compact metric space $(X,d)$ such that $(X^{n_k} ,d^{n_k} ) \rta (X,d)$ in the Gromov-Hausdorff topology. 
By~\cite[Lemma~A.1]{gpw-metric-measure} (or Proposition~\ref{prop-ghpu-embed} applied with $\mu  = 0$ and $\eta$ a constant curve) we can find a compact metric space $(W ,D)$ and isometric embeddings $X^{n_k} \rta W$ for $k\in\BB N$ and $X\rta W$ such that if we identify $X^{n_k}$ and $X$ with their embeddings, we have $X^{n_k} \rta X$ in the $D$-Hausdorff metric. 

By conditions~\ref{item-ghpu-cmpt-p} and~\ref{item-ghpu-cmpt-u}, the Prokhorov theorem, and the Arz\'ela-Ascoli theorem,
after possibly passing to a further subsequence we can find a finite Borel measure $\mu$ on $X$ and a curve $\eta$ in $X$ such that $\mu^{n_k} \rta \mu$ in the $D$-Prokhorov metric and $\eta^{n_k} \rta \eta$ in the $D$-uniform metric as $k \rta \infty$. 
Therefore $\frk X^{n_k} \rta \frk X$ in the GHPU metric.
\end{proof}

\begin{lem}[Completeness] \label{prop-ghpu-complete}
The pseudometric space $(\BB M^{\op{GHPU}} , \BB d^{\op{GHPU}})$ is complete. 
\end{lem}
\begin{proof}
Let $\frk X^n = (X^n , d^n, \mu^n,\eta^n)$ for $n\in\BB N$ be a Cauchy sequence with respect to $\BB d^{\op{GHPU}}$. It is clear that $\mcl K =\{\frk X^n\}_{n\in\BB N}$ satisfies the conditions of Lemma~\ref{prop-ghpu-compact-cond}, so has a convergent subsequence. The Cauchy condition implies that in fact the whole sequence converges. 
\end{proof}

Next we check separability.

\begin{lem}[Separability] \label{prop-ghpu-separable}
The space $\BB M^{\op{GHPU}}$ with the metric $\BB d^{\op{GHPU}}$ is separable.
\end{lem}

The proof of Lemma~\ref{prop-ghpu-separable} is slightly more difficult than one might expect. The reason for this is that we cannot simply approximate elements of $\BB M^{\op{GHPU}}$ by finite metric spaces, since such spaces do not admit non-trivial continuous paths. We get around this as follows. Given $\frk X = (X, d , \mu , \eta) \in  \BB M^{\op{GHPU}}$, we first find a finite subset $A$ of $X$ equipped with a measure $ \mu_0$ such that $(A, d|_{A} , \mu_0)$ closely approximates $(X,d,\mu)$ it in the Hausdorff and Prokhorov metrics. We then isometrically embed $(A,d|_A)$ into (a very high dimensional) Euclidean space  equipped with the $L^\infty$ distance and draw in a piecewise linear path which approximates $\eta$. 
 
\begin{proof}[Proof of Lemma~\ref{prop-ghpu-separable}]
Let $\BB F$ be the set of $(X , d , \mu , \eta) \in \BB M^{\op{GHPU}}$ for which the following is true.
\begin{itemize}
\item $X$ is a subset of $\BB R^N$ for some $N\in\BB N$ and $d$ is the restriction of the $L^\infty$ metric $d^\infty$ on $\BB R^N$ (i.e.\ $d^\infty(z,w) = \max_{i\in [1,N]_{\BB Z}} |z_i - w_i|$).
\item $X$ is the union of finitely many points in $\BB Q^N$ with and finitely many linear segments with endpoints in~$\BB Q^N$. 
\item The measure $\mu$ is supported on $X\cap \BB Q^N$, and $\mu(x) \in \BB Q \cap [0,\infty) $ for each $x\in X\cap \BB Q^N$. 
\item The curve $\eta$ is the concatenation of finitely many (possibly degenerate) linear segments with endpoints in $\BB Q^N$, each traversed at a constant $d^\infty$-speed which belongs to $\BB Q\cap [0,\infty)$. 
\end{itemize} 
It is clear that $\BB F$ is countable. We claim that $\BB F$ is dense in $\BB M^{\op{GHPU}}$. 
It is clear that $\BB F$ is dense in the set $\wh{\BB F} \subset \BB M^{\op{GHPU}}$ which is defined in the same manner as $\BB F$ but with every instance of $\BB Q$ replaced by $\BB R $. Hence it suffices to show that $\wh{\BB F}$ is dense in $\BB M^{\op{GHPU}}$. 
 
Suppose that we are given $\frk X = (X , d,\mu,\eta) \in \BB M^{\op{GHPU}}$ and $\ep > 0$. We will construct $\wh{\frk X}  = (\wh X , \wh d , \wh\mu ,\wh\eta) \in \wh{\BB F}$ which approximates $X$ in the GHPU sense.  
We first define a finite subset $A$ of $X$ as follows.
\begin{itemize}
\item Let $A_1$ be a finite $\ep$-dense subset of $X$.
\item Let $\mu_0$ be a finitely supported measure on $X$ with $\BB d^{\op{P}}_d(\mu, \mu_0) \leq \ep$ (which can be obtained, e.g., by sampling $M$ i.i.d.\ points from $\mu$ for $M$ large and defining the mass of each to be $\mu(X)/M$). Let $A_2$ be the support of $\mu$.
\item Let $K \in \BB N$ be chosen so that $d(\eta(s) , \eta(t)) \leq \ep$ whenever $s,t\in \BB R$ with $|t-s| \leq 1/K$ and $d(\eta(\pm t) , \eta(\pm\delta^{-1})) \leq \ep$ whenever $t \geq K$. For $k \in [-K^2 , K^2 ]_{\BB Z}$ let $y_k := \eta(k/K)$. Let
\eqbn
A_3 := \{y_{-K^2} ,\dots , y_{K^2} \} .
\eqen
\item Let $A := A_1\cup A_2\cup A_3$.
\end{itemize}

It is not hard to see (and is proven, e.g., in~\cite[Theorem~15.7.1]{mat-discrete-geo}) that there is an isometric embedding $\iota$ of the metric space $(A , d|_{A})$ into $(\BB R^N ,d^\infty)$ for $N = \# A$ (here we recall that $d^\infty$ is the $L^\infty$ metric on $\BB R^N$).   
For $k\in [-K^2,K^2]_{\BB Z}$, let $ \wh\eta_k$ be the straight line path in $\BB R^N$ from $\iota(y_{k-1})$ to $\iota(y_k)$ with constant $d^\infty$-speed which is traversed in $1/K$ units of time. By our choice of the $y_k$'s, $\wh\eta_k \in B_\ep(\iota(y_k) ; d^\infty)$. 
Let $\wh\eta$ be the concatenation of the paths $\wh \eta_k$.

Let $\wh X := \iota(A) \cup \wh \eta$. Let $\wh d := d^\infty|_{\wh X}$. Let $\wh\mu$ be the measure on $\wh X$ which is the pushforward of $\mu_0$ under $\iota$. Then $\wh{\frk X} := (\wh X , \wh d, \wh\mu,\wh\eta) \in \wh{\BB F}$. 
 
It remains only to compare $\frk X$ with $\wh{\frk X}$. Let $W$ be the set obtained from $X\sqcup \wh X$ by identifying $A \subset X$ with $\iota(A) \subset \wh X$. Let $D$ be the metric on $W$ which restricts to $d$ (resp.\ $\wh d$) on $X$ (resp.\ $\wh X$) and which is defined for $x\in X$ and $\wh x \in \wh X$ by
\eqbn
d_Z(x ,\wh x) = d_Z(\wh x, x) = \inf_{a \in A} \left( d(x,a) + \wh d(\iota(a) , \wh x) \right) .
\eqen
Note that this is well defined and satisfies the triangle inequality since $\iota$ is an isometry. 
By our choices of $A$, $\mu_0$, $K$, and the path $\wh\eta$, it follows that the GHPU distortion of $(W,D)$ and the natural inclusions of $X$ and $\wh X$ into $W$ is at most $4\ep$, so since $\ep > 0$ is arbitrary we obtain the desired separability.
\end{proof}

\begin{proof}[Proof of Proposition~\ref{prop-ghpu-metric}]
This follows by combining Lemmas~\ref{prop-ghpu-triangle},~\ref{prop-ghpu-pos-def},~\ref{prop-ghpu-complete}, and~\ref{prop-ghpu-separable}. 
\end{proof}

\subsection{Proofs for the local GHPU metric}
\label{sec-ghpu-local}

In this subsection we will prove Propositions~\ref{prop-local-ghpu-metric} and~\ref{prop-ghpu-embed-local}. 
These statements will, for the most part, be deduced from the analogous statements in the compact case proven in Section~\ref{sec-ghpu-properties}.
We first state a lemma which will enable us to relate GHPU convergence and local GHPU convergence. 
For the statement of this lemma and in what follows, it will be convenient to have the following definition.

\begin{defn}[Good radius] \label{def-good-radius}
Let $\frk X = (X,d,\mu,\eta) \in \BB M_\infty^{\op{GHPU}}$ and $r>0$. We say that~$r$ is a \emph{good radius} for~$\frk X$ if
\eqb  \label{eqn-no-bdy-mass}
\mu \left( \bdy B_r(\eta (0) ;d  )  \right) = 0  \quad \op{and} \quad |\eta^{-1}\left( \bdy B_r(\eta(0) ; d) \right)| = 0,
\eqe 
where $|\eta^{-1}\left( \bdy B_r(\eta(0) ; d) \right)|$ is the Lebesgue measure of $\eta^{-1}\left( \bdy B_r(\eta(0) ; d) \right)$.
\end{defn}
Since the sets $ \bdy B_r(\eta(0) ; d)$ for $r>0$ are disjoint, it follows that all but countably many radii are good.

\begin{lem} \label{prop-ball-ghpu-conv}
Let $\frk X   = (X   , d , \mu   , \eta  )  \in \BB M^{\op{GHPU}}$ and suppose that for each $\zeta>0$, each point $x\in X$ can be joined to $\eta(0)$ by a path of $d$-length at most $d(\eta(0) , x) + \zeta$. 
Let $R > r > 0$ and suppose that the radius $r$ is good, in the sense of~\eqref{eqn-no-bdy-mass}. 
For each $\ep> 0$, there exists $\delta> 0$ depending only on $\ep$ and the $R$-truncation $\frk B_R\frk X $ (Definition~\ref{def-ghpu-truncate}) such that the following is true.  
Let $\wt{\frk X}  = (\wt X   , \wt d , \wt\mu   , \wt\eta  ) \in \BB M^{\op{GHPU}} $ and suppose that
$d_\sqcup$ is a metric on $X\sqcup \wt X$ which restricts to $d$ on $X$ and $\wt d$ on $\wt X$ and whose GHPU distortion (recall~\eqref{eqn-ghpu-coprod-var}) satisfies $\op{Dis}_{\frk X , \wt{\frk X}}^{\op{GHPU},\sqcup}(d_\sqcup) \leq \delta$. 
Then 
\eqb \label{eqn-ball-ghpu-conv}
\op{Dis}_{\frk B_r\frk X , \frk B_r\wt{\frk X}}^{\op{GHPU},\sqcup}(d_\sqcup) \leq \ep .
\eqe 
\end{lem} 

The proof of Lemma~\ref{prop-ball-ghpu-conv} is a straightforward but tedious application of the definitions together with the triangle inequality, so we omit it. 
In light of Lemma~\ref{prop-ghpu-coprod}, Lemma~\ref{prop-ball-ghpu-conv} implies that if $\BB d^{\op{GHPU}}(\frk X ,\wt{\frk X}) < \delta$, then  $\BB d^{\op{GHPU}}\left(\frk B_r\frk X ,\frk B_r\wt{\frk X} \right) \leq \ep$.

As a consequence of Lemma~\ref{prop-ball-ghpu-conv}, we see that local GHPU convergence is really just GHPU convergence of curve-decorated metric measure spaces truncated at an appropriate sequence of metric balls. 

\begin{lem} \label{prop-local-ghpu-subsequence}
Let $\frk X $ and $\{\frk X^n\}_{n\in\BB N}$ be elements of $ \BB M_\infty^{\op{GHPU}}$. 
If $\{r_k\}_{k\in\BB N}$ is a sequence of positive real numbers tending to $\infty$ and the $r_k$-truncations (Definition~\ref{def-ghpu-truncate}) satisfy $\frk B_{r_k} \frk X^n \rta \frk B_{r_k} \frk X$ in the GHPU topology for each $k\in\BB N$, then $\frk X^n \rta \frk X$ in the local GHPU topology. 
Conversely, if $\frk X^n \rta \frk X$ in the local GHPU topology, then for each good radius $r$ (Definition~\ref{def-good-radius}), we have $\frk B_r \frk X^n \rta \frk B_r \frk X$ in the GHPU topology. 
\end{lem} 
\begin{proof}
Let $\mcl R$ be the set of good radii $r>0$ for $\frk X$. Recall that $(0,\infty)\setminus \mcl R$ is countable. 

Suppose first that we are given a sequence $r_k\rta\infty$ such that $\frk B_{r_k} \frk X^n \rta \frk B_{r_k} \frk X$ in the GHPU topology for each $k\in\BB N$. By Lemma~\ref{prop-ball-ghpu-conv} applied with $\frk B_{r_k} \frk X$ for $r_k > r$ in place of $\frk X$, we find that for each $r\in \mcl R$, $\BB d^{\op{GHPU}}\left(\frk B_r \frk X^n ,\frk B_r \frk X \right) \rta 0$. By the dominated convergence theorem applied to the formula~\eqref{eqn-ghpu-local-def}, it follows that $\frk X^n\rta \frk X$ in the local GHPU topology. 

Conversely, suppose $\frk X^n\rta \frk X$ in the local GHPU topology and let $r \in \mcl R$. By Lemma~\ref{prop-ball-ghpu-conv} applied with $R = r +1$, and $  \frk B_R \frk X$ for $R >r  + 1$ in place of $\frk X$, for each $k\in\BB N$ and each $\ep >0$, there exists $\delta = \delta(k,\delta) > 0$ such that whenever $R \geq r +1$ and $\BB d^{\op{GHPU}}\left(\frk B_R \frk X^n ,\frk B_R \frk X \right)  \leq \delta$, we have $\BB d^{\op{GHPU}}\left(\frk B_r \frk X^n ,\frk B_r \frk X \right) \leq \ep$. 
Local GHPU convergence implies that for large enough $n\in\BB N$, there exists $R \geq r +1$ with $\BB d^{\op{GHPU}}\left(\frk B_R \frk X^n ,\frk B_R \frk X \right)  \leq \delta$. Hence $\frk B_r \frk X^n \rta \frk B_r \frk X$.
\end{proof}

\begin{proof}[Proof of Proposition~\ref{prop-ghpu-embed-local}]
It is clear from Lemma~\ref{prop-local-ghpu-subsequence} that the existence of $(W, D)$ and isometric embeddings as in the statement of the lemma implies that $\frk X^n\rta \frk X$ in the local GHPU topology.

Conversely, suppose $\frk X^n\rta \frk X$ in the local GHPU topology.
The proof of this direction is a generalization of that of Proposition~\ref{prop-ghpu-embed}. 
To lighten notation, for $r>0$ and $n\in\BB N$ we write
\eqbn
B_r^n  := B_r(\eta^n(0) ; d^n) \quad \op{and} \quad B_r := B_r(\eta (0) ; d ) .
\eqen 

Choose a sequence of good radii $r_k\rta\infty$ for $\frk X$. 
By Lemma~\ref{prop-local-ghpu-subsequence}, we have, in the notation of Definition~\ref{def-ghpu-truncate}, $\frk B_{r_k} \frk X^n \rta \frk B_{r_k} \frk X$ for each $n\in\BB N$.  
For each $k \in \BB N$, choose $N_k \in \BB N$ such that for $n\geq N_k$, we have $\BB d^{\op{GHPU}}_\infty\left(\frk B_{r_k} \frk X^n , \frk B_{r_k} \frk X \right) \leq 1/k$. For $n\in\BB N$, let $k_n$ be the largest $k \in \BB N$ such that $N_k \leq n$, and note that $k_n \rta \infty$ as $n\rta\infty$. 
 
By Lemma~\ref{prop-ghpu-coprod}, for each $n\in\BB N$ there exists a metric $\wt d_\sqcup^n$ on $B_{r_{k_n} } \sqcup B_{r_{k_n} }^n  $ which restricts to $d^n$ on $B_{r_{k_n}}^n$ and $d$ on $B_{r_{k_n}}$ and whose GHPU distortion is at most $1/k_n$. 
We now extend $\wt d_\sqcup^n$ to a metric $d_\sqcup^n$ on $X\sqcup X$ by the formula
\eqbn
d_\sqcup^n(x,y) := 
\begin{dcases}
d(x,y) \quad &x,y\in X \\
d^n(x,y) \quad &x,y \in X^n \\
\inf_{(u,v) \in B_{r_{k_n} \times B_{r_{k_n} }^n} } \left( d (x, u) + d^n( y,v) +  \wt d_\sqcup^n(u,v) \right) \quad &x \in X ,\, y \in X^n  
\end{dcases}
\eqen
together with the requirement that $d_\sqcup^n(y,x) = d_\sqcup^n(x,y)$ when $y\in X$ and $x\in X^n$. It is easily verified that $d_\sqcup^n$ is a metric on $X\sqcup X^n$ and that $d_\sqcup^n|_{ B_{r_{k_n}} \sqcup B_{r_{k_n}}^n} = \wt d_\sqcup^n$. 

As in the proof of Proposition~\ref{prop-ghpu-embed}, let $W:= X\sqcup \bigsqcup_{n=1}^\infty X^n$ and identify $X$ and each $X^n$ with its natural inclusion into $W$. 
We define a metric $D$ on $W$ as follows. 
If $x , y \in W$ such that $x,y\in X \sqcup X^n$ for some $n\in\BB N$, we set $D(x,y) = d_\sqcup^n(x,y)$. 
If $x \in X^n$ and $y\in X^m$ for some $n,m\in\BB N$, we set
\eqbn
D(x,y) = \inf_{u \in X} \left( d_\sqcup^n(x, u) + d_\sqcup^m(u , y) \right) .
\eqen
Then $D$ is a metric on $W$ which restricts to $d^n$ on each $X^n$ and to $d$ on $X$.

For each $n\in\BB N$, the restriction of $D$ to $B_{r_{k_n}} \sqcup B_{r_{k_n}}^n$ agrees with the corresponding restriction of $d_\sqcup^n$, which agrees with $\wt d_\sqcup^n$. Since the GHPU distortion of $\wt d_\sqcup^n$ is at most $1/k_n \rta 0$ and $r_{k_n}\rta\infty$ as $n\rta\infty$, it follows from Lemma~\ref{prop-ball-ghpu-conv} that $\frk X^n\rta \frk X$ in the $D$-local HPU topology. 
 
By possibly replacing $W$ with its metric completion, we can take $W$ to be complete. Since $(X,d)$ and each $(X^n,d^n)$ is boundedly compact (since they are locally compact length spaces and by the Hopf-Rinow Theorem), it follows easily from $D$-Hausdorff metric convergence of the metric balls $B_r(\eta^n(0) ; d^n)$ that each metric ball in $X$ with finite radius is totally bounded, hence compact. 
\end{proof}

Next we record a compactness criterion for the local GHPU metric which will be used to prove completeness.
 
\begin{lem}[Compactness criterion] \label{prop-ghpu-compact-cond-local}
Let $\mcl K$ be a subset of $ \BB M_\infty^{\op{GHPU}}$ and suppose that there is a sequence $r_k \rta\infty$ such that for each $k\in\BB N$, the set of $r_k  $-truncations $\frk B_{r_k}\mcl K = \{\frk B_{r_k} \frk X \,:\, \frk X \in \mcl K\}$ (Definition~\ref{def-ghpu-truncate}) satisfies the conditions of Lemma~\ref{prop-ghpu-compact-cond} with $\frk B_{r_k} \frk X$ in place of $\frk X$, i.e.\ $\frk B_{r_k} \mcl K$ is totally bounded, has bounded total mass, and is equicontinuous. Then every sequence in $\mcl K$ has a subsequence which converges with respect to the local GHPU metric.
\end{lem} 
\begin{proof}
By the compactness criterion for the local Gromov-Hausdorff-Prokhorov topology~\cite[Theorem~2.9]{adh-ghp}, the set of metric measure spaces $\{(X,d,\mu , \eta )\,:\, (X,d,\mu,\eta) \in \mcl K\}$ is pre-compact with respect to the pointed local Gromov-Hausdorff-Prokhorov topology. 
Hence for any sequence $\{\frk X^n = (X^n,d^n,\mu^n,\eta^n)\}_{n\in\BB N}$ of elements of $\mcl K$, there exists a sequence $n_k\rta\infty$ and a locally compact pointed length space equipped with a finite measure $(X ,d ,\mu, x)$ such that $(X^{n_k} ,d^{n_k},\mu^{n_k},\eta^{n_k}(0)) \rta (X,d,\mu,x)$ in the pointed local Gromov-Hausdorff-Prokhorov topology.  

By the analog of Proposition~\ref{prop-ghpu-embed-local} for local Gromov-Hausdorff-Prokhorov convergence (which follows from Proposition~\ref{prop-ghpu-embed} by taking $\eta$ to be a constant curve) we can find a boundedly compact metric space $(W,D)$ and isometric embeddings of $(X^{n_k} ,d^{n_k})$ for $k\in\BB N$ and $(X,d)$ into $(W,D)$ such that if we identify these spaces with their embeddings, then $\eta^{n_k}(0)\rta x$, $B_r(\eta^{n_k}(0) ; d^{n_k}) \rta B_r(x ; d)$ in the $D$-Hausdorff metric for each $r > 0$, and $\mu^{n_k}|_{B_r(\eta^n(0) ;d^n)} \rta \mu|_{B_r(\eta(0) ;d)}$ in the $D$-Prokhorov metric for each $r>0$ such that $\mu\left(\bdy B_r(\eta(0) ;d)\right) = 0$.
By the Arz\'ela-Ascoli theorem and a diagonalization argument, after possibly passing to a further subsequence we can find a curve $\eta$ in $X$ such that $(X^{n_k} ,d^{n_k},\mu^{n_k},\eta^{n_k}) \rta (X,d,\mu,\eta)$ in the $D$-HPU sense (Definition~\ref{def-hpu-local}), so by Proposition~\ref{prop-ghpu-embed-local} $\frk X^{n_k} \rta (X,d,\mu,\eta)$ in the local GHPU metric.
\end{proof}

\begin{proof}[Proof of Proposition~\ref{prop-local-ghpu-metric}]
Symmetry and the triangle inequality are immediate from the formula~\eqref{eqn-ghpu-local-def} and the analogous properties for the GHPU metric, so $\BB d_\infty^{\op{GHPU}}$ is a pseudometric.

The fact that $\BB d_\infty^{\op{GHPU}}(\frk X_1,\frk X_2) = 0$ implies that $\frk X_1$ and $\frk X_2$ agree as elements of $\ol{\BB M}_\infty^{\op{GHPU}}$ follows from Proposition~\ref{prop-ghpu-embed-local}, Lemma~\ref{prop-isometry-limit}, and the same argument used to prove Lemma~\ref{prop-ghpu-pos-def}. 

Completeness follows from Lemma~\ref{prop-ghpu-compact-cond-local} (c.f.\ the proof of Lemma~\ref{prop-ghpu-complete}). 

Finally, we check separability. 
We know that $\BB M^{\op{GHPU}}$ is separable, so the set $\BB M^{\op{GHPU}}\cap \BB M_\infty^{\op{GHPU}}$ consisting of length spaces in $\BB M^{\op{GHPU}}$ is separable with respect to the GHPU metric. 
By Lemma~\ref{prop-local-ghpu-subsequence}, $\BB M^{\op{GHPU}}\cap \BB M_\infty^{\op{GHPU}}$ is also separable with respect to the local GHPU metric. 
So, it suffices to show that $\BB M^{\op{GHPU}}\cap \BB M_\infty^{\op{GHPU}}$ is dense in $\BB M_\infty^{\op{GHPU}}$. 

We cannot approximate an element of $\frk X = (X,d,\mu,\eta) \in \BB M_\infty^{\op{GHPU}}$ by the $r$-truncation $\frk B_r\frk X$ since $ B_r(\eta(0) ;d)  $ may not be a length space with the restricted metric. We instead construct a modified version of $\frk B_r\frk X$ which is a length space. 
Given $r > 0$, let $X_r$ be the quotient space of $B_r(\eta(0) ; d)$ under the equivalence relation which identifies $\bdy B_r(\eta(0) ; d)$ to a point. 
Let $\wt d_r$ be the quotient metric on $X_r$ and let $p_r : B_r(\eta(0); d) \rta X_r$ be the quotient map. 

We note that $\wt d_r$ is a metric, not a pseudometric since every point of $B_r(\eta(0) ;d) \setminus \bdy B_r(\eta(0) ; d)$ lies at positive distance from $\bdy B_r(\eta(0) ; d)$. Furthermore, the triangle inequality implies that the restriction of $\wt d_r$ to $p_r( B_{r/3}(\eta(0) ; d) )$ coincides with the corresponding restriction of $d$, pushed forward under $p_r$.  

Let $d_r$ be the smallest length metric on $X_r$ which is greater than or equal to $\wt d_r$. Equivalently, for $x,y\in X_r$, $d_r(x,y)$ is the infimum of the $\wt d_r$-lengths of paths from $x$ to $y$ contained in $X_r$. Then $X_r$ is complete and totally bounded with respect to $d_r$, so is compact with respect to $d_r$.

 Since $d$ is a length metric, we find that the restriction of~$d_r$ to $p_r( B_{r/3}(\eta(0) ; d) )$ coincides with the corresponding restriction of~$d$, pushed forward under~$p_r$. 

Let 
\eqbn
\frk X_r := \left( X_r , d_r , (p_r)_*\mu , p_r\circ \frk B_r \eta \right) ,
\eqen 
with $\frk B_r\eta$ as in Definition~\ref{def-ghpu-truncate}. 
Then $\frk X_r \in \BB M^{\op{GHPU}} \cap \BB M_\infty^{\op{GHPU}}$ and $\frk B_{r/3} \frk X_r$ agrees with $\frk B_{r/3} \frk X$ as elements of $\ol{\BB M}^{\op{GHPU}}$. 
Hence $\BB d_\infty^{\op{GHPU}}(\frk X , \frk X_r) \leq e^{-r/3}$. Since $r>0$ is arbitrary, we obtain the desired density. 
\end{proof}

\subsection{Conditions for GHPU convergence using the $m$-fold Gromov-Prokhorov topology}
\label{sec-ghpu-condition}

In this subsection we prove a lemma giving conditions for GHPU convergence which will be used in subsequent sections to prove convergence of uniform quadrangulations with boundary to the Brownian disk in the GHPU topology. To state the lemma we first need to consider a variant of the Gromov-Prokhorov topology~\cite{gpw-metric-measure}, which we now define. 

For $m\in\BB N$ let $\BB M^{\op{GP}}_m$ be the space of $(m+2)$-tuples $(X,d , \mu_1 , \dots, \mu_m)$ where $(X,d)$ is a separable metric space and $\{\mu_r\}_{  r\in [1,m]_{\BB Z}}$ are finite Borel measures on $X $. 
Given $k\in \BB N$ and $r \in [1,m]_{\BB Z}$ let $\{x^{j +  (r-1)k  } \}_{j\in [1, k]_{\BB Z}}$ be i.i.d.\  samples from $\mu_r$ and let $M_k(X,d, \mu_1,\dots ,\mu_m)$ be the $k m\times km$ matrix whose $i,j$th entry is $d(x^i , x^j)$ for $i , j \in [1,km]_{\BB Z}$.  
We define the \emph{$m$-fold Gromov-Prokhorov (GP) topology} on $\BB M^{\op{GP}}$ to be the weakest one for which the functional
\eqb \label{eqn-m-gp-functional}
 (X,d, \mu_1,\dots ,\mu_m) \mapsto \BB E\left[ \phi \left( M_k(X,d, \mu_1,\dots ,\mu_m)  , \mu_1(X) ,\dots  , \mu_m(X) \right) \right]
\eqe 
is continuous for each bounded continuous function $\phi : \BB R^{(km)^2} \times \BB R^m \rta \BB R$.  
Note that convergence in the $m$-fold GP topology is equivalent to convergence of each of these functionals. 

If $x \in X$ is a marked point, we set $x^0 = x$ and define $M_k^\bullet(X,d, \mu_1,\dots ,\mu_m , x)$ to be the be the random $(km+1) \times (km+1)$ matrix whose $(i,j)$th entry is $d(x^i , x^j)$ for $i,j \in [0,km]_{\BB Z}$, where $x^i$ for $i\in [1,km]_{\BB Z}$ is defined as above. We define the \emph{$m$-fold Gromov-Prokhorov (GP) topology} on pointed metric spaces with $m$ finite Borel measures to be the weakest one for which the functional
\eqb \label{eqn-m-gp-functional-pt}
 (X,d, \mu_1,\dots ,\mu_m ,x) \mapsto \BB E\left[ \phi \left( M_k^\bullet(X,d, \mu_1,\dots ,\mu_m , x)  , \mu_1(X) ,\dots  , \mu_m(X) \right) \right] 
\eqe 
is continuous for each bounded continuous function $\phi : \BB R^{(km+1)^2} \times \BB R^m \rta \BB R$.  

Like the Gromov-Prokhorov topology, the $m$-fold Gromov-Prokhorov topology also separates points. 
 
\begin{lem} \label{prop-2gp-determined}
Let $(X , d , \mu_1 , \dots , \mu_m)$ and $(\wt X , \wt d , \wt\mu_1 , \dots  , \wt\mu_m)$ be elements of $\BB M^{\op{GP}}_m$.  Suppose that the functionals~\eqref{eqn-m-gp-functional} agree on $(X , d , \mu_1 , \dots , \mu_m)$ and $(\wt X , \wt d , \wt\mu_1 , \dots  , \wt\mu_m)$ for each $k\in \BB N$.  Suppose further that the union of the closed supports of $\mu_r$ (resp.\ $\wt\mu_r$) for $r\in [1,m]_{\BB Z}$ is all of $X$ (resp.\ $\wt X$). Then there is an isometry $f : X\rta \wt X$ with $f_* \mu_r = \wt \mu_r$ for each $r\in [1,m]_{\BB Z}$. If $X$ and $\wt X$ are endowed with marked points $x$ and $\wt x$, respectively, for which the functionals~\eqref{eqn-m-gp-functional-pt} agree for each $k\in\BB N$, we can also take $f$ to satisfy $f(x) = \wt x$. 
\end{lem}
\begin{proof}
We treat the unpointed case; the pointed case is treated in an identical manner. 
For $r\in [1,m]_{\BB Z}$, let $\{  x^j_r  \}_{j\in \BB N}$ and $\{  \wt x^j_r  \}_{j\in\BB N}$ be i.i.d.\  samples from $\mu_r$ and $\wt\mu_r$, respectively. By our assumption about the supports of the measures $\mu_r$ and $\wt\mu_r$, the sets $ \{x^j_r : (r,j) \in [1,m]_{\BB Z} \times \BB N \} $ and $  \{\wt x^j_r : (r,j) \in [1,m]_{\BB Z} \times \BB N \}  $ are a.s.\ dense in $X$ and $\wt X$, respectively. The agreement of the functionals~\eqref{eqn-m-gp-functional} implies that $\mu_r(X) = \wt\mu_r(\wt X)$ for each $r\in [1,m]_{\BB Z}$. Furthermore, the collections of distances $\{d(x^j_r , x_{r'}^{j'}) :  (r,j) , (r',j') \in [1,m]_{\BB Z} \times \BB N \}$ and $\{\wt d(\wt x_r^j , \wt x_{r'}^{j'} ) : (r,j) , (r',j') \in [1,m]_{\BB Z} \times \BB N \}$ agree in law, so we can couple everything in such a way that these collections of distances agree a.s. Let $f : \{x^j_r : (r,j) \in [1,m]_{\BB Z} \times \BB N \} \rta \wt X$ be the function which sends each $x_j^r$ to $\wt x_j^r$. By our choice of coupling $f$ is distance preserving, and since the domain and image of $f$ are a.s.\ dense in $X$ and $\wt X$, respectively, $f$ extends by continuity to an isometry $X\rta \wt X$. 

For each $N\in\BB N$, the set $\{x_r^j : (r,j) \in [1,m]_{\BB Z} \times [N,\infty)_{\BB Z} \} $ is still a.s.\ dense in $X$, so the restriction of~$f$ to this set a.s.\ determines~$f$. By the Kolmogorov zero-one law, $f$ is a.s.\ equal to a certain deterministic map $X\rta \wt X$. 
In particular, for each $r\in [1,m]_{\BB Z}$ and each $A\subset X$, we have
\eqbn
\mu_r(A) = \mu_r(X) \BB P\left[ x_r^1 \in A \right] = \wt\mu_r(X) \BB P\left[  \wt x_1^r \in f(A) \right] = \wt\mu_r(f(A)), 
\eqen 
so $f_* \mu_r = \wt\mu_r$.
\end{proof}

We now state our condition for GHPU convergence.

\begin{lem} \label{prop-ghpu-condition}
Let $\frk X^n = (X^n ,d^n , \mu^n , \eta^n)$ for $n\in\BB N$ and $\frk X = (X , d , \mu , \eta)$ be elements of $\BB M^{\op{GHPU}}$. 
Suppose that the curves $\eta^n$ for $n\in\BB N$ (resp.\ $\eta$) are each constant outside some bounded interval $[0,T^n]$ (resp.\ $[0,T]$). 
Let $\nu^n$ (resp.\ $\nu$) be the pushforward of Lebesgue measure on $[0,T^n]$ (resp.\ $[0,T]$) under $\eta^n$ (resp.\ $\eta$). 
Suppose the following conditions are satisfied.
\begin{enumerate}
\item $(X^n ,d^n , \eta(0) ) \rta (X,d , \eta(0) )$ in the pointed Gromov-Hausdorff metric. \label{item-ghpu-gh}
\item $(X^n ,d^n , \mu^n , \nu^n , \eta^n(0) ) \rta (X , d, \mu , \nu , \eta(0))$ in the 2-fold GP topology. \label{item-ghpu-gp} 
\item The curves $\{\eta^n\}_{n\in\BB N}$ are equicontinuous, i.e.\ for each $\ep >0$ there exists $\delta>0$ such that for each $n\in\BB N$, we have $d^n(\eta^n(s), \eta^n(t)) \leq \ep$ whenever $s,t\in [0,T^n]$ with $|s-t| \leq \delta$.  \label{item-ghpu-equi}
\item The closed support of $\mu$ is all of $X$. \label{item-ghpu-support}
\item $\eta$ is a simple curve. \label{item-ghpu-lebesgue}
\end{enumerate}
Then $\frk X^n\rta \frk X$ in the GHPU topology. 
\end{lem} 
\begin{proof}
By Lemma~\ref{prop-ghpu-compact-cond} and condition~\ref{item-ghpu-gh}, we can find a sequence $\mcl N\subset \BB N$ and $\wt{\frk X} = ( \wt X,  \wt d, \wt \mu,\wt\eta) \in \BB M^{\op{GHPU}}$ such that $(X,d,\eta(0))$ and $(\wt X,\wt d , \wt\eta(0))$ are isometric as pointed metric spaces and $\frk X^n\rta \frk X$ in the GHPU topology. 
By Proposition~\ref{prop-ghpu-embed}, there is a compact metric space $(W ,D)$, and isometric embeddings of $X^n$ for $n\in\mcl N$ and $X$ into $(W,D)$ such that if we identify $X^n$ and $X$ with their images under these embeddings, we have $\frk X^n\rta \frk X$ in the $D$-HPU sense as $\mcl N\ni n\rta\infty$.

From the convergences $\nu^n \rta \wt\nu$, $\mu^n \rta \wt\mu$, and $\eta^n(0) \rta \wt\eta(0)$, we infer that $(X^n ,d^n , \mu^n , \nu^n , \eta^n(0) ) \rta (\wt X , \wt d, \wt\mu , \wt\nu , \wt\eta(0) )$ in the 2-fold GP topology. Hence condition~\ref{item-ghpu-support} and Lemma~\ref{prop-2gp-determined} imply that we can find a distance-preserving map $f :  X \rta \wt X$ (whose range is the union of the closed supports of $\wt\nu$ and $\wt\mu$) such that $f_* \wt\mu =  \mu$, $f_* \wt\nu =\nu$, and $f(\wt\eta(0)) = \eta(0)$.  Since $(X,d)$ and $(\wt X ,\wt d)$ are isometric as metric spaces and a compact metric space cannot be isometric to a proper subset of itself, we find that $f$ is in fact surjective, so is an isometry from $X$ to $\wt X$. 

We claim that also $f\circ \wt\eta = \eta$, which will imply that $(X , d , \wt\mu , \wt\eta) = (X , d  , \mu,\eta)$ as elements of $\ol{\BB M}^{\op{GHPU}}$.

Since each $\eta^n$ is constant outside of $[0,T^n]$ and $T^n\rta T$, we find that $\wt\eta$ is constant outside of $[0,T]$. 
It is clear that $\wt\nu$ is supported on $ \wt\eta$. 
For $0\leq a\leq b \leq T$ we have $\eta^n([a,b]) \rta \wt\eta([a,b])$ in the $D$-Hausdorff distance, so since $\nu^n \rta \wt \nu$ it follows that $\wt\nu\left( \wt\eta([a,b]) \right) \geq b-a$.
By condition~\ref{item-ghpu-lebesgue} and the existence of the isometry $f$ above, the measure $\wt\nu$ has no point masses so for each $\ep >0$ we can find $\delta>0$ such that $\wt\nu (B_\delta(\wt\eta([a,b]) ) ) \leq b-a + \ep$. Hence in fact $\wt\nu\left( \wt\eta([a,b]) \right) = b-a$. 
Therefore $\wt\nu$ is the pushforward under $\wt\eta$ of Lebesgue measure on $[0,T]$. 

The map $f$ takes the closed support of $\wt\nu$ to the closed support of $\nu$, so takes the range of $\wt\eta$ to the range of $\eta$. Since $f_* \wt\nu = \nu$ and $f(\wt\eta(0)) = \eta(0)$, for each $t\in [0,T]$ the set $f(\wt\eta([0,t]))$ is a connected subset of $\wt\eta$ containing $0$, with $\wt\nu$-mass equal to $t$. Therefore $\wt\eta^{-1}(f(\eta([0,t]))) = [0,t]$. In other words, the map $\wt\eta^{-1} \circ f \circ  \eta$ is the identity, so $f\circ  \eta = \wt\eta$. 
\end{proof}

\section{Schaeffer-type constructions}
\label{sec-quad-prelim}

In this section we will review constructions from the planar map literature which are needed for the proofs of our scaling limit results.
In Sections~\ref{sec-quad-bdy} and~\ref{sec-uihpq}, respectively, we review the Schaeffer-type constructions of uniform quadrangulations with boundary and the UIHPQ$_{\op{S}}$. 
In Section~\ref{sec-uihpq-dist},  we record some distance estimates in terms of the encoding functions in these constructions.
In Section~\ref{sec-pruning}, we recall how to ``prune" the UIHPQ to get an instance of the UIHPQ$_{\op{S}}$. 
In Section~\ref{sec-brownian-disk}, we recall the definition of the Brownian disk from~\cite{bet-mier-disk}.

\subsection{Encoding quadrangulations with boundary} 
\label{sec-quad-bdy}

Recall from Section~\ref{sec-quad-def} the set $\mcl Q^{ \bullet}(n,l)$ of boundary-rooted, pointed quadrangulations with $n$ internal faces and $2l$ boundary edges. In this subsection we review a variant of the Schaeffer bijection for elements of $\mcl Q^{ \bullet}(n,l)$ which is really a special case of the Bouttier-Di Francesco-Guitter bijection~\cite{bdg-bijection}. Our presentation is similar to that in~\cite[Section~3.3]{curien-miermont-uihpq} and~\cite[Section~3.3]{bet-mier-disk}. 
See Figure~\ref{fig-schaeffer-disk} for an illustration.

\begin{figure}[ht!]
 \begin{center}
\includegraphics{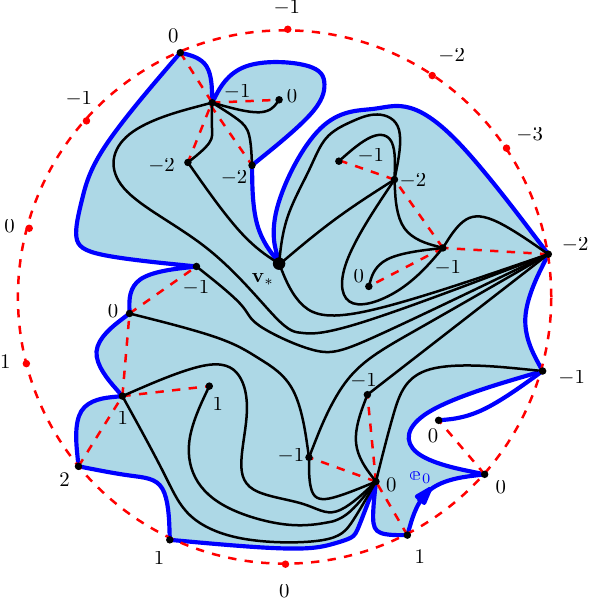} 
\caption{\label{fig-schaeffer-disk}Illustration of the Schaeffer-type encoding of a boundary-rooted, pointed quadrangulation with boundary $(Q, \BB e_0 , \BB v_*)$. The region enclosed by~$Q$ is shown in light blue. The graph~$F$ containing the trees~$\frk t_{k}$ is the union of the dashed red lines and the black vertices. The red vertices correspond to upward steps of~$b^0$, and do not belong to~$F$ or~$Q$. Each vertex is labeled with either its label $L (v)$ or the corresponding value of~$b^0$. The black edges are part of the interior of~$Q$ and the blue edges are part of~$\bdy Q$. The oriented boundary root edge $\BB e_0$ is indicated with an arrow. }
\end{center}
\end{figure}

For $l\in\BB N$, a \emph{bridge} of length $2l$ is a function $b^0 : [0,2l]_{\BB Z} \rta \BB Z$ such that $b^0(j+1) - b^0(j ) \in \{-1,1\}$ for each $j\in [0,2l-1]_{\BB Z}$ and $b^0(0) = b^0(2l) = 0$.  For a bridge $b^0$, we associate a function $b : [0, l]_{\BB Z} \rta \BB Z$ as follows. We set $j_0 = 0$ and for $k \in [1,l]_{\BB Z}$ we let $j_k$ be the $k$th smallest $j\in [0,2l-1]_{\BB Z}$ for which $b^0(j+1) - b^0(j) = - 1$. We then let $b(k) := b^0(j_k)$. 

For $n,l\in\BB N$, a \emph{treed bridge} of area $n$ and boundary length $2l$ is an $(l+1)$-tuple $(b^0 ; (\frk t_0 , \frk v_0, L_0) , \ldots , (\frk t_{l-1} , \frk v_{l-1} , L_{l-1}) )$ where $b^0$ is a bridge of length $2 l$ and $(\frk t_k , \frk v_k, L_k)$ for $k\in [0,l-1]_{\BB Z}$ is a rooted plane tree with a label function $L_k : \mcl V(\frk t_k) \rta \BB Z$ satisfying $ L_k(v) - L_k(v')  \in \{-1,0,1\}$ whenever $v$ and $v'$ are joined by an edge and $L_k(\frk v_k) = b(k)$, where~$b$ is constructed from~$b^0$ as above, such that the total number of edges in the trees $\frk t_k$ is $n$. Let $\mcl T^{ \bullet}(n,l)$ be the set of treed bridges of area~$n$ and boundary length~$2l$, together with a sign $\xi \in \{-,+\}$ (which will be used to determine the orientation of the root edge). 

We associate a treed bridge with a rooted, labeled planar map $(F ,e_0 , L)$ with two faces as follows.  Draw an edge from $\frk v_k$ to $\frk v_{k+1}$ for each $k\in [0,l-2]_{\BB Z}$ and an edge from $\frk v_{l-1}$ to $\frk v_0$. This gives us a cycle which we embed into $\BB C$ in such a way that the vertices $\frk v_k$ all lie on the unit circle. We extend this embedding to the trees $\frk t_k$ in such a way that each is mapped into the unit disk. This gives us a planar map $F$ with an inner face of degree $2n +  l$ (containing all of the trees $\frk t_k$) and an outer face of degree $ l$. Let $ e_0$ be the oriented edge of $F$ from $\frk v_{l-1}$ to $\frk v_0$. Let $L$ be the label function on vertices of $F$ inherited from the label functions $L_k$ for $k\in[0,l-1]_{\BB Z}$. 

We now associate a rooted, pointed quadrangulation with boundary to $(F , e_0 , L)$ and the sign $\xi$ via a variant of the Schaeffer bijection. Let $\BB p : [0,2n+l]_{\BB Z} \rta \mcl V(F)$ be the contour exploration of the inner face of $F$ started from $\frk v_1$, i.e.\ the concatenation of the contour explorations of the trees $\frk t_0, \ldots , \frk t_{l-1}$. Also define (by a slight abuse of notation) $L(i) = L(\BB p(i))$.  Note that each $\BB p(i)$ for $i\in [0,2n+l]_{\BB Z}$ is associated with a unique corner of the inner face of $F$ (i.e.\ a connected component of $B_\ep(\BB p(i) ) \setminus F$ for small $\ep > 0$).  Let $\BB v_*$ be an extra vertex not connected to any vertex of $F$, lying in the interior face of $F$.  For $i\in [0,2n+l]_{\BB Z}$, define the successor $s(i)$ of $i$ to be the smallest $i' \geq i$ (with elements of $[0,2n+l]_{\BB Z}$ viewed modulo $2n+l$) such that $L(i') = L(i)-1$, or let $s(i) = \infty$ if no such $i'$ exists. For $i\in [0,2n+l]_{\BB Z}$, draw an edge from (the corner associated with) $\BB p(i)$ to (the corner associated with) $\BB p(s(i))$, or an edge from $\BB p(i)$ to $\BB v_*$ if $s(i) = \infty$. Then, delete all of the edges of $F$ to obtain a map $Q$. We take $Q$ to be rooted at the oriented edge $\BB e_0 \in\mcl E(\bdy Q)$ from $\frk v_0$ to $\BB p(s (0))$ (if $\xi = -$) or from $\BB p(s (0))$ to $\frk v_0$ (if $\xi = +$), viewed as a half-edge on the boundary of the external face.

As explained in, e.g.,~\cite[Section~3.2]{curien-miermont-uihpq} and~\cite[Section~3.3]{bet-mier-disk}, this construction defines a bijection from $\mcl T^{\bullet}(n,l)$ to $\mcl Q^{ \bullet}(n,l)$.

\begin{remark} \label{remark-schaeffer-bdy}
It is explained in~\cite[Section~3.3.1]{curien-miermont-uihpq} that there is a canonical boundary path $\lambda : [1,2l]_{\BB Z} \rta \mcl E(\bdy Q)$ starting and ending at the terminal vertex $\BB v_0$ of the root edge $\BB e_0$ which traces all of the edges in $\bdy Q$ in cyclic order, defined as follows. Recall the definition of the times $j_k$ for $k\in\BB Z$ at which the walk $b^0$ has a downward step.  For each $k\in [1,l]_{\BB Z}$, there is a unique connected component of the complement of $Q$ in the inner face of the map $F$ which contains the edge from $\frk v_k$ to $\frk v_{k+1}$ (or from $\frk v_l$ to $\frk v_1$ if $k=l$) on its boundary.  It is easy to see from the Schaeffer bijection that there are precisely $j_{k+1} - j_k +1$ vertices and $j_{k+1}-j_k$ edges of $\bdy Q $ on the boundary of this component, counted with multiplicity. The ordered sequence of labels of the vertices coincides with the ordered sequence of values of $b^0(j)$ for $j\in [j_k  , j_{k+1}]_{\BB Z}$ (which by definition of $j_k$ is the same as $b(j_k) , b(j_k)-1 , b(j_k) , b(j_k+1), \ldots , b(j_k) + j_{k+1}-j_k-2 =  b(j_{k+1})$).  For $j\in [j_k , j_{k+1}-1]_{\BB Z}$, we let $\lambda(j)$ be the $j$th edge along the boundary of this component, in order started from $\frk v_k$ and counted with multiplicity.  Then $\lambda$ is a bijection from $[1,2l]_{\BB Z}$ to the set of edges of $\bdy Q$ if we count the latter according to their multiplicity in the external face.  
\end{remark}

As in the case of the ordinary Schaeffer bijection, the above construction can also be phrased in terms of walks.
For $i\in [0,2n+l]_{\BB Z}$, let $k_i\in [0,l-1]_{\BB Z}$ be chosen so that the vertex $\BB p(i)$ belongs to the tree $\frk t_{k_i}$ and let
\eqb \label{eqn-quad-bdy-path}
C(i) := \op{dist}\left( \BB p(i) , \frk v_{k_i} ; \frk t_{k_i} \right) - k_i  , \quad \forall i \in [0,2n+l-1]_{\BB Z}  \quad\text{and}\quad C(2n+l)  = - l.
\eqe  
Then $C$ is the concatenation of the contour functions of the trees $\frk t_{k}$, but with an extra downward step whenever it moves to a new tree.  Let
\eqb \label{eqn-quad-bdy-inf}
I(k) := \min\left\{ i \in [0,2n+l]_{\BB Z} : C(i)= -k \right\}  ,\quad \forall  k\in [0 ,l ]_{\BB Z}.
\eqe  
Then for $k\in [0,l-1]_{\BB Z}$, $I(k)$ is the first time $i$ for which $\BB p(i) \in \frk t_{k}$ and the range of $I$ is precisely the set of vertices lying on the outer boundary of the graph $F$.
Also let $L^0(i) := L(i) - b(k_i)$.
To describe the law of the pair $(C  , L^0)$ we need the following definition.
  
\begin{defn} \label{def-discrete-snake}
Let $[a,b]_{\BB Z}$ be a (possibly infinite) discrete interval and let $S : [a,b]_{\BB Z} \rta \BB Z$ be a (deterministic or random) path with $S(i) - S(i-1) \in \{-1,0,1\}$ for each $i\in [a+1,b]_{\BB Z}$. The \emph{head of the discrete snake} driven by $S$ is the function $H : [a,b]_{\BB Z} \rta \BB Z$ whose conditional law given $S$ is described as follows. We set $H(a) = 0$. Inductively, suppose $ i \in [a+1,b]_{\BB Z}$ and $H(i)$ has been defined for $j\in [a,i-1]_{\BB Z}$. If $S(i) - S(i-1) \in \{-1,0\}$, let $i'$ be the largest $j \in [a,i-1]_{\BB Z}$ for which $H(i) = H(i')$; or $i' = -\infty$. If $i' \not=-\infty$, we set $H(i) = H(i')$. Otherwise, we sample $H(i)- H(i-1)$ uniformly from $\{-1,0,1\}$.  
\end{defn}

The following lemma is immediate from the definitions and the fact that the above construction is a bijection.

\begin{lem} \label{prop-quad-bdy-law}
If we sample $(Q , \BB e_0, \BB v_*)$ uniformly from $\mcl Q^{ \bullet}(n,l)$, then the law of $C$ is that of a simple random walk started from 0 and conditioned to reach $-l$ for the first time at time $2n+l$. The process $L^0$ is the head of the discrete snake driven by $i \mapsto C(i)- \min_{j\in [1,i]_{\BB Z}} C(j)$. The pair $(C , L^0)$ is independent from $b^0$.  
\end{lem} 
  
\subsection{Encoding the UIHPQ} 
\label{sec-uihpq}
  
In this subsection we describe an infinite-volume analog of the bijection of Section~\ref{sec-quad-bdy} which encodes a UIHPQ which is alluded to but not described explicitly in~\cite[Section~6.1]{curien-miermont-uihpq} (see also~\cite{caraceni-curien-uihpq} for a different encoding).  See Figure~\ref{fig-schaeffer-uihpq} for an illustration.

\begin{figure}[ht!]
 \begin{center}
\includegraphics{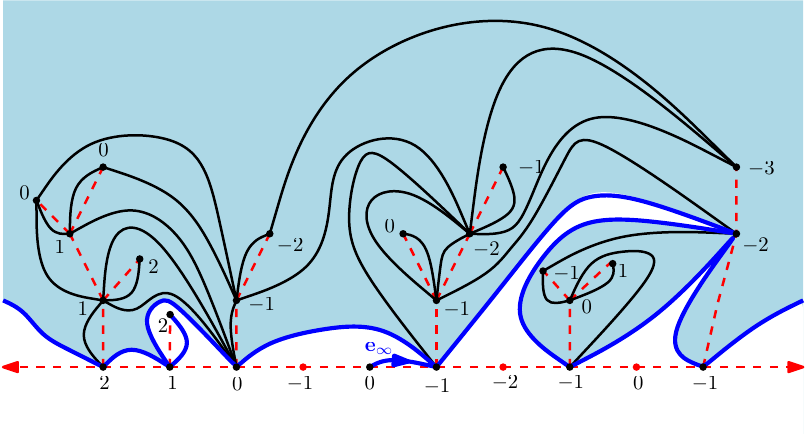} 
\caption{Illustration of the Schaeffer-type encoding of the UIHPQ. The graph $F_\infty$ containing the trees~$\frk t_{\infty,k}$ is the union of the dashed red lines and the black vertices. The red vertices correspond to upward steps of~$b_\infty^0$, and do not belong to $F_\infty$ or $Q_\infty$. Each vertex is labelled with either its label $L_\infty(v)$ or the corresponding value of $b_\infty^0$. The black edges are part of the interior of the UIHPQ, and the blue edges are part of its boundary. Note that $\bdy Q_\infty$ is not a simple path. The oriented boundary root edge $\BB e_\infty$ is indicated with an arrow. }\label{fig-schaeffer-uihpq}
\end{center}
\end{figure}

Let $b_\infty^0 : \BB Z\rta \BB N_0$ be a two-sided simple random walk reflected at $0$ (with increments uniform in $\{-1,1\}$). 
Let $\{j_k\}_{k\in\BB Z}$ be the ordered set of times for which $b_\infty^0(j+1) - b_\infty^0(j) = -1$, shifted so that $j_1$ is the smallest $j\geq 0$ for which $b_\infty^0(j+1) -  b_\infty^0(j) = -1$. Let $b_\infty(k) := b_\infty^0(j_k)$. 

Conditional on $b_\infty$, let $\{(\frk t_{ \infty , k} , \frk v_{\infty , k} , L_{\infty , k} )\}_{k\in\BB Z}$ be a bi-infinite sequence of independent triples where each $(\frk t_{\infty,k} , \frk v_{\infty,k})$ is a rooted Galton-Watson tree whose offspring distribution is geometric with parameter $1/2$ and, conditional on $\frk t_{\infty,k}$, each $L_{\infty,k}$ is uniformly distributed on the set of label functions $\mcl V(\frk t_{\infty,k}) \rta \BB Z$ which satisfy $L_{\infty,k}(\frk v_{\infty,k}) = b_\infty(k)$ and $|L_{\infty,k}(u)-L_{\infty,k}(v)| \leq 1$ whenever $u, v\in\mcl V(\frk t_{\infty,k})$ are connected by an edge.

We define a graph $F_\infty$ as follows. 
Equip $\BB Z$ with the standard nearest-neighbor graph structure and embed it as the real line in $\BB C$. For $k \in \BB Z$, embed the tree $\frk t_{\infty,k}$ into the upper half-plane in such a way that the vertex $\frk v_{\infty, k }$ is identified with $k \in \BB Z$ and none of the trees $\frk t_{\infty,k}$ intersect each other or intersect $\BB R$ except at their root vertices. The graph $F_\infty^0$ is the union of $\BB Z$ and the trees $\frk t_{\infty,k}$ for $k\in\BB Z$ with this graph structure, with each integer $j_k$ identified with the corresponding root vertex $\frk v_{\infty,k}$. 
 
 We define a label function $L_\infty$ on the vertices of $F_\infty$ by setting $L_\infty|_{\mcl V(\frk t_{\infty,k})} = L_{\infty,k}$ for each $k\in\BB Z$. We let $\BB p_\infty : \BB Z\rta \mcl V(F_\infty)$ be the contour exploration of the upper face of $F_\infty$, shifted so that $\BB p_\infty$ starts exploring the tree $\frk t_{\infty,1}$ at time 0. We then define the successor $s_\infty(i)$ of each time $i\in\BB Z$ exactly as in the Schaeffer bijection, except that there is no need to add an extra vertex since a.s.\ $\liminf_{i\rta\infty} L_\infty(i) = -\infty$. We then draw an edge connecting each vertex $\BB p_\infty(i)$ to $\BB p_\infty(s_\infty(i))$ for each $i \in \BB Z$ to obtain an infinite quadrangulation with boundary $Q_\infty$, which we take to be rooted at the oriented edge $\BB e_{\infty}$ which goes from $\frk v_{\infty,0}$ to $\BB p_\infty(s_\infty(0))$. Then $(Q_\infty, \BB e_{\infty })$ is an instance of the UIHPQ. Furthermore, our construction of $F_\infty$ gives rise to an embedding of $Q_\infty$ into $\BB H$. 
 
We note that the obvious analog of Remark~\ref{remark-schaeffer-bdy} holds in this setting. 
  
\begin{remark} \label{remark-schaeffer-bdy-uihpq}
There is a canonical choice of boundary path $\lambda_\infty : \BB Z \rta \mcl E(\bdy Q)$ which hits the terminal vertex $\BB v_\infty$ of the root edge $\BB e_\infty$ at time $0$ and which traces all of the edges in $\bdy Q$ in cyclic order, defined as follows. Recall the definition of the times $j_k$ for $k\in\BB Z$ at which the walk $b_\infty^0$ has a downward step.  
For each $k\in \BB Z$, there is a unique connected component of the complement of $Q_\infty$ in the upper face of the map $F_\infty$ which contains the edge from $\frk v_{\infty,k}$ to $\frk v_{\infty,k+1}$ on its boundary. 
There are precisely $j_{k+1} - j_k +1$ vertices and $j_{k+1}-j_k$ edges of $\bdy Q_\infty $ on the boundary of this component, counted with multiplicity.  
For $j\in [j_k , j_{k+1}-1]_{\BB Z}$, we let $\lambda_\infty(j)$ be the $j$th edge along the boundary of this component, in order started from $\frk v_{\infty,k}$ and counted with multiplicity. 
Then $\lambda_\infty$ is a bijection from $\BB Z$ to the set of edges of $\bdy Q_\infty$ if we count the latter according to their multiplicity in the external face.  
\end{remark}

As in Section~\ref{sec-quad-bdy}, we now define random paths which encode $(Q_\infty , \BB e_\infty)$. 
For $i\in \BB Z$, let $k_i$ be chosen so that the vertex $\BB p_\infty(i)$ belongs to the tree $\frk t_{\infty,k_i}$ and let
\eqb \label{eqn-quad-bdy-path-infty}
C_\infty(i) := \op{dist}\left( \BB p_\infty(i) , \frk v_{\infty,k_i} ; \frk t_{\infty,k_i} \right) - k_i    , \quad \forall i \in \BB Z.
\eqe   
Let
\eqb \label{eqn-quad-bdy-inf-infty}
I_\infty(k) := \min\left\{ i \in \BB Z : C_\infty(i)= -k \right\}  ,\quad \forall  k\in \BB Z.
\eqe  
Then $I_\infty(k)$ is the first time $i$ for which $\BB p_\infty(i) \in \frk t_{\infty, k }$ and the range of $I_\infty$ is the set of vertices lying on the outer boundary of the graph $F_\infty$. 
We let $L_\infty(i) := L_\infty(\BB p_\infty(i))  $ and $L_\infty^0(i) := L_\infty(i) - b_\infty(k_i)$.

\begin{lem} \label{prop-uihpq-encode-law}
The pair $(C_\infty ,L_\infty^0)$ is independent from $b_\infty^0$ and its law can be described as follows. 
The law of $C_\infty|_{\BB N_0}$ is that of a simple random walk started from 0 and the law of $C_\infty(-\cdot)|_{\BB N_0}  $ is that of a simple random walk started from 0 and conditioned to stay positive for all time (see, e.g.,~\cite{bertoin-doney-nonnegative} for a definition of this conditioning for a large class of random walks).
Furthermore, $L_\infty^0$ is the head of the discrete snake driven by $i \mapsto C_\infty(i) - \min_{j\in (-\infty,i]_{\BB Z}} C_\infty(j)$ (Definition~\ref{def-discrete-snake}). 
\end{lem}
\begin{proof}
The process $C_\infty$ is the concatenation of the contour functions of countably many i.i.d.\  Galton-Watson trees with offspring distribution given by a geometric random variable with parameter $1/2$, separated by downward steps. Each of these contour functions has the law of a simple random walk run until the first time it hits $0$. Therefore, $ C_\infty$ has the law described in the statement of the lemma. Since each $L_\infty^0|_{[I(k) , I(k+1)-1]_{\BB Z}}$ is the head of the discrete snake driven by $C_\infty|_{[I(k) , I(k+1)-1]_{\BB Z}}$, we see that $L_\infty^0$ is the  head of the discrete snake driven by $C_\infty$. The pair $(C_\infty, L_\infty^0)$ is determined by the labeled trees $\{(\frk t_{ \infty , k} , \frk v_{\infty , k} , L_{\infty , k} )\}_{k\in\BB Z}$, so is independent from $b_\infty^0$.
\end{proof}

\subsection{Distance bounds for quadrangulations with boundary} 
\label{sec-uihpq-dist}

In this subsection we record elementary upper and lower bounds for distances in quadrangulations with boundary in terms of the encoding processes in the Schaeffer bijection.  
We start with an upper bound. 
  
\begin{lem} \label{prop-schaeffer-dist}
Suppose we are in the setting of Section~\ref{sec-quad-bdy}.
In particular, let $(Q , \BB e_0 , \BB v_*) \in \mcl Q^\bullet(n,l)$, let $(C , L , b)$ be its Schaeffer encoding process, and let $\BB p : [0,2n+l]_{\BB Z} \rta \mcl V(Q)$ be the projection map.  
For $i_1,i_2\in \BB Z$, we have
\eqbn
\op{dist}\left(\BB p (i_1) ,\BB p (i_2) ; Q \right) \leq   L (i_1)  + L (i_2) - 2\min_{j \in [i_1\wedge i_2 , i_1\vee i_2]_{\BB Z} } L (j)  +2
\eqen
\end{lem}
\begin{proof}
In the finite-volume case, this follows from~\cite[Lemma 3]{miermont-tess} (which gives the analogous estimate in a more general setting).
The infinite-volume version follows from exactly the same argument.
See also~\cite[Lemma 3.1]{legall-topological} (which is the analogous estimate for quadrangulations without boundary, and is proven in the same manner)
and/or~\cite[Equation (3)]{bet-disk-tight} (which states the precise estimate given in the present lemma in the finite-volume case).
\end{proof}

We also have a lower bound for distances, which is a variant of the so-called cactus bound for the Brownian map (see, e.g.,~\cite[Proposition 5.9]{legall-miermont-notes}). 

\begin{lem}  \label{prop-cactus}
Suppose we are in the setting of Section~\ref{sec-quad-bdy}.
In particular, let $(Q , \BB e_0 , \BB v_*) \in \mcl Q^\bullet(n,l)$, let $(C , L  )$ be its Schaeffer encoding process, let $I(m)$ for $m\in [0,l]_{\BB Z}$ be as in~\eqref{eqn-quad-bdy-inf}, and let $\BB p : [0,2n+l]_{\BB Z} \rta \mcl V(Q)\setminus \{\BB v_*\} $ be the contour exploration.   
For $i_1,i_2\in [1,2n+l^n]_{\BB Z}$ with $i_1< i_2$, let 
\eqb
J(i_1,i_2) := \left( I\left([0,l]_{\BB Z} \right) \cap [i_1,i_2]_{\BB Z} \right) \cup \{i_1,i_2\} \quad \op{and} \quad 
J'(i_1,i_2) :=   \left( I\left([0,l]_{\BB Z} \right) \setminus [i_1,i_2]_{\BB Z} \right) \cup \{i_1,i_2\} ,
\eqe 
so that $\BB p(J(i_1,i_2))$ (resp.\ $\BB p(J'(i_1,i_2))$) consists of $\BB p(i_1)$, $\BB p(i_2)$, and the set of vertices of $\bdy Q$ which are contained in the image of $\BB p\circ I$ and which are (resp.\ are not) contained in $\BB p([i_1,i_2]_{\BB Z})  $
Then  
\eqb \label{eqn-cactus}
\op{dist}\left(\BB p (i_1) ,\BB p (i_2) ; Q \right) \geq L (i_1)  + L (i_2) - 2 \max\left\{ \min_{j \in  J (i_1,i_2)  } L(j) ,   \min_{j  \in J '(i_1,i_2)} L(j) \right\}   .
\eqe 
The analogous estimate also holds in the setting of the UIHPQ (but in this case the minimum over $J'(i_1,i_2)$ in~\eqref{eqn-cactus} is a.s.\ equal to $-\infty$, so only the first term in the maximum is present). 
\end{lem} 
\begin{proof}
This follows from essentially the same proof as the ordinary cactus bound for quadrangulations without boundary (see~\cite[Proposition 5.9(ii)]{legall-miermont-notes}) but in the finite-volume case one has to consider two paths in the graph $F$ associated with the treed bridge from Section~\ref{sec-quad-bdy} since $F$ has a single cycle. 
See also the proof of~\cite[Theorem 5]{bet-disk-tight} for a lower bound which immediately implies the one in the statement of the lemma in the finite-volume case. 
\end{proof}

\subsection{Pruning the UIHPQ to get the UIHPQ$_{\op{S}}$}
\label{sec-pruning}

\begin{figure}[ht!]
 \begin{center}
\includegraphics{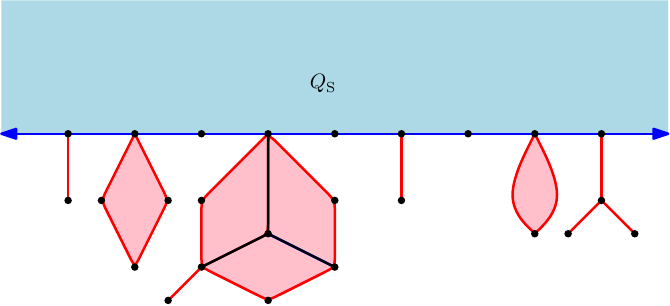} 
\caption{Pruning the red dangling quadrangulations from the UIHPQ $Q_\infty$ (light blue and pink regions with blue and red boundary) produces a UIHPQ$_{\op{S}}$ $Q_{\op{S}}$ (light blue region with blue boundary). }\label{fig-pruning}
\end{center}
\end{figure}

Recall from Section~\ref{sec-quad-def} that the UIHPQ$_{\op{S}}$ is the Benjamini-Schramm limit of uniformly random quadrangulations with simple boundary, as viewed from a uniformly random vertex on the boundary, as the area and then the perimeter tend to $\infty$. In this subsection we explain how to \emph{prune} an instance of the UIHPQ to obtain an instance of the UIHPQ$_{\op{S}}$. 

Suppose $(Q_\infty , \BB e_\infty)$ is a UIHPQ. There are infinitely many vertices $v \in \mcl V(\bdy Q_\infty)$ which have multiplicity at least $2$ in the external face, so are hit twice by the boundary path $\lambda_\infty : \BB Z\rta \mcl E(\bdy Q_\infty)$ of Remark~\ref{remark-schaeffer-bdy-uihpq}.  Attached to each of these vertices is a finite dangling quadrangulation which is disconnected from $\infty$ in $Q_\infty$ by removing a single boundary vertex.

Let $Q_{\op{S}}$ be the largest subgraph of $Q_\infty$ with the property that none of its vertices or edges can be disconnected from $\infty$ in $Q_\infty$ by removing a single boundary vertex. In other words, $Q_{\op{S}}$ is obtained by removing all of the ``dangling quadrangulations" of $Q_\infty$ which are joined to $\infty$ by a single vertex. Also let $ \BB e_{\op{S}}$ be the edge immediately to the left of the vertex which can be removed to disconnect $\BB e_\infty$ from $\infty$ (if such a vertex exists) or let $\BB e_{\op{S}} = \BB e_\infty$ if $\BB e_\infty$ belongs to $\bdy Q_{\op{S}}$. Then $(Q_{\op{S}} , \BB e_{\op{S}})$ is an instance of the UIHPQ$_{\op{S}}$. 

One can recover a canonical boundary path $\lambda_{\op{S}} : \BB Z\rta\mcl E(\bdy Q_{\op{S}})$ which traces the edges of $\bdy Q_\infty$ and hits the terminal endpoint of $\BB e_{\op{S}}$ at time $0$ from the analogous boundary path $\lambda_\infty$ of the UIHPQ by skipping all of the intervals of time during which $\lambda_{\op{S}}$ is tracing a dangling quadrangulation. 

It is shown in~\cite[Section~6.1.2]{curien-miermont-uihpq} and explained more explicitly in~\cite[Section~6]{caraceni-curien-uihpq} that if we start with an instance $(Q_{\op{S}} , \BB e_{\op{S}})$ of the UIHPQ$_{\op{S}}$, we can construct a UIHPQ $(Q_\infty , \BB e_\infty)$ which can be pruned as above to recover $(Q_{\op{S}} , \BB e_{\op{S}})$ via an explicit sampling procedure. 
Conditional on $Q_{\op{S}}$, let $\{(q_v , e_v)\}_{v \in \mcl V(\bdy Q_{\op{S}})}$ be an independent sequence of random finite quadrangulations with general boundary with an oriented boundary root edge, with distributions described as follows.
Let $v_0$ be the right endpoint of the root edge $\BB e_{\op{S}}$. 
Each $q_v$ for $v\not=v_0$ is distributed according to the so-called \emph{free Boltzmann distribution on quadrangulations with general boundary}, which is given by
\eqb \label{eqn-free-boltzman}
\BB P\left[ (q_v,e_v) = (\frk q, \frk e) \right] = C^{-1} \left( \frac{1}{12} \right)^n \left( \frac18 \right)^l
\eqe 
for any quadrangulation $\frk q$ with $n$ interior faces and $2l$ boundary edges (counted with multiplicity) with a distinguished oriented root edge $\frk e\in \bdy \frk q$, where here $C >0$ is a normalizing constant. 
The quadrangulation $q_{v_0}$ is instead distributed according to
\eqb \label{eqn-free-boltzman-root}
\BB P\left[ (q_{v_0} , e_{v_0}) = (\frk q , \frk e) \right] = \wt C^{-1} (2l+1) \left( \frac{1}{12} \right)^n \left( \frac18 \right)^l  
\eqe 
for a different normalizing constant $\wt C$. (Intuitively, the reason for the extra factor $2l+1$ in~\eqref{eqn-free-boltzman-root} is that a dangling quadrilateral with longer boundary length is more likely to contain the root edge.)  

For each $v \in \mcl V(\bdy Q_{\op{S}})$, identify the terminal endpoint of $e_v$ with $v$. This gives us a new quadrangulation $Q_\infty$ with general boundary. We choose an oriented root edge $\BB e_\infty$ for $Q_\infty$ by uniformly sampling one of the oriented edges of $\mcl E(\bdy q_{v_0}) \cup \{\BB e_{\op{S}} \}$. Then $(Q_\infty , \BB e_\infty)$ is a UIHPQ with general boundary which can be pruned to recover $(Q_{\op{S}} , \BB e_{\op{S}})$.

\subsection{The Brownian disk}
\label{sec-brownian-disk}

In this subsection we will review the definition of the Brownian disk from~\cite{bet-mier-disk}, which is the scaling limit of finite uniform  quadrangulations with boundary~\cite{bet-mier-disk}.  
The construction is a finite-volume analog of the construction of the Brownian half-plane in Section~\ref{sec-bhp} and a continuum analog of the Schaeffer-type construction in Section~\ref{sec-quad-bdy}. 
We follow closely the exposition given in~\cite[Section~3.1]{gwynne-miller-gluing}. 
 
Fix an area parameter $a>0$ and a boundary length parameter $\ell>0$. Let $X : [0,a] \rta [0,\infty)$ be a standard Brownian motion started from $\ell$ and conditioned to hit $0$ for the first time at time $a$ (such a Brownian motion is defined rigorously in, e.g.,~\cite[Section~2.1]{bet-mier-disk}). 
For $s,t\in [0,1]$, set
\eqb \label{eqn-X-dist}
d_X(s,t) := X(s) + X(t) - 2\inf_{u \in [s\wedge t , s\vee t]} X(u) .
\eqe  
Conditioned on $X$, let $Z^0$ be the centered Gaussian process with
\eqb \label{eqn-Z^0-def}
\op{Cov}(Z^0(s) ,Z^0(t) ) = \inf_{u\in [s\wedge t , s\vee t]} \left( X(u) - \inf_{v \in [0,u]} X(v) \right) , \quad s,t\in [0,a] .
\eqe
Using the Kolmogorov continuity criterion, one can check that $Z^0$ a.s.\ admits a continuous modification which is $\alpha$-H\"older continuous for each $\alpha <1/4$. For this modification we have $Z_s^0 = Z_t^0$ whenever $d_X(s,t) = 0$.
 
Let $\frk b$ be $\sqrt 3$ times a Brownian bridge from $0$ to $0$ independent from $(X,Z)$ with time duration $\ell$. 
For $r \in [0,\ell]$, let 
\eqb \label{eqn-disk-inf}
T(r) := \inf\left\{t \geq 0 \,:\, X_t = \ell-r \right\}
\eqe
and for $t \in [0,a]$, let $T^{-1}(t) := \sup\left\{r \in [0,\ell] \,:\, T(r) \leq t \right\}$. Set
\eqbn
Z(t) := Z^0(t) +  \frk b(T^{-1}(t)) .
\eqen
We view $[0,a]$ as a circle by identifying~$0$ with~$a$ and for $s,t\in [0,a]$ we define $\ul Z(s,t)$ to be the minimal value of~$Z$ on the counterclockwise arc of $[0,a]$ from $s$ to $t$. 
For $s,t \in [0,a]$, define
\eqb \label{eqn-d_Z}
d_Z\!\left(s,t \right) = Z(s) + Z(t) - 2\left( \ul Z(s,t) \vee \ul Z(t,s) \right)  
\eqe 
and
\eqb \label{eqn-dist0-def}
d^0(s,t) = \inf \sum_{i=1}^k d_Z(s_i , t_i)
\eqe 
where the infimum is over all $k\in\BB N$ and all $2k+2$-tuples $( t_0, s_1 , t_1 , \ldots , s_{k } , t_{k } , s_{k+1}) \in [0,a]^{2k+2}$ with $t_0 = s$, $s_{k+1} = t$, and $d_X(t_{i-1} , s_i) = 0$ for each $i\in [1,k+1]_{\BB Z}$. Equivalently, $d^0$ is the largest pseudometric on $[0,a]$ which is at most $d_Z$ and is zero whenever $d_X$ is $0$. 

The \emph{Brownian disk} with area $a$ and perimeter $\ell$ is the quotient space $ H  = [0,a]/\{d^0 = 0\}$ equipped with the quotient metric, which we call $d $. It is shown in~\cite{bet-mier-disk} that $(H,d) $ is a.s.\ homeomorphic to the closed disk. 

Let $\BB p : [0,a] \rta H$ for the quotient map.
The \emph{area measure} of $H$ is the pushforward $\mu$ of Lebesgue measure on $[0,a]$ under $\BB p$. 
The \emph{boundary} of $H$ is the set $\bdy H = \BB p\!\left(\{T_r \,:\, r \in [0,\ell] \} \right)$ (this set is the topological boundary of $H$ by~\cite[Proposition 21]{bet-disk-tight}, and is homeomorphic to the circle). 
We note that $\bdy H$ has a natural orientation, obtained by declaring that the path $t\mapsto \BB p(t)$ traces $\bdy H$ in the counterclockwise direction. 
The \emph{boundary measure} of $H$ is the pushforward $\nu$ of Lebesgue measure on $[0,\ell]$ under $r\mapsto \BB p(T(r))$.
The \emph{boundary path} of $H$ is the curve $\eta : [0,\ell] \rta \bdy H$ defined by $\eta(r) = \BB p(T(r))$. 

By~\cite[Corollary~1.5]{lqg-tbm2}, the law of the metric measure space $(H,d,\mu,\nu)$ is the same as that of the $\sqrt{8/3}$-LQG disk with area $a$ and boundary length $\ell$, equipped with its $\sqrt{8/3}$-LQG area measure and boundary length measure.

\section{Scaling limit of the UIHPQ and UIHPQ$_{\op{S}}$}
\label{sec-uihpq-conv}

In this section we will prove Theorems~\ref{thm-uihpq-ghpu} and~\ref{thm-uihpqS-ghpu}.  We will extract Theorem~\ref{thm-uihpq-ghpu} from \cite{bet-mier-disk} in a manner which is similar to that in which the scaling limit result for the uniform infinite planar quadrangulation without boundary (UIPQ) in~\cite{curien-legall-plane} was extracted from \cite{miermont-brownian-map,legall-uniqueness}.  We start in Section~\ref{sec-disk-ghpu} by showing that one can improve the Gromov-Hausdorff scaling limit result for finite uniformly random quadrangulations with boundary toward the Brownian disk~\cite{bet-mier-disk} to get convergence in the GHPU topology.  In Section~\ref{sec-bhp-coupling}, we will show that one can couple an instance of the Brownian half-plane with a Brownian disk in such a way that metric balls of a certain radius centered at the root point coincide with high probability.  In Section~\ref{sec-map-coupling}, we prove an analogous coupling result for the UIHPQ with a finite uniformly random quadrangulation with boundary.  In Section~\ref{sec-uihpq-proof}, we will deduce Theorem~\ref{thm-uihpq-ghpu} from these coupling results and the scaling limit result of Section~\ref{sec-disk-ghpu}.

In Section~\ref{sec-uihpqS-proof}, we will deduce Theorem~\ref{thm-uihpqS-ghpu} from Theorem~\ref{thm-uihpq-ghpu} using the pruning procedure discussed in Section~\ref{sec-pruning}.

\subsection{Convergence to the Brownian disk in the GHPU topology}
\label{sec-disk-ghpu}

It is proven in~\cite{bet-mier-disk} that uniformly random quadrangulations with boundary converge in the scaling limit to the Brownian disk in the Gromov-Hausdorff topology, and it is not hard to see from the proof in~\cite{bet-mier-disk} that one has convergence in the stronger GHPU topology as well. 
We will explain why this is the case just below. 

Fix $\ell > 0$ and let $(H , d)$ be a unit area Brownian disk with boundary length $\ell$.  Let $\mu$ (resp.\ $\eta$) be the natural area measure (resp.\ boundary path) on $H$, as in Section~\ref{sec-brownian-disk}. 
Let
\eqb \label{eqn-disk-ghpu}
\frk H := \left(H,d,\mu,\eta   \right) 
\eqe
so that $\frk H$ is an element of $ \BB M^{\op{GHPU}}$.

Let $\{l^n\}_{n\in\BB N}$ be a sequence of positive integers with $(2n)^{-1/2} l^n \rta \ell$. For $n\in\BB N$, let $(Q^n , \BB e_0^n , \BB v_*^n )$ be sampled uniformly from $\mcl Q^\bullet (n, l^n)$ (Section~\ref{sec-quad-def}).  We view $Q^n$ as a connected metric space in the manner of Remark~\ref{remark-ghpu-graph}.  Let $d^n$ be the graph distance on $Q^n$, re-scaled by $(9/8)^{1/4} n^{-1/4}$.  Let $\mu^n$ be the measure on $\mcl V(Q^n)$ which assigns a mass to each vertex equal to $(4n)^{-1}$ times its degree.  Let $\lambda^n : [0, 2l^n] \rta \bdy Q^n$ be the boundary path of $Q^n$ started from $\BB e_0^n$, extended by linear interpolation.  Let $\eta^n(t) := \lambda^n\left(2^{3/2}  n^{ 1/2} t \right)$ for $t\in [0, 2^{-1/2} n^{-1/2} l^n ]$. 
Let
\eqb \label{eqn-uniform-quad-ghpu}
\frk Q^n := \left( Q^n , d^n , \mu^n ,\eta^n \right) .
\eqe
 
\begin{thm} \label{thm-disk-ghpu}
In the setting described just above, we have $\frk Q^n \rta \frk H$ in law in the GHPU topology. 
\end{thm} 
\begin{proof}
We will deduce the theorem statement from the scaling limit result in~\cite{bet-mier-disk} together with Lemma~\ref{prop-ghpu-condition}.  The key point is that the encoding processes for $Q^n$ from Section~\ref{sec-quad-bdy} converge jointly with the metric spaces $(Q^n , d^n)$ to the encoding processes for $(H,d)$ in Section~\ref{sec-brownian-disk} and the metric space $(H,d)$, in the uniform topology and the Gromov-Hausdorff topology, respectively; and the measures and curves defined above are determined by the encoding processes in a relatively simple way.  This allows us to check the conditions of Lemma~\ref{prop-ghpu-condition}. 
 
Let $(X,Z , \frk b)$ be the encoding process for $(H,d )$ and for $r \in [0,\ell]$ let $T(r)$ be the first time that $X(t) = \ell-r$, as in~\eqref{eqn-disk-inf}.  
Let $\BB p : [0,1]\rta H$ be the quotient map and for $s,t\in [0,1]$ write 
\eqbn
\wt d(s,t) := d (\BB p(s) , \BB p(t)) 
\eqen
so that $\wt d$ is a pseudometric on $[0,1]$. 
Recall that $\mu$ is the pushforward of Lebesgue measure on $[0,a]$ under $\BB p$ and $\eta(r) = \BB p(T(r))$ for $r\in [0,\ell]$. 
Let $\nu$ be the pushforward of Lebesgue measure on $[0,\ell]$ under $\eta$, i.e.\ the boundary length measure of $H$. 
 
For $n\in\BB N$, let $( C^n , L^n , b^{0,n})$ be the Schaeffer encoding triple of $(Q^n , \BB e_0^n , \BB v_*^n )$ as in Section~\ref{sec-quad-bdy}. For $m\in [1, l^n]_{\BB Z}$, define $I^n(m) := \min\left\{ i \in [1,2n+l^n]_{\BB Z} \,:\,  C^n(i)  = -m \right\} $,
as in~\eqref{eqn-quad-bdy-inf}. 
Extend $C^n$ and $L^n$ to $[0,2n]$ by linear interpolation and for $t\in [0,1]$, define
\eqbn
X^n(t) := (2n)^{-1/2} \left(  C^n((2n + l^n) t) +  l^n \right)   \quad \op{and} \quad Z^n(t) := \left(\frac{9}{8n}\right)^{1/4} L^n((2n + l^n) t) .
\eqen
For $r\in [0, (2n)^{-1/2} l^n]$, also define (in analogy with~\eqref{eqn-disk-inf})
\eqb \label{eqn-quad-hit}
T^n(r) := (2n)^{-1} I^n(\lfloor (2n)^{1/2} r \rfloor) . 
\eqe 
Let $\BB p^n : [0,2n+l^n]_{\BB Z} \rta \mcl V(Q^n)$ be the contour exploration as in Section~\ref{sec-quad-bdy}. For $s,t\in [0,1]$ with $(2n+l^n)s , (2n+l^n) t \in \BB N_0$, let 
\eqbn
\wt d^n(s,t) := d^n\left( \BB p^n((2n+l^n)s ) , \BB p^n( (2n+l^n) t )   \right) 
\eqen
and extend $\wt d^n$ to $[0,1]^2$ by linear interpolation. 

In the variant of the Schaeffer bijection described in Section~\ref{sec-quad-bdy}, we add one edge from the vertex $\BB p^n(i)$ to the successor vertex (which has a smaller label) for each $i\in [0,2n+l^n]_{\BB Z}$. Consequently, if we let $\wh\mu^n$ be the pushforward under $\BB p^n$ of $(2n)^{-1}$ times counting measure on $[0,2n+l^n]_{\BB Z}$, then for $v\in \mcl V(Q^n)$ it holds that $\wh\mu^n(v)$ equals $(2n)^{-1}$ times the number of edges of $Q^n$ which connect $v$ to a vertex with a smaller label. Since $\mu^n$ assigns mass to each vertex equal to $(4n)^{-1}$ times its degree, the $d^n$-Prokhorov distance between $\wh\mu^n$ and $\mu^n$ is at most a universal constant times $n^{-1/4}$ (note that the scaling factors for $\wh\mu^n$ and $\mu^n$ differ by a factor of 2 since each edge is counted twice---once for each of its endpoints---when considering the measure $\mu^n$). 
 
Let $\nu^n$ be the pushforward of $2^{-3/2} n^{-1/2}$ times counting measure on $[1,2l^n]_{\BB Z}$ under the (linearly interpolated) boundary path $\lambda^n$. Equivalently $\nu^n$ is the counting measure on $\mcl V(\bdy Q^n)$ (with vertices counted with multiplicity), rescaled by $2^{-3/2} n^{-1/2}$.  The measure $\nu^n$ does not admit a simple description in terms of $(C^n , L^n , b^{0,n})$ but we can describe a closely related measure as follows. 

Let $\wh\nu^n$ be the pushforward under $\BB p^n\circ I^n$ of $2^{-1/2} n^{-1/2}$ times the counting measure on $[1,l^n]_{\BB Z}$ (with $I^n$ as above). Note that the scaling factor here is off by a factor of 2 as compared to the scaling factor of $\nu^n$ so that $\wh\nu^n$ has the same total mass as $\nu^n$.
The measure $\wh\nu^n$ can equivalently be described as follows.  Let $j_0^n = 0$ and for $k\in [1,l^n]_{\BB Z}$, let $j_k^n$ be the $k$th smallest downward step of the random walk bridge $b^{0,n}$.  Then $\wh\nu^n$ is the measure on $\mcl V(\bdy Q^n)$ which assigns mass $2^{-1/2} n^{-1/2}$ to the $j_k^n$th vertex of $\bdy Q^n$ in counterclockwise cyclic order started from the root vertex.  Since $b^{0,n}$ is a simple random walk bridge, we can use Hoeffding's concentration inequality for binomial random variables to find that except on an event of probability decaying faster than any power of $n$, 
\eqb \label{eqn-max-bridge-excursion}
\max_{k\in [1,l^n]_{\BB Z}}| j_k^n - 2k| = n^{o_n(1)}  
\eqe  
which implies that $\lambda^n(2k)$ differs from $\BB p^n( I^n(k))$ for $k\in\BB N$ by at most $n^{o_n(1)}$ units of boundary length and 
\eqbn
|\nu^n(A) - \wh\nu^n(A)| \leq n^{-1/2 + o_n(1)} ,\quad \forall A\subset \bdy Q^n. 
\eqen
 
It is shown in~\cite[Section~5]{bet-mier-disk} that, in the notation above,  
\eqb \label{eqn-disk-ghpu-gp}
(X^n , Z^n , \wt d^n) \rta (X,Z, \wt d) \quad \op{and}\quad \left(Q^n , d_{Q^n}  , \eta^n(0) \right) \rta \left( H , d  , \eta(0) \right) 
\eqe 
in law in the uniform topology and the pointed GH topology, respectively. With $T$ and $T^n$ the hitting time processes as in~\eqref{eqn-disk-inf} and~\eqref{eqn-quad-hit}, respectively, we have $(X^n , Z^n ,  T^n) \rta (X, Z , T)$ in law in the uniform topology in the first two coordinates and the Skorokhod topology in the third coordinate.  Since $(X , Z)$ a.s.\ determines $T$, $\wt d$, and $(H , d , \eta(0) )$, the convergence~\eqref{eqn-disk-ghpu-gp} in law occurs jointly with the convergence $T^n \rta T$ in law.  

By the Skorokhod representation theorem, we can couple $\{(Q^n , \BB e_0^n , \BB v_*^n)\}_{n\in\BB N}$ with $(\op{BD} , d , x )$ in such a way that the convergence~\eqref{eqn-disk-ghpu-gp} occurs a.s.\  and also $T^n \rta T$ a.s.\ in the Skorokhod topology. 
Henceforth fix a coupling for which this is the case. 
Note that the Borel-Cantelli lemma implies that in any such coupling,~\eqref{eqn-max-bridge-excursion} holds for large enough $n$. 
 
By~\cite[Lemma~3.2]{gwynne-miller-gluing}, re-phrased in terms of the pseudometric $\wt d$ on $[0,1]$, there a.s.\ exists $C>0$ such that for each $r_1 , r_2 \in [0,\ell]$ with $r_1 < r_2$,
\eqb \label{eqn-use-disk-holder}
\wt d\left(T(r_1) , T(r_2) \right) \leq C (r_2-r_1)^{1/2} \left( |\log (r_2-r_1)|+1 \right)^2 .
\eqe 
The Skorokhod convergence $T^n \rta T$ together with the uniform convergence $\wt d^n \rta \wt d$ therefore implies that for each $\ep > 0$, there a.s.\ exists $\delta>0$ such that for each $n\in\BB N$ and each $r_1 ,r_2\in [0,\ell]$ with $|r_1 -r_2| \leq \delta$, we have $\wt d^n(T^n(r_1) , T^n(r_2) ) \leq  \ep$. From this together with~\eqref{eqn-max-bridge-excursion} and the discussion just after, we infer that the re-scaled boundary paths $\{\eta^n\}_{n\in \BB N}$ are equicontinuous. 

We will now apply Lemma~\ref{prop-ghpu-condition} to deduce our desired GHPU convergence. 
The lemma does not apply directly in our setting since the curve $\eta$ is not simple (it satisfies $\eta(0) = \eta(\ell)$), so a straightforward truncation argument is needed. 
For $r \in [0, (2n)^{-1/2} l^n]$, let $\wh\nu^n_r$ be the restriction of $\wh\nu^n$ to the counterclockwise arc of $\bdy Q^n$ started from $\eta^n(0)$ with $\wh\nu^n$-length $r$.
Also let $\nu_r^n$ be the pushforward of Lebesgue measure on $[0, r ]$ under $\eta^n$. 
For $r\in [0,\ell]_{\BB Z}$, let $\nu_r$ be the pushforward of Lebesgue measure on $[0,r]$ under $\eta$.  

Our choice of coupling and our above description of the relationship between the measures $\mu^n$ and $\wh\mu^n$, 
\eqbn
\left( Q^n , d^n , \mu^n , \wh\nu^n_r , \eta^n(0) \right) \rta  \left( H , d   , \mu , \nu_r   , \eta(0) \right) 
\eqen
in the 2-fold pointed Gromov-Prokhorov topology for each $r \in (0,\ell)$. By~\eqref{eqn-max-bridge-excursion}, we also have this convergence with $\nu^n_r$ in place of $\wh\nu^n_r$. 
By Lemma~\ref{prop-ghpu-condition}, we infer that 
\eqbn
\left( Q^n , d^n , \mu^n , \eta^n|_{[0,r]} \right) \rta  \left( H , d   , \mu , \eta|_{[0,r]} \right) 
\eqen
in the GHPU topology. Since $r$ can be made arbitrarily close to $\ell$, by combining this with equicontinuity of the curves $\eta^n$ we see that in fact $\frk Q^n \rta \frk H$ in the GHPU topology.
\end{proof}

\subsection{Coupling the Brownian half-plane with the Brownian disk}
\label{sec-bhp-coupling}

In this section, we will prove the following coupling statement, which is an analog of~\cite[Proposition~4]{curien-legall-plane} for Brownian surfaces with boundary.  This coupling statement is needed for the proof of Theorem~\ref{thm-uihpq-ghpu} and also implies Proposition~\ref{prop-bhp-wedge}, as explained just below the proposition statement. 

\begin{prop} \label{prop-bhp-coupling}
Fix $a,\ell > 0$ and let $(H,d)$ be a Brownian disk with area $a$ and boundary length $\ell$.  Also let $(H_\infty , d_\infty)$ be a Brownian half-plane.  For each $\ep \in (0,1)$ there exists $\alpha >0$ and a coupling of $(H,d)$ with $(H_\infty , d_\infty)$ such that the following is true.  Let $\mu$ (resp.\ $\mu_\infty$) and $\eta$ (resp.\ $\eta_\infty$) be the area measure and boundary path of $H$ (resp.\ $H_\infty$).  In the notation of Definition~\ref{def-ghpu-truncate}, it holds with probability at least $1-\ep$ that
\eqbn
\frk B_\alpha\left( H , d , \mu ,\eta  \right) \quad \op{and} \quad \frk B_\alpha\left( H_\infty , d_\infty , \mu_\infty,\eta_\infty   \right) .
\eqen
agree as elements of $\ol{\BB M}^{\op{GHPU}}$. 
\end{prop}

Note that scaling distances in a Brownian surface by $r > 0$ corresponds to scaling boundary lengths by $r^2$ and areas by $r^4$.  Indeed, this follows by the scaling properties of Brownian motion. Hence Proposition~\ref{prop-bhp-coupling} (applied with $a=\ell = 1$) implies that for any $r > 0$ and any $\ep \in (0,1)$, there exists $R > 0$ such that if $\left( H  , d, \mu ,\eta  \right)$ is as in Proposition~\ref{prop-bhp-coupling} for $a = R$ and $\ell = R^{1/2}$, then we can couple $(H   , d )$ with $(H_\infty, d_\infty)$ in such a way that $\frk B_r\left( H  , d  , \mu  ,\eta   \right)$ and $\frk B_r\left( H_\infty , d_\infty , \mu_\infty,\eta_\infty   \right)$ coincide as elements of $\ol{\BB M}^{\op{GHPU}}$ with probability at least $1-\ep$. 
 
The proof of Proposition~\ref{prop-bhp-coupling} requires two lemmas. The first is a continuum analog of Lemma~\ref{prop-cactus}.

\begin{lem} \label{prop-bhp-cactus}
Suppose we are in the setting of Section~\ref{sec-brownian-disk}. 
In particular, fix $a , \ell > 0$, let $(H , d)$ be a Brownian disk with area $a$ and boundary length $\ell$, let $\BB p : [0,a] \rta H$ be the quotient map, let $(X , Z )$ be the encoding process, and let $T(r)$ for $r\in [0,\ell]$ be as in~\eqref{eqn-disk-inf}. 
Almost surely, the following is true. 
For $s,t\in [0,a]$ with $s \leq t$, let 
\eqbn
S (s,t) :=  (T([0,\ell]) \cap [s,t]) \cup  \{s,t\}  \quad \op{and} \quad
S'(s,t) :=  (T([0,\ell]) \setminus [s,t]) \cup  \{s,t\}  
\eqen 
so that $\BB p(S(s,t))$ (resp.\ $\BB p(S'(s,t))$) consists of $\BB p(s)$, $\BB p(t)$, and the set of points in $\bdy H$ which are (resp.\ are not) contained in $\BB p([s,t])$.
Then 
\eqb \label{eqn-bhp-cactus}
d \left( \BB p(s) , \BB p (t) \right) \geq Z (s) + Z (t) - 2 \max\left\{ \inf_{u \in S(s,t)} Z(u) , \inf_{u \in S'(s,t)} Z(u) \right\} . 
\eqe 
The analogous estimate also holds in the setting of the Brownian half-plane (note that in this case the infimum over $S'(s,t)$ in~\eqref{eqn-bhp-cactus} is a.s.\ equal to $-\infty$). 
\end{lem}
\begin{proof}
This follows from essentially the same argument used to prove the cactus bound in the case of the Brownian map (see, e.g.,~\cite[Equation (4)]{curien-legall-plane}). 
See also~\cite[Lemma 22]{bet-disk-tight} for an analogous estimate in the case when $\BB p(s) , \BB p(t) \in \bdy H$. 
\end{proof}

Next we prove a coupling statement for the encoding processes.

\begin{lem} \label{prop-bhp-path-coupling}
Fix $a,\ell >0$ and let $(X , Z  )$ be the encoding process for the Brownian disk with area $a$ and boundary length $\ell$ and let $T$ be as in~\eqref{eqn-disk-inf}. Also let $(X_\infty , Z_\infty  )$ be the encoding process for the Brownian half-plane and for $r\in\BB R$ let $T_\infty(r) = \inf\{t\in [0,\ell] \,:\, X_\infty(t) = -r \}$, as in Section~\ref{sec-bhp}. 
For each $\ep , \delta \in (0,1)$, there exists $0 < \delta_0 < \delta_1 \leq \delta$ and a coupling of $(X , Z )$ with $(X_\infty , Z_\infty  )$ such that with probability at least $1-\ep$, the following is true. We have $T(r) = T_\infty(r) $ for each $r\in [0 , \delta_1]$, $X(t) - \ell = X_\infty(t)$ for each $t\in [0,T(\delta_1)]$, and for each $t\in [T_\infty(\delta_0) , T_\infty(\delta_1)]$, 
\eqbn \label{eqn-bhp-path-couple}
  Z(t) - Z ( T(\delta_0)  )  = Z_\infty(t) - Z_\infty ( T_\infty(\delta_0)  )  .
\eqen 
\end{lem}
\begin{proof}By~\cite[Proposition~9]{bet-mier-disk}, for each $r \in [0,\ell]$ the Radon-Nikodym derivative of the law of $(X-\ell)|_{[0,T(r)]}$ with respect to the law of $X_\infty|_{[0,T_\infty(r)]}$ is given by $f_r(T_\infty(r)  )$, where
\eqbn
 f_r(t) := \BB 1_{(t < a)} \frac{(\ell-r) a^{3/2} }{\ell (a - t)^{3/2}} \exp\left( \frac{\ell^2}{2r } - \frac{(\ell-r)^2}{2(a- t)} \right) .
\eqen
Since $f_r(t)$ is continuous in both variables and $f_0(0)  = 1$, if we are given $\ep \in (0,1)$, then we can find $\zeta \in [0,\ell]$ such that $|f_r(t) - 1| \leq \ep/4$ for $t,r\in [0,\zeta]$. Choose $\delta_1 \in (0, \delta \wedge \zeta]$ such that $\BB P[ T_\infty(\delta_1) \leq \zeta ] \geq 1-\ep/4$. Then except on an event of probability at most $\ep/4$, the Radon-Nikodym derivative of the law of $(X-\ell)|_{[0,T(\delta_1)]}$ with respect to the law of $X_\infty|_{[0,T_\infty(\delta_1)]}$ is between $1-\ep/4$ and $1+\ep/4$, so we can couple these two restricted processes so that they agree with probability at least $1- \ep/2$. Henceforth assume we have chosen such a coupling. 

Recalling~\eqref{eqn-Z^0-def} and~\eqref{eqn-Z^0-def-infty}, we see that for any fixed realization $x$ of $(X-\ell)|_{[0,T(\delta_1)]}$, the regular conditional law of $Z^0|_{[0,T(\delta_1)]}$ given $\{(X-\ell)|_{[0,T(\delta_1)]} = x\}$
coincides with the regular conditional law of $Z_\infty|_{[0,T_\infty(\delta_1)]}$ given $\{ X_\infty|_{[0,T_\infty(\delta_1)]} = x\}$. 
Hence we can extend our coupling so that in fact
\eqb \label{eqn-snake-coupling-bhp}
T(\delta_1) = T_\infty(\delta_1)\quad \op{and}\quad (X-\ell , Z^0 )|_{[0,T(\delta_1)]} = (X_\infty , Z_\infty)|_{[0,T_\infty(\delta_1)]} 
\eqe 
with probability at least $1- \ep/2$. 

Recall that $Z$ (resp.\ $Z_\infty$) is obtained from $Z^0$ (resp.\ $Z_\infty^0$) by adding the composition of $\frk b$ (resp.\ $\frk b_\infty$) with its running infimum process, where $\frk b$ is $\sqrt3$ times a Brownian bridge and $\frk b_\infty$ is $\sqrt3$ times a standard linear Brownian motion. 

For $0<\delta_0 < \delta_1$, the law of $(\frk b - \frk b(\delta_0))|_{[\delta_0,\delta_1]}$ is absolutely continuous with respect to the law of $(\frk b_\infty - \frk b_\infty(\delta_0))|_{[\delta_0,\delta_1]}$, with Radon-Nikodym derivative tending to $1$ in probability as $\delta_0$ increases to $\delta_1$. We can find $\delta_0 \in (0,\delta_1)$ for which this Radon-Nikodym derivative lies in $[1-\ep/4 , 1+\ep/4]$ with probability at least $1-\ep/4$, so we can couple $(\frk b - \frk b(\delta_0))|_{[\delta_0,\delta_1]}$ with $(\frk b_\infty - \frk b_\infty(\delta_0))|_{[\delta_0,\delta_1]}$ in such a way that 
\eqb \label{eqn-bridge-coupling-bhp}
 (\frk b - \frk b(\delta_0))|_{[\delta_0,\delta_1]} = (\frk b_\infty - \frk b_\infty(\delta_0))|_{[\delta_0,\delta_1]} 
\eqe
with probability at least $1- \ep/2$. 

Since $\frk b$ (resp.\ $\frk b_\infty$) is independent from $(X , Z^0)$ (resp.\ $(X_\infty,Z^0_\infty)$), we can couple $(X,Z,\frk b)$ and $(X_\infty,Z_\infty,\frk b_\infty)$ so that~\eqref{eqn-snake-coupling-bhp} and~\eqref{eqn-bridge-coupling-bhp} hold simultaneously with probability at least $1-4\ep$. For such a coupling, the conditions in the statement of the lemma are satisfied. 
\end{proof}

\begin{proof}[Proof of Proposition~\ref{prop-bhp-coupling}]
Let $(X,Z)$ and $(X_\infty,Z_\infty)$ be the encoding processes, as in Proposition~\ref{prop-bhp-coupling}.  By Proposition~\ref{prop-bhp-coupling}, for each $\ep ,\delta \in (0,1)$ there exists $0 < \delta_0 < \delta_1 \leq \delta$ and a coupling of $(X , Z )$ with $(X_\infty , Z_\infty  )$ such that with probability at least $1-\ep/4$, we have $T(r) = T_\infty(r) $ for each $r\in [0 , \delta_1]$, $X(t) - \ell = X_\infty(t)$ for each $t\in [0,T(\delta_1)]$, and for each $t\in [T_\infty(\delta_0) , T_\infty(\delta_1)]$, 
\eqbn \label{eqn-bhp-path-agree}
  Z(t) - Z ( T(\delta_0)  )  = Z_\infty(t) - Z_\infty ( T_\infty(\delta_0)  )  .
\eqen 
Note that with $\frk b_\infty $ (resp.\ $\frk b$) the Brownian motion (resp.\ Brownian bridge) from Section~\ref{sec-bhp} (resp.\ Section~\ref{sec-brownian-disk}), we have $Z_\infty\circ T = \frk b_\infty $ (resp.\ $Z \circ T = \frk b$), so~\eqref{eqn-bhp-path-agree} implies that also $\frk b(r) - \frk b (\delta_0) = \frk b_\infty(r) - \frk b_\infty(\delta_0)$ for each $r\in [\delta_0,\delta_1]$. 

If we choose $\delta_0$ and $\delta_1$ sufficiently small, then it is unlikely that the infimum of $\frk b$ over $[0,\delta_0]$ is smaller than the infimum of $\frk b$ over $[\delta_1,\ell]$.
Since $\frk b$ is a constant times a Brownian bridge, $\frk b(\delta_0)$ is a.s.\ strictly larger than the minimum value of $\frk b$ on each of $[0,\delta_0]$ and $[\delta_1 , \ell]$. Furthermore, if we choose $\delta_1$ sufficiently close to $\delta_0$ (leaving $\delta_0$ fixed), then by continuity it is likely that $|Z(s) - \frk b(\delta_0)|$ is close to 0 for each $s \in [T(\delta_0) , T(\delta_1)]$. Hence by choosing $\delta_0$ sufficiently small and then choosing $\delta_1$ sufficiently close to $\delta_0$, we can arrange that except on an event of probability at most $ \ep/4$, 
\eqb \label{eqn-bhp-path-max}
\max\left\{  \inf_{r \in [0 , \delta_0] } \frk b(r) ,  \inf_{r \in [\delta_1 , \ell] } \frk b(r)  \right\}   < \inf_{s \in [T(\delta_0) , T( \delta_1 )]} Z(s)  .
\eqe  
 
Let $\delta_* :=  \frac12(\delta_0 + \delta_1)$. 
Also define
\eqbn
\rho := - \left(   \inf_{r \in [ \delta_0 , \delta_* ]_{\BB Z}}  \left(   \frk b_\infty(r)  - \frk b_\infty(\delta_0)   \right)    \vee    \inf_{r \in [\delta_* , \delta_1 ]_{\BB Z}} \left(   \frk b_\infty(r)  - \frk b_\infty(\delta_0)  \right) \right) \geq 0  
\eqen
and 
\eqbn
\rho' := \frk b_\infty(\delta_*) - \frk b_\infty(\delta_0) +\rho.
\eqen 
Since $\frk b_\infty$ is a constant times a standard Brownian motion, there exists $\alpha >0$ such that with probability at least $1-\ep/2$ we have $\frac14 \rho' \geq   \alpha$. 

By definition, the quotient map $\BB p_\infty : \BB R \rta H$ satisfies $\BB p_\infty(T_\infty(\delta_*)) = \eta_\infty(\delta_*)$. By invariance of the law of the Brownian half-plane under re-rooting along the boundary, which follows from~\eqref{eqn-bhp-reroot}, 
\eqbn
\left( H_\infty , d_\infty , \mu_\infty ,\eta_\infty(\delta_* + \cdot)  \right) \eqD \left( H_\infty , d_\infty , \mu_\infty ,\eta_\infty   \right) .
\eqen
It is immediate from Theorem~\ref{thm-disk-ghpu} that the analogous property also holds for the Brownian disk. 
Hence it suffices to show that whenever~\eqref{eqn-bhp-path-agree} and~\eqref{eqn-bhp-path-max} occur,
\eqb \label{eqn-bhp-coupling-show}
\frk B_{\frac14 \rho'} \left( H , d , \mu ,\eta(\delta_* + \cdot)  \right) \quad \op{and} \quad \frk B_{\frac14 \rho'}\left( H_\infty , d_\infty , \mu_\infty,\eta_\infty(\delta_* + \cdot)   \right) 
\eqe 
agree as elements of $\ol{\BB M}^{\op{GHPU}}$.  

Henceforth assume that~\eqref{eqn-bhp-path-agree} and~\eqref{eqn-bhp-path-max} occur (which happens with probability at least $1-\ep/2$).  
Let $\BB p  : [0, a] \rta H$ and $\BB p_\infty : \BB R\rta H_\infty$ be the quotient maps. 
Let $r_0 \in [\delta_0 , \delta_*]_{\BB Z}$ and $r_1 \in [\delta_* , \delta_1  ]_{\BB Z}$ be chosen so that 
\eqbn
\frk b_\infty(r_0) -  \frk b_\infty(r_1) \leq -\rho \quad \op{and} \quad \frk b_\infty(r_1)  -  \frk b_\infty(\delta_0)   \leq -\rho . 
\eqen
By Lemma~\ref{prop-bhp-cactus} together with~\eqref{eqn-bhp-path-agree} and~\eqref{eqn-bhp-path-max} (the latter equation is used to deal with the second term in the maximum in~\eqref{eqn-bhp-cactus}), if $t\geq T(r_1)$ then 
\alb
d \left( \BB p(t) , \eta (\delta_*)  \right) 
&\geq Z (t) + \frk b(\delta_*) - 2 \max \left\{ \inf_{u \in S (  T (\delta_*) , t )} Z (u) , \inf_{u \in S' (  T (\delta_*) , t )} Z (u)  \right\} \\
&\geq Z (t) + \frk b(\delta_*) - Z(t)   +  \rho -   \frk b(\delta_*)   
 \geq \rho' .
\ale
We have a similar estimate if $t\leq T (r_0)$. 
Again using Lemma~\ref{prop-bhp-cactus}, we similarly obtain that if $t \in [0,a] \setminus [T(r_0) , T(r_1)]$ then $d_\infty \left( \BB p_\infty(t), \eta_\infty (\delta_*)  \right) \geq \rho'$.

Hence
\eqbn
B_{\rho'}\left(\eta (\delta_*) ; d  \right) \subset \BB p \left( [T (r_0) , T (r_1)  ] \right) \quad\op{and}\quad 
B_{\rho'}\left(\eta_\infty(\delta_*) ; d_\infty \right) \subset \BB p_\infty\left( [T_\infty(r_0) , T_\infty(r_1)  ] \right)  . 
\eqen
From the definitions~\eqref{eqn-dist0-def-infty} and~\eqref{eqn-dist0-def} of the pseudometrics $d^0$ and $d_\infty^0$, respectively, and the triangle inequality, we find that the sets $\BB p^{-1}\left( B_{\frac14 \rho'}\left(\eta (\delta_*) ; d  \right) \right)$ and $\BB p_\infty^{-1}\left( B_{\frac14 \rho'}\left(\eta_\infty (\delta_*) ; d_\infty  \right) \right)$ (resp.\ the distances $d^0(s,t)$ and $d_\infty^0(s,t)$ for $s$ and $t$ in these sets) are given by the same deterministic functionals of $X|_{[0 , T(\delta_1)]}$ and $(Z-Z(T(\delta_0))) |_{[T(\delta_0) , T(\delta_1)]}$ and $(Z-Z(T(\delta_0))) |_{[T(\delta_0) , T(\delta_1)]}$. Note that we use~\eqref{eqn-bhp-path-max} to resolve the discrepancy between the definitions~\eqref{eqn-d_Z-infty} and~\eqref{eqn-d_Z}. The measures $\mu$ and $\mu_\infty$ and the paths $\eta$ and $\eta_\infty$ are determined by the same local functionals of the Schaeffer encoding functions. From these considerations, we see that the map $\BB p(t) \mapsto \BB p_\infty(t)$ for $t\in [T_\infty(r_0) , T_\infty(r_1)]$ is well defined and restricts to a measure- and curve-preserving isometry from $B_{\frac14\rho'}\left(\eta (\delta_*) ; d  \right)$ to $B_{\frac14\rho }\left(\eta_\infty (\delta_*) ; d_\infty  \right)$. Hence~\eqref{eqn-bhp-coupling-show} holds.
\end{proof}

\subsection{Coupling the UIHPQ with a finite uniform quadrangulation with boundary}
\label{sec-map-coupling}

In this section we prove a discrete analog of Proposition~\ref{prop-bhp-coupling}, which is also an analog of~\cite[Proposition~9]{curien-legall-plane} for maps with boundary. Throughout this section, we let $(Q_\infty , \BB e_\infty)$ be a UIHPQ. We also fix $\ell > 0$, a sequence of positive integers $\{l^n\}_{n\in\BB N}$ with $(2n)^{-1/2} l^n \rta \ell$, and for $n\in\BB N$ we let $(Q^n , \BB e_0^n , \BB v_*^n)$ be sampled uniformly from $\mcl Q^\bullet(n,l^n)$ (defined in Section~\ref{sec-quad-def}). 
Also let $\BB v_\infty$ (resp.\ $\BB v_0^n$) be the initial endpoint of $\BB e_\infty$ (resp.\ $\BB e_0^n$), so that $\BB v_\infty = \lambda_\infty(0)$, with $\lambda_\infty$ the linearly interpolated boundary path of $Q_\infty$.

\begin{prop} \label{prop-map-coupling}
For each $\ep \in (0,1)$ there exists $\alpha >0$ and $n_*\in\BB N$ such that for $n\geq n_*$, there is a coupling of $(Q_\infty , \BB e_{\infty})$ with $(Q^n , \BB e_0^n , \BB v_*^n)$ with the following property. With probability at least $1-\ep$, the graph metric balls $B_{\alpha n^{1/4}}(\BB v_0^n ; Q^n)$ and $B_{\alpha n^{1/4}}(\BB v_{\infty} ; Q_\infty)$ equipped with the graph structures they inherit from $Q^n$ and $Q_\infty$, respectively, are isomorphic (as graphs) via an isomorphism which takes $\BB e_0^n$ to $\BB e_\infty$ and $ \bdy Q^n  \cap B_{\alpha n^{1/4}}(\BB v_0^n ; Q^n)$ to $ \bdy Q_\infty  \cap B_{\alpha n^{1/4}}(\BB v_{\infty} ; Q_\infty)$. Furthermore, we can arrange that this isomorphism is an isometry for the metrics on $B_{\alpha n^{1/4}}(\BB v_0^n ; Q^n)$ and $B_{\alpha n^{1/4}}(\BB v_{\infty} ; Q_\infty)$ which they inherit from $Q^n$ and $Q_\infty$, respectively. 
\end{prop}

The proof of Proposition~\ref{prop-map-coupling} is similar to that of Proposition~\ref{prop-bhp-coupling}. 
We will first construct a coupling of the encoding functions, then transfer to a coupling of the maps using Lemma~\ref{prop-cactus}.
The following lemma is needed to bound the Radon-Nikodym derivative of the encoding functions.

 \begin{lem} \label{prop-srw-hit}
Let $\{S(i)\}_{i\in\BB N_0}$ be a simple random walk on $\BB Z$ started from 0, with steps distributed uniformly in $\{-1,1\}$. Let $\{\mcl F_i\}_{i\in\BB N_0}$ be the filtration generated by $S$. For $m \in \BB N_0$, let $I(m) := \inf\{i\in \BB N_0 \,:\, S(i) = -m\}$. Let $N$ be a stopping time for $\{\mcl F_i\}_{i\in\BB N_0}$. Then for $m, n \in \BB N $, it holds on the event $\{ N < I(m) \wedge n \}$ that
\eqbn
\BB P\left[ I(m) = n \,|\, \mcl F_{N} \right]  =   \frac{1}{\sqrt{2\pi}} (S(N) +m) (n-N)^{-3/2} \exp\left(- \frac{ (S(N) +m)^2 }{2(n-N)}  \right)     + o \left(  (S(N)+m)^{-2} \right) 
\eqen
with the rate of the $o \left(  (S_{N}+m)^{-2} \right)$ error universal and deterministic.
\end{lem}
\begin{proof}
By Donsker's theorem, we have $m^2 I(m) \rta T$ in law, where $T$ is the first time a standard linear Brownian motion hits $-1$. By the local limit theorem for stable laws (see, e.g.,~\cite[Section~50]{gk-limit-theorems}), we find that for $n\in\BB N$, 
\eqbn
\lim_{n\rta\infty} \sup_{n\in\BB N} \left| m^2 \BB P\left[ I(m) = n \right] - g\left( \frac{n}{m^2} \right) \right| = 0,
\eqen
where
\eqbn
g(t) = \frac{1}{\sqrt{2\pi} t^{3/2}} e^{-\frac{1}{2t}} 
\eqen
is the density of $T$. By the strong Markov property, on the event $\{N < I(m) \wedge n\}$ we have 
\eqbn
\BB P\left[ I(m) = n \,|\, \mcl F_{N} \right]   = \BB P\left[ I(-s+m) = n-N \right]  |_{s = S_{N}} .
\eqen
The statement of the lemma follows.
\end{proof} 

In what follows, we define the encoding paths $ C_\infty , L_\infty^0 , L_\infty$, $b_\infty$, and $b_\infty^0$ as in Section~\ref{sec-uihpq} and the analogous finite-volume objects $C^n ,   L^{0,n} , L^n , b^n$, and $b^{0,n}$ as in Section~\ref{sec-quad-bdy} for the quadrangulation $Q^n$, but with an additional superscript $n$.  For $m \in [0,l^n] $ (resp.\ $m\in\BB Z$) we let $I^n(m)$ (resp.\ $I_\infty(m)$) be the smallest $i\in [0,2n+l^n]_{\BB Z}$ (resp.\ $i\in\BB Z$) for which $C^n(i) = -m$ (resp.\ $C_\infty(i) = -m$), as in~\eqref{eqn-quad-bdy-inf} (resp.~\eqref{eqn-quad-bdy-inf-infty}). We extend these functions to $[0,2n+l]$ (resp.\ $\BB R$) by setting $I^n(s) = I^n(\lfloor s \rfloor)$ (resp.\ $I_\infty(s) = I_\infty(\lfloor s \rfloor)$).  The following lemma gives us a coupling of the Schaeffer encoding functions of the maps in Proposition~\ref{prop-map-coupling}.

\begin{lem} \label{prop-path-coupling}
For each $\ep , \delta \in (0,1)$, there exists $0<\delta_0 < \delta_1 \leq \delta$ and $n_*\in\BB N$ such that for $n\geq n_*$, there exists a coupling of the encoding pairs $(C^n , L^n  )$ and $(C_\infty , L_\infty )$ such that with probability at least $1-\ep$, the following is true. 
We have $I^n(m) = I_\infty(m)$ for each $ m \in [0 , \delta_1 n^{1/2} ] $. Furthermore, for each $i\in [0 , I^n(\delta_1 n^{1/2})]_{\BB Z}$ we have $C^n(i) = C_\infty(i)$ and for each $i\in [I^n(\delta_0 n^{1/2}) , I^n(\delta_1 n^{1/2})]_{\BB Z}$ we have 
\eqbn 
 L^n(i) - L^n( I^n(\delta_0 n^{1/2}) ) = L_\infty(i) - L_\infty(I_\infty(\delta_0 n^{1/2}) ) .
\eqen 
\end{lem}
\begin{proof}
Recall that the law of $ C^n$ is that of a simple random walk conditioned to first hit $-l^n$ at time $2n+l^n$ and the law of $ C_\infty|_{\BB N_0} $ is that of an unconditioned simple random walk (Lemmas~\ref{prop-quad-bdy-law} and~\ref{prop-uihpq-encode-law}).  
By Lemma~\ref{prop-srw-hit} and Bayes' rule, for $\delta\in (0,\ell)$ and $n\in \BB N$, then the Radon-Nikodym derivative of the law of $  C^n|_{[0, I^n(\delta n^{1/2}) ]_{\BB Z}}$ with respect to the law of $  C_\infty|_{[0, I_\infty(\delta n^{1/2})]_{\BB Z}}$ is given by $f_\delta^n\left( I_\infty(\delta  n^{1/2} ) \right)$
where for $k \in  [ 0 , 2n + l^n ]_{\BB Z}$,
\alb
f_\delta^n (k  ) = \frac{  ( l^n -   \lfloor \delta  n^{1/2} \rfloor   ) (2n + l^n - k   )^{-3/2} \exp\left(- \frac{ (l^n -   \lfloor \delta  n^{1/2} \rfloor )^2 }{2(2n + l^n - k )}  \right)     + o \left(  (   l^n -   \lfloor \delta  n^{1/2} \rfloor      )^{-2} \right)  }{ l^n  (2n + l^n    )^{-3/2} \exp\left(- \frac{ (l^n )^2 }{2(2n + l^n )}  \right)     + o \left(  (   l^n   )^{-2} \right)      }    \BB 1_{(k < 2n +l^n)} .
\ale 
Since $(2n)^{-1/2} l^n \rta \ell$, if we are given $\ep \in (0,1)$, we can find $\zeta > 0$ and $n_0\in\BB N $ such that for $n\geq n_0$, $1\leq  k \leq \zeta n$, and $0 < \delta_1 \leq \zeta  $, we have $\left| f_{\delta_1}^n (k ) - 1 \right| \leq \frac{\ep}{4}$. 
Since $(2n)^{-1/2}  C^n((2n)^{-1} \cdot)$ converges in law in the uniform topology to an appropriate conditioned Brownian motion~\cite[Lemma~14]{bet-pos-genus}, we can find $ \delta_1 \in (0, \delta\wedge \zeta]$ and $n_* \geq n_0$ such that for $n\geq n_*$, it holds with probability at least $1-\ep/4$ that $I(\delta_1 n^{1/2}) \leq \zeta n$. If $n\geq n_*$, then except on an event of probability at most $1-q/4$, the Radon-Nikodym derivative of the law of $ C^n|_{[0, I^n(\delta_1 n^{1/2}) ]_{\BB Z}}$ with respect to the law of $ C_\infty|_{[0, I_\infty(\delta_1 n^{1/2}) ]_{\BB Z}}$ lies in $[1-\ep/4 , 1+\ep/4]$. Hence we can couple these restricted processes together so that with probability at least $1-\ep/2$,
\eqb \label{eqn-head-coupling}
I^n(\delta_1 n^{1/2}) = I_\infty(\delta_1 n^{1/2}) \quad\op{and} \quad  C^n|_{[0, I^n(\delta_1 n^{1/2}) ]_{\BB Z}} =  C_\infty|_{[0, I_\infty(\delta_1 n^{1/2}) ]_{\BB Z}} .
\eqe 
 
Since the conditional law of the shifted label function $L^{0,n}|_{[0, I^n(\delta_1 n^{1/2})]_{\BB Z}}$ given $C^n|_{[0, I^n(\delta_1 n^{1/2}) ]_{\BB Z}}$ coincides with the conditional law of $L_\infty^0|_{[0, I_\infty(\delta_1 n^{1/2})]_{\BB Z}}$ given $C_\infty|_{[0, I_\infty(\delta_1 n^{1/2}) ]_{\BB Z}}$, we also obtain a coupling of $(C^n , L^{0,n})$ with $(C_\infty , L_\infty^0)$ such that with probability at least $1-\ep/2$, 
\eqb \label{eqn-snake-coupling}
(C^n , L^{0,n})|_{[0, I^n(\delta_1 n^{1/2})]_{\BB Z}} =   (C_\infty , L_\infty^0)|_{[0, I_\infty(\delta_1 n^{1/2}) ]_{\BB Z}}  .
\eqe
 
Recall that $L^n$ (resp.\ $L_\infty$) is obtained from $(C_\infty , L_\infty^0)$ and the bridge $b^{0,n}$ (resp.\ $(C_\infty , L_\infty^0)$ and the walk $b_\infty^0$) in the manner described in Section~\ref{sec-quad-bdy} (resp.\ Section~\ref{sec-uihpq}). 
Recall also the processes $b^n$ and $b_\infty$ obtained from $b^{0,n}$ and $b_\infty^0$, respectively, by considering only times when the path makes a downward step. 
A similar absolute continuity argument to the one given above shows that after possibly increasing $n_*$, we can find 
$\delta_0 \in (0,\delta_1]$ and a coupling of $b^{0,n}$ with $b_\infty^0$ such that with probability at least $1-\ep/2$,  
\eqb \label{eqn-bridge-coupling}
b^n(k) - b^n(\lfloor \delta_0 n^{1/2} \rfloor) = b_\infty(k) - b_\infty( \lfloor \delta_0 n^{1/2} \rfloor) ,\quad \forall k \in [\delta_0 n^{1/2} , \delta_1 n^{1/2}]_{\BB Z}  .
\eqe 

The pair $(C^n , L^{0,n})$ (resp.\ $(C_\infty , L_\infty^0)$) is independent from $b^{0,n}$ (resp.\ $b_\infty^0$), so for $n\geq n*$, we can couple $(C^n , L^{0,n} , b^{0,n})$ with $(C_\infty , L_\infty^0 , b_\infty^0)$ in such a way that~\eqref{eqn-snake-coupling} and~\eqref{eqn-bridge-coupling} hold simultaneously with probability at least $1-\ep$. Such a coupling satisfies the conditions in the statement of the lemma.
\end{proof}

\begin{proof}[Proof of Proposition~\ref{prop-map-coupling}]
The proof is essentially identical to that of Proposition~\ref{prop-bhp-coupling}, but we give the details for the sake of completeness.
By Lemma~\ref{prop-path-coupling}, we can find $0 < \delta_0 < \delta_1 < \ell$ and $n_*\in\BB N $ such that for $n\geq n_*$, there exists a coupling of $(C^n , L^n)$ with $(C_\infty , L_\infty)$ such that with probability at least $1-\ep/4$, $I^n(m) = I_\infty(m)$ for each $  m \in [0 , \delta_1 n^{1/2} ]$, $C^n(i) = C_\infty(i)$ for each $i\in [0,I^n(\delta_1 n^{1/2})]$, and for each $i\in [I^n(\delta_0 n^{1/2}) , I^n(\delta_1 n^{1/2})]_{\BB Z}$, 
\eqb \label{eqn-path-agree-event}
 L^n(i) - L^n( I^n(\delta_0 n^{1/2}) ) = L_\infty(i) - L_\infty(I_\infty(\delta_0 n^{1/2}) ) .
\eqe 
Note that with $ b^n  $ (resp.\ $ b_\infty$) the process from Section~\ref{sec-quad-bdy} (resp.\ Section~\ref{sec-quad-bdy}), we have $L^n(I^n(m)) = b^n(m+1) $ (resp.\ $L_\infty(I_\infty(m)) = b_\infty(m+1)$), so~\eqref{eqn-bhp-path-agree} implies that also $b^n(m)- b^n(\lceil \delta_0 n^{1/2} \rceil) = b_\infty(m)- b_\infty(\lceil \delta_0 n^{1/2} \rceil) $ for each $  m \in [\delta_0 n^{1/2} , \delta_1 n^{1/2} ]$.

By choosing $\delta_0$ sufficiently small and then $\delta_1$ sufficiently close to $\delta_0$ and possibly increasing $n_*$ (c.f.\ the argument right before~\eqref{eqn-bhp-path-max} in the proof of Proposition~\ref{prop-map-coupling}), we can arrange that for $n\geq n_*$, it holds except on an event of probability at most $1-\ep/4$ that 
\eqb \label{eqn-path-max}
\max\left\{\min_{m \in  [ 0,  \delta_1 n^{1/2}  ]_{\BB Z} }  b^n(m) ,   \min_{m \in  [ \delta_1 n^{1/2} , l^n]_{\BB Z} }  b^n(m)  \right\} <   \min_{i \in  [I^n(\delta_0 n^{1/2}  ), I^n(\delta_1 n^{1/2})]_{\BB Z} } L^n(i)  .
\eqe

Let
\eqbn
m_* := \left\lfloor \frac12(\delta_0 + \delta_1) n^{1/2} \right\rfloor .
\eqen
Also define
\eqbn
r := - \left(   \min_{m \in [ \delta_0 n^{1/2} , m_*-1 ]_{\BB Z}}  \left(   b_\infty(m)  - b_\infty(\lceil \delta_0 n^{1/2}  \rceil)    \right)    \vee    \min_{m \in [m_*+1,  \delta_1 n^{1/2}  ]_{\BB Z}} \left( b_\infty(m)  - b_\infty(\lceil \delta_0 n^{1/2}  \rceil) \right) \right)   , 
\eqen
so that $r\geq 0$, and let
\eqbn
r' := b_\infty(m_*)  - b_\infty(\lceil \delta_0 n^{1/2}  \rceil) +  r     .
\eqen 
Since $m\mapsto   b_\infty(m ) - b_\infty(\lfloor \delta_0 n^{1/2} \rfloor)$ is obtained from a simple random walk by skipping its upward steps, we can find $\alpha >0$ such that for large enough $n$, it holds with probability at least $1-\ep/2$ that $\frac14 r'-1 \geq   \alpha n^{1/4}$. 

Recall the contour functions $\BB p^n : [0,2n+l^n]_{\BB Z} \rta \mcl V(Q^n)$ and $\BB p_\infty : \BB Z\rta \mcl V(Q_\infty)$. 
Let $\wt{\BB v}_0^n := \BB p^n(I^n(m_*))$ and $\wt{\BB v}_\infty^n := \BB p_\infty(I_\infty(m_*))$. 
If we let $\wt{\BB e}_0^n$ (resp.\ $\wt{\BB e}_\infty$) be the edge of $\bdy Q^n$ (resp.\ $\bdy Q_\infty$) immediately to the left of $\wt{\BB v}_0^n$ (resp.\ $\wt{\BB v}_\infty$), with one of the two possible orientations chosen with probability $1/2$ each. Then by re-rooting invariance,
\eqbn
(Q^n , \wt{\BB e}_0^n , \BB v_*^n ) \eqD (Q^n , \BB e_0^n , \BB v_*^n) \quad \op{and} \quad (Q , \wt{\BB e}_\infty) \eqD (Q , \BB e_{\infty}) .
\eqen
Hence it suffices to show that whenever~\eqref{eqn-path-agree-event} and~\eqref{eqn-path-max} occur, it holds that $B_{\frac14 (r' - 3) } \left( \wt{\BB v}_0^n  ; Q^n \right)$ and $ B_{\frac12 (r' - 3)} \left( \wt{\BB v}_\infty  ; Q_\infty \right)$ are isomorphic via a graph isomorphism satisfying the conditions in the statement of the proposition. 

Henceforth assume that~\eqref{eqn-path-agree-event} and~\eqref{eqn-path-max} occur.
Let $m_0 \in [ \delta_0 n^{1/2}   ,m_*-1]_{\BB Z}$ and $m_1 \in [i_*+1 ,  \delta_1 n^{1/2} ]_{\BB Z}$ be chosen so that 
\eqbn
  b_\infty(m_0) - b_\infty(\lceil \delta_0 n^{1/2} \rceil)  \leq -r \quad \op{and} \quad b_\infty(m_1)  - b_\infty(\lceil \delta_0 n^{1/2} \rceil) \leq -r . 
\eqen
By  Lemma~\ref{prop-cactus} together with~\eqref{eqn-path-agree-event} and~\eqref{eqn-path-max} (the latter equation is used to deal with the second term in the maximum in the estimate of Lemma~\ref{prop-cactus}), if $i \in \BB N_0 $ with $i \geq    I^n(m_1) $ then  
\alb
\op{dist}\left(\BB p^n(i) , \wt{\BB v}^n ; Q^n \right) 
&\geq L^n(i) + b^n(m_*) - 2 \max\left\{  \inf_{j \in J^n( I^n(m_*)   ,  i)   } L^n(j)  , \inf_{j \in (J')^n( I^n(m_*)   ,  i)   } L^n(j)  \right\} \\
&\geq L^n(i) + b^n(m_*) - L^n(i) + r -  b^n(\lfloor \delta_0 n^{1/2} \rfloor) 
 \geq r' .
\ale
We have a similar estimate if $i \leq I^n(m_0)$. 
Again using Lemma~\ref{prop-cactus}, we obtain that if $i \in [0,2n]_{\BB Z}\setminus [I^n(m_0) , I^n(m_1)]_{\BB Z}$ then
\eqbn
\op{dist}\left( \BB p_\infty(i) , \wt{\BB v}_\infty ; Q_\infty \right) \geq r'  .
\eqen 
Hence $\mcl V\left(  B_{r' } \left( \wt{\BB v}_{\infty}   ; Q_\infty \right)  \right) \subset \BB p_\infty( [I_\infty(m_0) , I_\infty(m_1)]_{\BB Z})$ and $\mcl V\left(  B_{r' } \left( \wt{\BB v}_0^n  ; Q^n \right)  \right) \subset \BB p^n([I^n(m_0) , I^n(m_1)]_{\BB Z})$. By the local nature of the Schaeffer bijection and~\eqref{eqn-path-agree-event} it holds that $B_{\frac14 r'-1 } \left( \wt{\BB v}_0^n   ; Q^n \right) $ and $ B_{ \frac14 r'-1 } \left( \wt{\BB v}_{\infty}   ; Q_\infty \right)$ are isomorphic as graphs via an isomorphism satisfying the conditions in the proposition statement. The triangle inequality implies that any such isomorphism is an isometry when these balls are equipped with the metrics they inherit from $Q^n$ and $Q_\infty$, respectively. 
\end{proof}

\subsection{Proof of Theorem~\ref{thm-uihpq-ghpu}}
\label{sec-uihpq-proof}
 
We will deduce the scaling limit statement for the UIHPQ from Theorem~\ref{thm-disk-ghpu} and the coupling results (Propositions~\ref{prop-bhp-coupling} and~\ref{prop-map-coupling}).

For $R>0$, let $(H_R , d_R )$ be a Brownian disk with area $R$ and boundary length $R^{1/2}$ and let $\mu_R$ and $\eta_R: [0,R] \rta \bdy H_R$ be its area measure and boundary path, respectively. 
Let
\eqbn
\frk H_R := \left(H_R , d_R , \mu_R , \eta_R \right) .
\eqen
By Brownian scaling and the definition of the Brownian disk from Section~\ref{sec-brownian-disk}, $\frk H_R$ has the same law as $\left(H_1 , R^{1/4} d_1 , R \mu_1 , \eta_1( R^{-1/2} \cdot) \right)$. 

Let $(Q_R^n , \BB e_R^n , \BB v_{*,R}^n)$ be sampled uniformly from the set of boundary-rooted pointed quadrangulations with $ \lfloor R n \rfloor$ interior faces and perimeter $ \lfloor 2^{3/2} (R n)^{1/2} \rfloor$. 
View $Q_R^n$ as a topological space in the manner of Remark~\ref{remark-ghpu-graph}. 
Let $d_R^n$ be $(9/8)^{1/4}  n^{-1/4}$ times the graph distance on $Q_R^n$ and let $\mu_R^n$ be the measure on $\mcl V(Q_R^n)$ which assigns mass to each vertex equal to $ (4n)^{-1}$ times its degree.
Let $\lambda^n_R : [0, \lfloor 2^{3/2}  (R n)^{1/2} \rfloor] \rta \bdy Q_R^n$ be the boundary path of $ Q_R^n$ started from $\BB e_R^n$ (extended to $\BB R$ by linear interpolation) and let $\eta_R^n(t) := \lambda_R^n\left( 2^{3/2} n^{1/2} t \right)$ for $t\in [0,  R^{1/2}]$.  
Let
\eqbn
\frk Q_R^n := \left( Q_R^n , d_R^n , \mu_R^n , \eta_R^n \right) .
\eqen
By Theorem~\ref{thm-disk-ghpu} and the aforementioned scaling relation between $\frk H_R$ and $\frk H_1$, we find that $\frk Q_R^n \rta \frk H_R$ in the GHPU topology for each $R>0$. 
 
Now fix $r>0$ and $\ep \in (0,1)$.   
We recall the $r$-truncation operator $\frk B_r$ from Definition~\ref{def-ghpu-truncate}. 
By Proposition~\ref{prop-bhp-coupling} and the scale invariance of the law of the Brownian half-plane, we can find $R > r$ and a coupling of $(H_\infty,d_\infty)$ with $(H_R ,d_R)$ such that with probability at least $1-\ep$,  
$\frk B_r \frk H_\infty$ and $\frk B_r \frk H_R$
agree as elements of $\ol{\BB M}^{\op{GHPU}}$.

By Proposition~\ref{prop-map-coupling} (applied with $\lfloor R n \rfloor$ in place of $n$), after possibly increasing $R$, we can find $n_* \in \BB N$ such that for $n\geq n_*$, we can couple $(Q_\infty  , \BB e_\infty )$ with $(Q^n_R  , \BB e_R^n , \BB v_{*,R}^n )$ in such a way that with probability at least $1-\ep$, $\frk B_r \frk Q_\infty^n $ and $\frk B_r \frk Q_R^n$
agree as elements of $\ol{\BB M}^{\op{GHPU}}$.

Since $\frk Q_R^n \rta \frk H_R$ in the GHPU topology and a.s.\ $r$ is a good radius for $\frk H_R$ (Definition~\ref{def-good-radius}), Lemma~\ref{prop-local-ghpu-subsequence} implies that $\frk B_r \frk Q_R^n \rta \frk B_r \frk H_R$ in law in the GHPU topology. 
Since $\ep \in (0,1)$ is arbitrary, the existence of the above couplings implies that $\frk B_r \frk Q_\infty^n \rta \frk B_r\frk H_\infty$ in law in the GHPU topology. 
Since $r>0$ is arbitrary and by Lemma~\ref{prop-local-ghpu-subsequence}, it follows that $\frk Q_\infty^n \rta \frk H_\infty$ in law in the local GHPU topology.   \qed

\subsection{Proof of Theorem~\ref{thm-uihpqS-ghpu}}
\label{sec-uihpqS-proof}

In this section we will deduce Theorem~\ref{thm-uihpqS-ghpu} from Theorem~\ref{thm-uihpq-ghpu}. 
We define the UIHPQ$_{\op{S}}$ $(Q_{\op{S}} , \BB e_{\op{S}})$ and its associated elements of $\BB M_\infty^{\op{GHPU}}$, $\frk Q_{\op{S}}^n$ for $n\in\BB N$, as in the discussion just above Theorem~\ref{thm-uihpqS-ghpu}.

We assume throughout this section that we have coupled $(Q_{\op{S}} , \BB e_{\op{S}})$ with an instance of the UIHPQ $(Q_\infty , \BB e_\infty)$ in the manner described in Section~\ref{sec-pruning}, so that $Q_{\op{S}}$ is obtained from $Q_\infty$ by pruning the dangling quadrangulations from $Q_\infty$. 
Let $\{ (q_v , e_v) \}_{v\in \mcl V(\bdy Q_{\op{S}} )}$ be these dangling quadrangulations (with their oriented boundary root edges). Then the pairs $(q_v , e_v)$ except for the one dangling from the right endpoint of $\BB e_{\op{S}}$ are i.i.d.\  samples from the free Boltzmann distribution on quadrangulations with general boundary (Section~\ref{sec-pruning}) and
\eqbn
Q_\infty  = Q_{\op{S}}  \cup \bigcup_{v\in\mcl V(\bdy Q_{\op{S}} )} q_v    .
\eqen     
Since each of the quadrangulations $q_v $ for $v\in \bdy Q_{\op{S}} $ is disconnected from $Q_{\op{S}} $ by removing the single vertex $v$, it follows that no geodesic between vertices of $Q_{\op{S}} $ enters $q_v \setminus \{v\}$. Hence for $r>0$ and $v\in \mcl V( Q_{\op{S}})$, 
\eqb \label{eqn-uihpq-balls}
 B_r\left( v; Q_{\op{S}}  \right) = B_r\left(  v   ; Q_\infty   \right) \cap  Q_{\op{S}}   .
\eqe 

In what follows we will bound the diameters, areas, and boundary lengths of the extra quadrangulations $q_v$, with the eventual goal of showing that $Q_{\op{S}}$ and $Q_\infty$ have the same scaling limit in the GHPU topology (up to multiplying boundary lengths by a constant factor). 

Before proving these bounds, it will be convenient to have a general lemma for how many vertices along the boundary can be contained in a metric ball. 
Let $\{v_k\}_{k\in \BB Z}$ be the vertices of $\bdy Q_{\op{S}}$, listed in the order in which they are hit by the boundary path $\lambda_{\op{S}}$ and enumerated so that $v_0$ is the right endpoint of $\BB e_{\op{S}}$ (so that $q_{v_0}$ is the one dangling quadrangulation which does not agree in law with the others; recall Section~\ref{sec-pruning}).  
  
Let $K_r^n$ be the largest $k\in\BB N$ for which either $v_k $ or $v_{-k} $ belongs to $  B_{r n^{1/4}}\left( \BB e_{\op{S}} ;  Q_{\op{S}} \right)$. 
By~\eqref{eqn-uihpq-balls}, $K_r^n$ is also the largest $k\in \BB N$ for which either $v_k $ or $v_{-k} $ belongs to $ B_{r n^{1/4}}\left( \BB e_{\op{S}} ;  Q_\infty \right)$.

\begin{lem} \label{prop-bdy-swallow}
For each $r > 0$ and each $\ep \in (0,1)$, there exists $C  = C(r,\ep) > 0$ such that for each $n\in\BB N$, 
\eqbn
\BB P\left[ K_r^n \leq C n^{1/2} \right] \geq 1-\ep .
\eqen
\end{lem}
\begin{proof} 
Let $(C_\infty , L_\infty ,  b_\infty)$ be the Schaeffer encoding process of $(Q_\infty,\BB e_\infty)$ as in Section~\ref{sec-uihpq}.  Since $ b_\infty $ is obtained from a two-sided simple random walk by skipping the upward steps, we can find $C_0 = C_0(r,\ep)$ such that with probability at least $1-\ep/4$,
\eqbn
\min_{k \in [0,C_0 n^{1/2}]_{\BB Z} } b_\infty(k) \leq - r n^{1/4} \quad \op{and} \quad \min_{k\in [ -C_0 n^{1/2} , 0]_{\BB Z} }  b_\infty(k) \leq -r n^{1/4} .
\eqen
By Lemma~\ref{prop-cactus}, if this is the case then (with $\BB p_\infty$ the contour function as in Section~\ref{sec-uihpq})
\eqb \label{eqn-bdy-swallow0}
\op{dist}\left( \BB p_\infty(0) , \BB p_\infty(i) ; Q_\infty \right) 
\geq L_\infty(\BB p_\infty(0)) + L_\infty(\BB p_\infty(i)) - \left(  L_\infty(p_\infty(i)) +   \min_{k \in [0,C_0 n^{1/2}]_{\BB Z} } b_\infty(k)    \right)
\geq r n^{1/4}
\eqe 
whenever $i\geq I_\infty(\lfloor C_0 n^{1/2} \rfloor)$. We have a similar bound when $i\leq   I_\infty(-\lfloor C_0 n^{1/2} \rfloor) $.

The number of vertices of $\bdy Q_\infty$ in $\BB p_\infty([I_\infty(-\lfloor C_0 n^{1/2} \rfloor) , I_\infty(\lfloor C_0 n^{1/2} \rfloor) ]_{\BB Z})$ is at most the sum of the quantities $|b_\infty(k) - b_\infty(k-1)|+1$ for $k\in [-C_0n^{1/2} , C_0n^{1/2}]_{\BB Z}$ (c.f.\ Remark~\ref{remark-schaeffer-bdy-uihpq}). By the law of large numbers and~\eqref{eqn-bdy-swallow0} we can find $C_1 = C_1(r,\ep) > C_0$ such that with probability at least $1-\ep/2$, there are at most $C_1 n^{1/2}$ vertices of $\bdy Q_\infty$ in $B_{r n^{1/4}}(\BB e_\infty ;  Q_\infty)$. 
Since the dangling quadrangulation $q_{v_0}$ is a.s.\ finite, applying the preceding bound with $r - A$ in place of $r$ for $A$ a deterministic, $n$-independent constant shows that we can find $C = C(r,\ep) > C_1$ such that with probability at least $1-\ep$, there are at most $C n^{1/2}$ vertices of $\bdy Q_\infty$ in $B_{r n^{1/4}}(\BB e_{\op{S}} ;  Q_\infty)$. 

Our coupling of $Q_\infty$ and $Q_{\op{S}}$ implies that in this case, there are at most $C n^{1/2}$ vertices of $\bdy Q_{\op{S}}$ in $ B_{rn^{1/4}}\left(\BB e_{\op{S}}  ; Q_{\op{S}} \right) $. 
The statement of the lemma follows. 
\end{proof}
  
We next prove a bound for the diameter of dangling quadrangulations.  

\begin{lem} \label{prop-dangling-diam-uihpq}
For each $r>0$ and each $\delta  > 0$, it holds with probability tending to $1$ as $n\rta\infty$ that  
\eqbn
\max_{v \in \mcl V\left(\bdy Q_{\op{S}} \cap B_{rn^{1/4}}( \BB e_{\op{S}} ; Q_{\op{S}} ) \right) } \op{diam}\left(q_v   \right) \leq \delta n^{1/4}
\eqen    
where the diameter is taken with respect to the internal graph metric on $q_v$. 
\end{lem}

We will deduce Lemma~\ref{prop-dangling-diam-uihpq} from Proposition~\ref{prop-map-coupling} and an analogous bound for finite-volume quadrangulations.
Fix $\ell  >0$ and a sequence of positive integers $\{l^n\}_{n\in\BB N}$ with $l^n\rta \ell$. 
For $n\in\BB N$, let $(Q^n , \BB e_0^n , \BB v_*^n)$ be sampled uniformly $\mcl Q^{ \bullet}(n,  l^n)$. Let $\op{Core}(Q^n)$ be the quadrangulation obtained by removing from $Q^n$ each vertex and each edge which can be disconnected from $\BB v_*^n$ by deleting a single vertex of $\bdy Q^n$. Let $\mcl C^n$ be the set of connected components of the set of vertices and edges removed in this manner plus the vertices which can be deleted to disconnect these vertices and edges from $\BB v_*^n$. Then $\mcl C^n$ is a set of quadrangulations with general boundary which ``dangle" from $\op{Core}(Q^n)$.

\begin{lem} \label{prop-dangling-diam}
For each $\delta > 0$, it holds with probability tending to $1$ as $n\rta\infty$ that
\eqbn
  \max_{q\in \mcl C^n} \op{diam}\left( q  \right) \leq \delta n^{1/4}  
\eqen
where here the diameter is taken with respect to the internal graph metric on $q$. 
\end{lem} 
\begin{proof}
This is essentially proven in~\cite[Section~5]{bet-mier-disk} but for the sake of clarity we explain how the precise statement of the lemma follows from existing results in the literature. 
For $n\in\BB N$ let $d^n$ be the graph distance on $Q^n$, rescaled by $(9/8)^{1/4} n^{-1/4}$. 
Also let $(H , d , z)$ be a Brownian disk with area $1$ and boundary length $\ell$ together with a marked interior point, sampled uniformly from its area measure. 
We know from~\cite[Theorem~1]{bet-mier-disk} (c.f.\ Theorem~\ref{thm-disk-ghpu}) that $(Q^n , d^n , \BB v_*^n) \rta (H , d , z)$ in the pointed Gromov-Hausdorff topology. 
By the Skorokhod representation theorem, we can couple so that this convergence occurs a.s. 
By the analog of~\cite[Lemma~A.1]{gpw-metric-measure} for the pointed Gromov-Hausdorff topology (which follows from Proposition~\ref{prop-ghpu-embed}), we can find a compact metric space $(W , D)$ and isometric embeddings $Q^n \rta W$ and $H \rta W$ such that the following is true. 
If we identify $Q^n$ and $H$ with their embeddings into $W$, then we a.s.\ have $Q^n \rta H$ in the $D$-Hausdorff distance and $D(\BB v_*^n , z) \rta 0$. 

Suppose now by way of contradiction that the statement of the lemma is false. Then we can find $\delta > 0$ such that with positive probability, there is an infinite sequence $\mcl N$ of positive integers such that for $n\in\mcl N$, there is a $q^n\in \mcl C^n$ with $d^n$-diameter $\geq \delta$. Let $v^n$ be the vertex of $\bdy Q^n$ with the property that removing $v^n$ from $\bdy Q^n$ disconnects $q^n$ from $\BB v_*^n$. By possibly replacing $\mcl N$ with a further subsequence, we can arrange that $v^n \rta x \in H$ as $\mcl N \ni n \rta\infty$. 

For $n\in\mcl N$ and $\zeta\in (0,\delta)$, let $A^n_\zeta := q^n \setminus B_\zeta(v^n ; d^n)$ and let $U^n_\zeta := Q^n \setminus (q^n \cup B_\zeta(v^n ; d^n))$. Then for each $n\in\BB N$ we have $d^n(A_\zeta^n , U_\zeta^n) \geq \zeta$ and $A_\zeta^n \cup U_\zeta^n \cup B_\zeta(v^n ; d^n) = \mcl V(Q^n)$. For each rational $\zeta>0$, we can find a subsequence $\mcl N_\zeta$ of $\mcl N$ along which $A^n_\zeta \rta A_\zeta\subset H$ and $U_\zeta^n \rta U_\zeta \subset H$ in the $D$-Hausdorff distance. The sets $A_\zeta$ and $U_\zeta$ lie at distance at least $\zeta$ from each other (so are disjoint) and $A_\zeta \cup U_\zeta \cup B_\zeta(x ; d) = H$. Since $\BB v_*^n$ lies at uniformly positive $d^n$-distance from $\bdy Q^n$ with probability tending to $1$ as $n\rta\infty$ and by our choice of the vertices $v^n$, it follows that it is a.s.\ the case that for small enough $\zeta>0$, each of $A_\zeta$ and $U_\zeta$ is non-empty. Hence for each sufficiently small $\zeta>0$, removing $B_\zeta(x ; d)$ disconnects $H$ into two non-empty components. This contradicts the fact that $H$ has the topology of a disk~\cite[Theorem~2]{bet-disk-tight}.
\end{proof}

\begin{proof}[Proof of Lemma~\ref{prop-dangling-diam-uihpq}] 
Let $\{v_k\}_{k\in \BB Z}$ and $K_r^n$ be as in the discussion just before Lemma~\ref{prop-bdy-swallow}.
Also fix $\ep\in (0,1)$ and let $C = C(r,\ep)$ be chosen as in that lemma.
By Theorem~\ref{thm-uihpq-ghpu}, we can find $\rho =\rho(r,\ep) > r$ such that with probability at least $1-\ep$, each $v_k$ for $k \in [-C n^{1/2} , Cn^{1/2}]_{\BB Z}$ is contained in $B_{\rho n^{1/4}}(\BB e_{\op{S}} , Q_{\op{S}} )$. Then with probability at least $1-2\ep$,
\eqb \label{eqn-dangling-diam-good}
  \mcl V\left(\bdy Q_{\op{S}} \cap B_{rn^{1/4}}(\BB e_{\op{S}} , Q_{\op{S}} ) \right) \subset \{v_k \,:\, k \in [-C n^{1/2} , Cn^{1/2}]_{\BB Z} \} \subset \mcl V\left( \bdy Q_\infty\cap B_{\rho n^{1/4}}(\BB e_\infty , Q_\infty) \right) .
\eqe 
By Proposition~\ref{prop-map-coupling}, we can find $R = R(r,\ep ,\delta) \in\BB N$ with $R > \rho + \delta$ such that for large enough $n\in\BB N$, we can couple $(Q^{Rn} , \BB e_0^{Rn} , \BB v_*^{Rn})$ with $(Q_\infty , \BB e_\infty)$ in such a way that with probability at least $1-\ep$, the graph metric balls $B_{(\rho+\delta) n^{1/4} }( \BB e_\infty  ; \BB Q_\infty   )$ and $B_{ \rho + \delta}( \BB e_0^{R n}  ; R^{-1/4} d^{Rn} )$ equipped with the restricted graph metrics are isometric via a graph isomorphism which preserves the intersection of these metric balls with $\bdy Q_\infty$ and $\bdy Q^{Rn}$, respectively. If this is the case and~\eqref{eqn-dangling-diam-good} holds, then each $q_v$ for $v\in \mcl V\left(\bdy Q_{\op{S}} \cap B_{r n^{1/4}}(   \BB e_{\op{S}} , Q_{\op{S}}   ) \right)$ with internal diameter $\geq \delta n^{1/4}$ corresponds to a unique dangling quadrangulation in $\mcl C^{Rn}$ with internal diameter $\geq \delta n^{1/4}$. By Lemma~\ref{prop-dangling-diam}, the probability that such a quadrangulation exists tends to $0$ and $n\rta\infty$. 
Since $\ep \in (0,1)$ was arbitrary, we conclude.
\end{proof}

Next we turn our attention to a bound for the areas of the dangling quadrangulations. 

\begin{lem} \label{prop-dangling-area}
For each $r>0$ and each $\ep > 0$, there exists $A = A(r,\ep) >0$ such that for each $n\in\BB N$, it holds with probability at least $1-\ep$ that
\eqbn
\sum_{v\in \mcl V(\bdy Q_{\op{S}} ) \cap B_{r n^{1/4}}( \BB e_{\op{S}} ; Q_{\op{S}}      )} \mu_\infty^n(q_v ) \leq A n^{-1/3  } , 
\eqen
where here we recall that $\mu_\infty^n$ is the measure which assigns each vertex of $Q_\infty^n$ a mass equal to $(4n)^{-1}$ times its degree.
\end{lem} 
\begin{proof}
Let $\{v_k\}_{k\in \BB Z}$ and $K_r^n$ be as in the discussion just before Lemma~\ref{prop-bdy-swallow}.
By Lemma~\ref{prop-bdy-swallow}, for each $\ep > 0$ there exists $C = C(r , \ep) > 0$ such that with probability at least $1-\ep$, we have $K_r^n \leq C n^{1/2}$, in which case
\eqbn
\sum_{v\in \mcl V(\bdy Q_{\op{S}} ) \cap B_{rn^{1/4}}( \BB e_{\op{S}} ; Q_{\op{S}})} \mu_\infty^n(q_v ) \preceq n^{-1}\sum_{k=-\lfloor C n^{1/2} \rfloor}^{\lfloor C n^{1/2} \rfloor} \# \mcl E(q_{v_k } ) 
\eqen 
with universal implicit constant. 

By~\cite[Equation (24)]{caraceni-curien-uihpq}, we have the tail estimate $\BB P\left[ \# \mcl E(q_v) > m \right] \sim c m^{-3/4}$ for $v\not=v_0$, for a universal constant $c>0$. 
By the heavy-tailed central limit theorem, the random variables
\eqbn
n^{-2/3} \sum_{k=-\lfloor C n^{1/2} \rfloor}^{\lfloor C n^{1/2} \rfloor} \# \mcl E(q_{v_k } ) 
\eqen
converge in law to a non-degenerate limiting distribution. 
The statement of the lemma follows.
\end{proof}

It remains to prove a bound for the boundary length of the dangling quadrangulations (which will explain why we use a different scaling in the definitions of the re-scaled boundary paths $\eta_\infty^n$ and $\eta_{\op{S}}^n$). 
For $t \geq 0$, let $\sigma_{\op{S} }^n(t)$ be equal to $  n^{-1/2}$ times the sum of the boundary lengths of the dangling quadrangulations $q_v $ attached at vertices $v$ which are hit by $\eta_{\op{S}}^n $ between time $0$ and time $t$ plus $  n^{-1/2}$ times the total number of such quadrangulations, so that for $t\geq 0$, 
\eqb \label{eqn-uihpqS-time-change}
\eta_{\op{S}}^n \left( t \right) 
=  \eta_{\infty}^n \left( 2^{-3/2} \sigma_{\op{S}}^n(t) +O_n(n^{-1/2}) \right) .
\eqe 
Also let $\wt \sigma_{\op{S} }^n(-t)$ be equal to $ n^{-1/2}$ times the sum of the boundary lengths of the dangling quadrangulations $q_v $ attached at vertices $v$ which are hit by $\eta_{\op{S}}^n $ between time $-t$ and time $0$ plus $  n^{-1/2}$ times the total number of such quadrangulations so $\wt\sigma_{\op{S}}$ satisfies an analog of~\eqref{eqn-uihpqS-time-change} for negative times.  
  
\begin{lem} \label{prop-dangling-length}
We have $\sigma_{\op{S}}^n \rta (t\mapsto 2^{3/2} t) $ and $\wt\sigma_{\op{S}}^n \rta (t\mapsto 2^{3/2} t) $ in law with respect to the topology of uniform convergence on compact subsets of $[0,\infty)$.  
\end{lem}
\begin{proof}
We will prove the statement for $\sigma_{\op{S}}^n$; the statement for $\wt\sigma_{\op{S}}^n$ is proven identically.  Let $\{v_k\}_{k\in \BB Z}$ be as in the discussion just before Lemma~\ref{prop-bdy-swallow}. Then $q_{v_0 } $ is a.s.\ finite and by~\cite[Equation (23)]{caraceni-curien-uihpq}, for $k \in \BB N $ we have $\BB E\left[ \op{Perim}(q_{v_k } ) + 1 \right] = 3$, where here $\op{Perim}$ denotes the perimeter. The random variables $\op{Perim}(q_{v_k } )+1$ for $k\in\BB N$ are i.i.d.\  and $\op{Perim}(q_{v_0})$ is a.s.\ finite, so by the strong law of large numbers there a.s.\ exists a random $M \in\BB N$ such that for $m \geq M$, 
\eqb \label{eqn-lln-dangling-sum}
\left| \frac{1}{m} \sum_{k=0}^m (\op{Perim}(q_{v_k } )+1) - 3 \right| \leq \ep.
\eqe  
The law of $M$ does not depend on $n$, so we can find a deterministic $m_0\in\BB N$ such that with probability at least $1-\ep$,~\eqref{eqn-lln-dangling-sum} holds for each $m\geq m_0$. Let $E$ be the event that this is the case. Since $\sum_{k=0}^{m_0} (\op{Perim}(q_{v_k } )+1) < \infty$ a.s., we can find a deterministic $C>0$, independent from $n$, such that with probability at least $1-\ep$ this sum is at most $C$. Let $F$ be the event that this is the case. 

We have
\eqbn
\sigma_{\op{S}} (t) = n^{-1/2} \sum_{k=0}^{ \lfloor \frac{ 2^{3/2} }{3} n^{1/2} t  \rfloor}  (\op{Perim}(q_{v_k } )+1) 
\eqen
so if $E$ occurs and $\lfloor \frac23 n^{1/2} t  \rfloor \geq m_0$,  
\eqbn
\left| \sigma_{\op{S}} (t) - 2^{3/2} t \right| \leq \ep t  .
\eqen
Hence if $T>0$, then on $E\cap F$, 
\eqbn
\sup_{t\in [0,T]} \left| \sigma_{\op{S}} (t) - 2^{3/2} t \right| \leq \ep T + O_n( n^{-1/2} )  .
\eqen
Since $\ep > 0$ is arbitrary, the statement of the lemma follows.
\end{proof}

\begin{proof}[Proof of Theorem~\ref{thm-uihpqS-ghpu}] 
Let $\left\{\left( (Q_{\op{S}}^n , \BB e_{\op{S}}^n) , (Q_\infty^n , \BB e_\infty^n ) \right) \right\}_{n\in\BB N}$ be a sequence of copies of the coupling $\left( (Q_{\op{S}}  , \BB e_{\op{S}}   ) , (Q_\infty  , \BB e_\infty  ) \right)$ used in this section and let $\{q_v^n\}_{n\in\BB N}$ be the associated dangling quadrangulations. 
Define the elements $\frk Q_\infty^n = (Q^n , d_\infty^n , \mu_\infty^n , \eta_\infty^n) $ and $ \frk Q_{\op{S}}^n = (Q_{\op{S}}^n , d_{\op{S}}^n, \mu_{\op{S}}^n , \eta_{\op{S}}^n)$ of $\BB M_\infty^{\op{GHPU}}$ as in Section~\ref{sec-results} with respect to the $n$th pair in this sequence.

By Theorem~\ref{thm-uihpq-ghpu} and the Skorokhod representation theorem, we can find a coupling of the sequence $\left\{\left( (Q_{\op{S}}^n , \BB e_{\op{S}}^n) , (Q_\infty^n , \BB e_\infty^n ) \right) \right\}_{n\in\BB N}$ and $(H_\infty , d_\infty)$ such that a.s.\ $\frk Q_\infty^n \rta \frk H_\infty$ in the local GHPU topology.  By Proposition~\ref{prop-ghpu-embed-local}, we can a.s.\ find a random boundedly compact metric space $(W , D)$ and isometric embeddings of $(Q_\infty^n , d_\infty^n)$ for $n\in\BB N$ and $(H_\infty , d_\infty)$ into $(W,D)$ such that if we identify these spaces with their embeddings, then a.s.\ $\frk Q_\infty^n \rta \frk H_\infty$ in the $D$-local HPU topology (Definition~\ref{def-hpu-local}). 
  
Now fix a deterministic $r>0$ and recall~\eqref{eqn-uihpq-balls}.
By Lemma~\ref{prop-dangling-diam-uihpq}, in any such coupling 
\eqbn
B_r (\eta_{\op{S}}^n(0) ; d_{\op{S}}^n) \rta  B_r(\eta_\infty(0) ; d_\infty) 
\eqen
in probability with respect to the $D$-Hausdorff metric.
By Lemma~\ref{prop-dangling-area}, in any such coupling
\eqbn
\mu_{\op{S}}^n|_{B_r(\eta_{\op{S}}^n(0) ; d_{\op{S}}^n)} \rta \mu_{\op{S}}|_{B_r(\eta_\infty(0) ; d_\infty)} 
\eqen
in probability with respect to the $D$-Prokhorov metric.
By~\eqref{eqn-uihpqS-time-change} and Lemmas~\ref{prop-dangling-diam-uihpq} and~\ref{prop-dangling-length}, in any such coupling
\eqbn
 \frk B_r\eta_{\op{S}}^n   \rta \frk B_r\eta_\infty
\eqen
in probability with respect to the $D$-uniform metric, where $\frk B_r$ is as in Definition~\ref{def-ghpu-truncate}. Therefore $\frk B_r\frk Q_{\op{S}}^n \rta \frk B_r\frk H_\infty$ in law in the GHPU topology, so by Lemma~\ref{prop-local-ghpu-subsequence} $\frk Q_{\op{S}}^n \rta \frk H_\infty$ in law in the local GHPU topology. 
\end{proof}

\bibliography{cibiblong,cibib}
\bibliographystyle{hmralphaabbrv}

\end{document}